\documentclass[a4paper]{article}

\usepackage[english]{babel}
\usepackage{lmodern}
\usepackage[T1]{fontenc}
\usepackage{hyperref}
\usepackage{amsfonts,amsmath,amssymb,amsopn,amsthm}
\usepackage[all]{xy}
\usepackage{vmargin}
\usepackage{makeidx}
\usepackage{graphicx}
\usepackage{tikz}
\usepackage{mathabx}
\usepackage{mathbbol} 

\title{Absolute calculus and prismatic crystals on cyclotomic rings\footnotetext{\normalsize{\textit{Mathematics Subject Classification (MSC2020):} 14F30; 14F40}}}
\author{Michel Gros, Bernard Le Stum \& Adolfo Quir\'os}
\date{}
\usepackage[ style=alphabetic,citestyle=alphabetic,sorting=nyt,sortcites=true,autopunct=true,babel=hyphen,hyperref=true,abbreviate=false,backref=true]{biblatex}
\AtEveryBibitem{ 
 \clearfield{urldate} \clearfield{review} \clearfield{series} \clearfield{doi} \clearfield{isbn} \clearfield{issn} \clearfield{note} }
\defbibheading{bibempty}{}
\DeclareNolabel{\nolabel{\regexp{[\p{Z}]+}}}
\newtheorem{thm}{Theorem}[section]
\newtheorem{prop}[thm] {Proposition}
\newtheorem{cor}[thm] {Corollary}
\newtheorem{lem}[thm] {Lemma}
\theoremstyle{definition}
\newtheorem{dfn}[thm] {Definition}
\newtheorem{rmk}[thm] {Remark}
\newtheorem{rmks}[thm] {Remarks}

\newenvironment{xmps}[1][Examples.]{\begin{trivlist} \item[\hskip \labelsep {\bfseries #1}]}{\end{trivlist}}
\usepackage[pagewise,mathlines,displaymath]{lineno}

\newcommand{\Addresses}{{
 \bigskip
 \footnotesize

Michel Gros, \textsc{IRMAR, Universit\'e de Rennes,
Campus de Beaulieu, 35042 Rennes cedex, France}\par\nopagebreak
\texttt{michel.gros@univ-rennes.fr}

\medskip

Bernard Le Stum, \textsc{IRMAR, Universit\'e de Rennes,
Campus de Beaulieu, 35042 Rennes cedex, France}\par\nopagebreak
\texttt{bernard.le-stum@univ-rennes.fr}

\medskip

Adolfo Quir\'os, \textsc{Departamento de Mat\'ematicas, Facultad de Ciencias, Universidad Aut\'onoma de Madrid, E-28049 Madrid, Spain}\par\nopagebreak
\texttt{adolfo.quiros@uam.es}

}}

\begin{document}

\maketitle

\begin{abstract}

Let $p$ be a prime, $W$ the ring of Witt vectors of a perfect field $k$ of characteristic $p$ and $\zeta$ a primitive $p$th root of unity. We introduce a new notion of calculus over $W$ that we call absolute calculus. It may be seen as a singular version of the $q$-calculus used in previous work, in the sense that the role of the coordinate is now played by $q$ itself. We show that what we call a weakly nilpotent $\mathbb\Delta$-connection on a finite free module is equivalent to a prismatic vector bundle on $W[\zeta]$. As a corollary of a theorem of Bhatt and Scholze, we finally obtain that a $\mathbb\Delta$-connection with a frobenius structure on a finite free module is equivalent to a lattice in a crystalline representation. We also consider the case of de Rham prismatic crystals as well as Hodge-Tate prismatic crystals.

\end{abstract}

\tableofcontents

\section*{Introduction}
\addcontentsline{toc}{section}{Introduction}

Let $G_K := \mathrm{Gal}(\overline K/K)$ be the absolute Galois group of a complete discrete valuation field $K$ of characteristic $0$ with perfect residue field $k$ of characteristic $p>0$.
The study of $p$-adic representations of $G_K$, initiated by Jean-Marc Fontaine (see for example \cite{Fontaine94}), has recently known a major breakthrough with a theorem of Bhargav Bhatt and Peter Scholze.
Using the notion of prismatic site they have introduced in order to get a conceptual understanding of integral $p$-adic Hodge theory, they have proved (\cite{BhattScholze21}, theorem 5.6) that there is an equivalence
\[
\mathrm{Vect}^\mathrm{\phi}(\mathcal O_K, \mathcal O_{\mathbb\Delta}) \simeq \mathrm{Rep}^{\mathrm{crys}}_{\mathbb Z_p}(G_K)
\]
between the category of prismatic $F$-crystals on $\mathcal O_K$ and the category of lattices in crystalline representations of $G_K$.
Forgetting the frobenius structures, understanding the category $\mathrm{Vect}(\mathcal O_K, \mathcal O_{\mathbb\Delta})$ of (absolute) prismatic vector bundles on $\mathcal O_K$ becomes therefore a fundamental question.
Several partial results, in various settings, have already been obtained for Hodge-Tate prismatic crystals and de Rham prismatic crystals, for example in \cite{BhattLurie22}, \cite{BhattLurie22b}, \cite{GaoMinWang22}, \cite{AnschuetzHeuerLeBras22}, \cite{GaoMinWang23}.
Let us also mention the  very recent work of Zeyu Liu (\cite{Liu25}) in which he describes absolute prismatic crystals (in various settings) using the stacky approach of \cite{BhattLurie22}, \cite{BhattLurie22b}.

Our goal is to address this issue from a different point of view, using what we call \emph{absolute calculus}. This is similar to the (relative) $q$-calculus we have used before \cite{GrosLeStumQuiros23a, GrosLeStumQuiros22b}, but now the parameter $q$ coincides with the \'etale coordinate $x$. Informally, if $x$ is a variable, then the \emph{relative} $q$-derivation is defined by $\partial_q(x^k) = (k)_qx^{k-1}$ where $(k)_q := 1 + q + \ldots + q^{k-1}$ denotes the $q$-analog of $k$ (written $[k]_q$ in \cite{BhattScholze22}).
The \emph{absolute} $q$-derivation then corresponds to the case $x = q$: the condition reads $\partial_q(q^k) = (k)_qq^{k-1}$. It will actually be necessary to consider a variant $\partial_{\mathbb\Delta}$ that we call \emph{absolute $\mathbb\Delta$-derivation} which is given by
\begin{equation} \label{defqpder}
\partial_{\mathbb\Delta}(q^k) = (pk)_qq^{k-1}.
\end{equation}

From now on, we will assume that $\mathcal O_K = W[\zeta]$, where $W = W(k)$ is the ring of Witt vectors of $k$ and $\zeta$ is a primitive $p$th root of unity.
We define a \emph{$\mathbb\Delta$-derivation} on a $W[[q-1]]$-module $M$ to be a $W$-linear map $\partial_{M,\mathbb\Delta}:M \longrightarrow M$ that satisfies
\[
\forall s \in M, \quad \partial_{M,\mathbb\Delta}(qs) = (p)_q s+q^{p+1} \partial_{M,\mathbb\Delta}(s).
\]
Giving the datum of a $\mathbb\Delta$-derivation is equivalent to giving the datum of a \emph{$\mathbb\Delta$-connection} (see definition \ref{qconp}).
A fundamental example of $\mathbb\Delta$-connection/derivation is provided by $\partial_{\mathbb\Delta}$ on $M=W[[q-1]]$ defined by formula \eqref{defqpder}.
Note that, if we denote by $\gamma$ the endomorphism given by $\gamma(q) = q^{p+1}$, then \[ q\partial_{M,\mathbb\Delta} = \frac {\gamma - \mathrm{Id}}{q-1} \] is the usual operator appearing in the theory of Wach modules.

We use the terms \emph{absolute} and \emph{prismatic} as well as the subscript $\mathbb\Delta$ to emphasize that the objects we consider are closely related to the absolute prismatic site of ${\mathcal{O}}_K$.
More precisely, our main result is the following (when $p$ is odd (resp. when $p=2$), we call a $\mathbb\Delta$-connection \emph{weakly nilpotent} if $\partial_{M,\mathbb\Delta}$ (resp. $\partial^2_{M,\mathbb\Delta}- \partial_{M,\mathbb\Delta}$) is nilpotent modulo $(p,q-1)$):

\textbf{Theorem} (see theorem \ref{prisdR}). 
\textit{
There exists an equivalence
\[
\mathrm{Vect}(\mathcal O_K, \mathcal O_{\mathbb\Delta}) \simeq \nabla^{\mathrm{wn}}_{\mathbb\Delta}(W[[q-1]]), \quad E \leftrightarrow (M, \partial_{M,\mathbb\Delta})
\]
between the category of prismatic vector bundles on $\mathcal O_K$ and the category of finite free modules on $W[[q-1]]$ endowed with a weakly nilpotent $\mathbb\Delta$-connection.
Moreover, this equivalence induces an isomorphism
\[
\mathrm R\Gamma(\mathcal O_{K\mathbb\Delta}, E) \simeq \mathrm R\Gamma_{\mathrm{dR},\mathbb\Delta}(M) := [M \overset {\partial_{M,\mathbb\Delta}} \longrightarrow M]
\]
on cohomology.
}

As a basic example, the structural sheaf of the prismatic topos corresponds under this equivalence to $W[[q-1]]$ endowed with $\partial_{\mathbb\Delta}$ (note for the nilpotency requirement that $(p)_q \in (p, q-1)$). Several properties witness to the relevance of $\partial_{\mathbb\Delta}$: Griffiths transversality with respect to the Nygaard filtration (proposition \ref{Griff}.1 and remark \ref{Nyg}.2), preservation of the $(p)_q$-adic filtration (example 3 after proposition \ref{acevd}) and of the $(p^\bullet)_q$-filtration (remark \ref{thetarmks}.3), compatibility with the frobenius morphism $\phi(q)=q^p$, meaning $\partial_{\mathbb\Delta}\circ \phi = (p)_q q^{p-1}\phi \circ \partial_{\mathbb\Delta}$ (equation \eqref{frobtriv}), a formula whose shape attests to the logarithmic nature of $\partial_{\mathbb\Delta}$ and anticipates the appearance of ``monodromy'' operators associated to it, etc. We also discuss in the same explicit way other examples, such as the $\mathbb\Delta$-connection corresponding to the Breuil-Kisin twist on prismatic crystals.

Using the same strategy as in the proof of theorem \ref{prisdR} (that we recalled above), we obtain analogous descriptions for the category of prismatic de Rham crystals (in the style of \cite{BhattScholze21}, construction 7.16) in propositions \ref{Bdreq_qp}, as well as for the category of prismatic Hodge-Tate crystals (see \cite{BhattLurie22b}) in proposition and \ref{HTeq}. This makes completely transparent the relations between the various theories and diagram \eqref{bigdiag} summarizes our equivalences, all of them compatible with the natural cohomology functors on each side.

Frobenius structures can also be added to the picture and then  the category of prismatic F-crystals on $\mathcal O_K$ turns out to be  equivalent to the category of Wach modules (definition \ref{Wach} and proposition \ref{Wath}).
From this perspective, the aforementioned theorem of Bhatt and Schole (\cite{BhattScholze21}, theorem 5.6)  is equivalent to Berger's theorem (\cite{Berger04}, proposition III.4.2).
Using a slightly different approach than ours, Abhinandan (\cite{Abhinandan24}) has obtained the same result in the absolutely unramified case.

Our approach differs from others justifying, we believe, the introduction of our formalism: 
we adopt the classical path relating crystals and modules with connection, taking as starting point just our explicit description of prismatic envelopes inherited from \cite{GrosLeStumQuiros23a}, theorem 5.2, and developing methodically the corresponding twisted calculus.
This provides the naturalness (compare with \cite{GaoMinWang22}, remark 1.18) of all constructions mentioned above, in particular when ``specializing'' (sections 8 to 10).

Our approach also produces naturally a formal groupoid, denoted by $G$ in the introduction of section \ref{stratder}, that lifts (see remark \ref{HTrmks}.4) the formal group of \cite{BhattLurie22b}, proposition 9.5, and should play a role in the description of the prismatization of $\mathcal{O}_K$ by generalizing that of its Hodge-Tate locus in loc.~cit. We leave this investigation for future work. We also hope that the new insight presented here for $O_K=W[\zeta]$ will, beyond its own intrinsic interest and its value as an example, be useful for a better understanding of other situations. For example, some results concerning the case ${\mathcal{O}}_K={\mathbb{Z}}_p$ studied in detail in \cite{BhattLurie22}, should become accessible in a new way just by using Galois descent as suggested in loc.~cit., remark 4.8.4.

We give now a quick description of the content of the present article, that can be divided into two main parts: sections 1-6, devoted to the elaboration of absolute calculus, and sections 7-10, where our methods are applied to the description of various kinds of prismatic crystals.

Section \ref{twispow} is devoted to twisted powers on an arbitrary ring $R$. We recall various aspects of this notion from \cite{GrosLeStumQuiros22a} using the language of twisted Stirling numbers (definition \ref{defStirling}). Our main player will be a ring $R\langle \omega \rangle_{\mathbb\Delta}$ of \emph{$\mathbb\Delta$-divided powers} (definition \ref{defdtw}) that, when $p$ is invertible (proposition \ref{poweg2}) or when $(p)_q=0$ (proposition \ref{poweg}), has simpler incarnations in, respectively, a polynomial algebra or a (usual) divided polynomial algebra.

Section \ref{secabs} starts the development of absolute calculus on $R=W[[q-1]]$, where $W$ is an adic ring. We introduce and study the notions of \emph{$\mathbb\Delta$-derivation}, an example being $\partial_{\mathbb\Delta}$ above, and, more generally, of \emph{$\nabla_{\mathbb\Delta}$-module} (definition \ref{nabR} and proposition \ref{acevd}) and give several examples ($(p)_q$-adic filtration, Breuil-Kisin twist\dots) that will be looked at in the article.

Section \ref{higherd} is about Taylor maps and higher derivatives.
We develop a formalism of \emph{$\mathbb\Delta$-Taylor expansions $\theta^{(\infty)}_{n,\mathbb\Delta}(\alpha)$ of infinite level and finite order} $n \in {\mathbb{N}}$ for $\alpha \in R$ (definition \ref{Taylor}).
As in the untwisted case, there also exists a variant $\theta_{n,\mathbb\Delta}(\alpha)$ (definition \ref{finTayl}).
The coefficients of the images of $\alpha \in R$ by these different maps, expressed in natural $R$-bases, are respectively the \emph{higher $\mathbb\Delta$-derivatives} $\partial_{\mathbb\Delta}^{[k]}(\alpha)$ and $\partial_{\mathbb\Delta}^{\langle k \rangle}(\alpha)$.
Explicit formulas for their action on $R$ can be given (lemma \ref{tetaqn} and proposition \ref{tetaqn2}): for all $k,n \in {\mathbb{N}}$ 
\[
\partial_{\mathbb\Delta}^{[k]}(q^n) = (p)_q^k{n \choose k}_{q^p}q^{n-k},
\]
and the relation between infinite and finite level follows exactly the same pattern as in the ``classical'' case: $\partial_{\mathbb\Delta}^{\langle k \rangle} = (k)_{q^p}!\partial_{\mathbb\Delta}^{[k]}$. Note that $\partial_{\mathbb\Delta}^{\langle k \rangle}$ is far from being the $k$th iterate of $\partial_{\mathbb\Delta}$, explaining our choice of brackets for the notation.

Section \ref{padsit}, and everything thereafter, works over $R=W[[q-1]]$, where $p$ is a prime and $W$ is a $p$-adically complete ring (up to this point, $p$ could be any positive integer and $W$ any adic ring). We identify (theorem \ref{prismenv}) our now completed twisted ring $\widehat{R\langle \omega \rangle}_{\mathbb\Delta}$ with the prismatic envelope of the diagonal in order to rely on the universal property of the latter. This allows us to define a {\emph{flip map}} (proposition \ref{extflip}) and a {\emph{comultiplication map}} (proposition \ref{extcomult}) on $\widehat{R\langle \omega \rangle}_{\mathbb\Delta}$. 

Section \ref{stratder} uses the Taylor, flip and comultiplication maps to develop \emph{absolute hyper calculus} and to study the relation between what we shall call $\mathbb\Delta$-hyperstratifications and $\mathbb\Delta$-connections. We define (definitions \ref{hyperst} and \ref{TaylMod}) \emph{$\mathbb\Delta$-hyperstratifications} and \emph{$\mathbb\Delta$-Taylor maps} on $R$-modules and show (proposition \ref{eqHyperTayl}) that they are equivalent notions. Moreover, we prove (lemma \ref{locTayl}) that the datum of a $\mathbb\Delta$-hyperstratification on an $R$-module $M$ is equivalent to the datum of continuous endomorphisms $\partial_{M,\mathbb\Delta}^{\langle k \rangle}$ ($k \in {\mathbb{N}}$) of $M$ satisfying some natural compatibilities. After introducing (example 2 after proposition \ref{inher}) the linear derivation $L_{\mathbb\Delta} : \widehat{R\langle \omega \rangle}_{\mathbb\Delta} \rightarrow \widehat{R\langle \omega \rangle}_{\mathbb\Delta}$, we prove the {\emph{little Poincar\'e lemma}} (proposition \ref{little}): the sequence
\[
0 \rightarrow R \rightarrow \widehat{R\langle \omega \rangle}_{\mathbb\Delta} \stackrel{L_{\mathbb\Delta}}\longrightarrow \widehat{R\langle \omega \rangle}_{\mathbb\Delta} \rightarrow 0
\]
is split exact. 

Section \ref{diffop} is devoted to what we call $\mathbb\Delta$-differential operators.
We define (definition \ref{defdifo}) \emph{$\mathbb\Delta$-differential operators} between $R$-modules,
as well as their $R$-linearization and their composition, and we use them to prove (theorem \ref{invlem}) that, for a finite projective $R$-module, it is equivalent to endow it with a $\mathbb\Delta$-hyperstratification (a natural way to interpret crystals) or with a $\mathbb\Delta$-connection that is nilpotent modulo $(p,q-1)$ (what we call \emph{weakly nilpotent}).
Actually, when $p=2$, this nilpotence condition has to be refined a bit
(fortunately, our formalism finds again the known differences \cite{BhattLurie22b}, \S9, between the ramified and unramified cases).

We are now ready to apply our formalism to the study of various prismatic crystals.

Section \ref{prissec} establishes the link with prismatic crystals.
We assume from now on that $W$ is the ring of Witt vectors of a perfect field and let $\zeta$ be a primitive $p$th root of unity.
We start by showing (theorem \ref{eqTayl}) that the category of prismatic crystals on $W[\zeta]$ is equivalent to the category of complete $W[[q-1]]$-modules endowed with a $\mathbb\Delta$-hyperstratification. Using this, we prove our main theorem \ref{prisdR} described above.
The strategy is similar to that in \cite{GrosLeStumQuiros23a} or \cite{GrosLeStumQuiros22b} and relies on theorem \ref{invlem}. The isomorphism in cohomology follows from the prismatic Poincar\'e lemma \ref{poinc} as in \cite{GrosLeStumQuiros23b}, theorem 6.2.
We end this section by explaining in proposition \ref{nabGam} the relation to the theory of $\Gamma$-modules.

Section \ref{generic} studies what happens when one makes $p$ invertible and gives descriptions in the style of theorem \ref{prisdR} for \emph{de Rham prismatic crystals}.
In this case there is no real need for twisted derivations, since usual logarithmic derivatives do the job. In a way analogous to what we have done before, we develop a logarithmic calculus where the relevant operators are log-derivatives $\partial_{\log}^{\langle k \rangle} = (p)_q^k\partial^k$, with $\partial^k$ the $k$th iterate of the usual derivative with respect to $q$.
These log-derivatives give rise to log-connections, and we prove (proposition \ref{Bdreq}) that the category of de Rham prismatic crystals is equivalent to the category of finite free modules with a log-connection satisfying an appropriate nilpotency condition. Using, with $p$ invertible, the same approach as in section \ref{prissec}, we can also describe de Rham prismatic crystals in terms of $\mathbb\Delta$-connections (proposition \ref{Bdreq_qp}). Moreover, we give explicit formulas (proposition \ref{invfor}) relating $\partial_{\log}^{\langle k \rangle}$'s and $\partial_{\mathbb\Delta}^{\langle k \rangle}$'s in this situation. A similar problem is studied, in greater generality but maybe in a less explicit way, in \cite{BhattScholze21}, 7.15-7.17, in order to establish the relation with the classical Breuil-Kisin theory.

Section \ref{reduced} works modulo $(p)_q$ and describes \emph{Hodge-Tate prismatic crystals}.
As before, one has two points of view for their descriptions.
It is again not necessary to use twisted derivations, since we can obtain (proposition \ref{Sen}) a description of Hodge-Tate prismatic crystals in terms of monodromy/Sen operators, with a nilpotency condition, by redoing the theory of log-calculus on a PD-algebra. We recover in our situation a similar result to that of \cite{BhattLurie22b}.
As in section \ref{generic}, we get another description (proposition \ref{HTeq}) through a reduction modulo $(p)_q$ of the general theorem \ref{prisdR}, and we can again go from one point of view to the other through explicit formulas (proposition \ref{invforSen}). We finish with a review (diagram \eqref{bigdiag}) of the various equivalences.

Section \ref{frob} discusses frobenius structures on prismatic crystals. We define the notion of a frobenius structure on an $R$-module endowed with a $\mathbb\Delta$-connection and, as mentioned above, we obtain (corollary \ref{corBS}), using theorem 5.6 of \cite{BhattScholze21}, an equivalence between the category of lattices in crystalline Galois representations of the absolute Galois group of ${\rm{Frac}}(W[\zeta])$ and the category of finite free modules on $W[[q-1]]$ endowed with a $\mathbb\Delta$-connection that has a frobenius structure (weak nilpotency is automatic here, see lemma \ref{FrobWN}).
We also show that this last category is equivalent to the category of Wach modules relative to $W[\zeta]$ and essentially recover (corollary \ref{Berger}) a theorem of Laurent Berger (\cite{Berger04} building on work of Jean-Marc Fontaine and Nathalie Wach, completed by Pierre Colmez).
We finally briefly explain how our methods can be applied to the theory of $(\varphi,\Gamma)$-modules.

\subsection*{Acknowledgments}

We thank the referee for helping improve the presentation.

We are grateful to Hui Gao for very relevant remarks on the first version of this article.

M. G. thanks Charles Simonyi endowment and IAS (Princeton) for their generosity and hospitality during the final step of elaboration of this article. 

M. G. and B. L. S. benefit from the support of the French government program ANR-11-LABX-0020-01.

A. Q. was supported by grant PID2022-138916NB-I00, funded by MCIN/AEI/10.13039/501100011033 and by ERDF A way of making Europe, and by the Madrid Government (Comunidad de Madrid, Spain) under an agreement with UAM in the context of the V Regional Programme of Research and Technological Innovation.

\subsection*{Notations}

In contrast with the terminology we have used in previous work, we will write ``$\mathbb\Delta$-something'' instead of ``$q$-something of level $-1$ with respect to $p$'' and use $\mathbb\Delta$ instead of $q(-1)$ in the notation.


\section{Twisted powers} \label{twispow}

We review some results about the notion of a twisted power introduced in \cite{GrosLeStumQuiros22a}  (see also section 1 of \cite{GrosLeStumQuiros22b}) that will be systematically used starting at section \ref{higherd}.
The reader is advised to jump directly to section \ref{secabs} and come back to this one when twisted powers enter the game.

We fix some $p \in \mathbb N$ that will usually be a prime, but the case of a power of a prime for example and even the case $p=1$, may also be of interest.

We fix a commutative ring $R$ and some $q \in R$.
At some point, there will also be an extra $x \in R$ but, starting at section \ref{secabs}, we shall focus on the case $x=q$.

Recall from definition 2.5 of \cite{LeStumQuiros15} for example, that the \emph{twisted binomial coefficients} can be defined recursively for $n,k \in \mathbb N$ by the \emph{twisted Pascal identities} 
\[
{n+1 \choose k}_{q} = {n \choose k-1}_{q} + q^k {n \choose k}_{q}
\]
with the initial conditions
\[
{0 \choose 0}_{q} = 1 \quad \mathrm{and} \quad {n \choose 0}_{q} = {0 \choose k}_{q} = 0 \ \mathrm{for}\ k,n > 0.
\]
We shall also consider, for $n \in \mathbb N$, the \emph{twisted integer} (called $q$-analog in the introduction)
\[
(n)_{q} := {n \choose 1}_{q} = 1 + q+q^2 + \ldots + q^{n-1}
\]
and the \emph{twisted factorial}
\[
(n)_{q}! := (n)_{q} (n-1)_{q} \cdots (2)_{q}(1)_{q}.
\]
We will use the prefix ``$q$-'' instead of the attribute \emph{twisted} when we want to specify the parameter.
We will also drop the epithet/prefix/subscript in the case $q=1$.

We now recall the following related notions (formulas 5.19 and 5.20 in \cite{Ernst12}):

\begin{dfn} \label{defStirling}
The \emph{twisted Stirling numbers of the first kind} (resp.\ \emph{of the second kind}) are defined recursively for $n,k \in \mathbb N$ by
\begin{linenomath}
\begin{align*}
&s_{q}(n+1,k) = s_{q}(n,k-1) - (n)_{q} \ s_{q}(n,k)
\\
\textrm{(resp.} \quad &S_{q}(n+1,k) = S_{q}(n,k-1) + (k)_{q} \ S_{q}(n,k))
\end{align*}
\end{linenomath}
with the initial conditions
\begin{linenomath}
\begin{align*}
&s_{q}(0,0) = 1 \quad \mathrm{and} \quad s_{q}(n,0) = s_{q}(0,k) = 0 \ \mathrm{for}\  n,k > 0
\\
\textrm{(resp.} \quad &S_{q}(0,0) = 1 \quad \mathrm{and} \quad S_{q}(n,0) = S_{q}(0,k) = 0\ \mathrm{for}\  n,k > 0).
\end{align*} 
\end{linenomath}
\end{dfn}

We have the remarkable formulas for all $m,n \in \mathbb N$,
\[
\sum_{k} s_{q}(n,k)S_{q}(k,m) = \delta_{n,m} \quad \mathrm{and} \quad \sum_{k} S_{q}(n,k)s_{q}(k,m) = \delta_{n,m}
\]
where $\delta_{n,m}$ denotes the Kronecker symbol.
In other words, the infinite lower triangular matrices $[s_{q}(n,k)]$ and $[S_{q}(n,k)]$ are inverse to each other.
These identities can be directly proved by induction (see also theorem 5.2.7 of \cite{Ernst12}).

It may be worth noticing that, for all $n,k \in \mathbb N$,
\[
s_{q}(n,k) = (-1)^{n-k} \mathcal{S}_{n-k}((1)_{q}, (2)_{q}, \ldots, (n-1)_{q})
\]
where $\mathcal{S}_{k}$ denotes the $k$th elementary symmetric polynomial, and
\[
(q-1)^{n-k}s_{q}(n,k) = \sum_{j=k}^n (-1)^{n-j} q^{{n-j \choose 2}} {n \choose j}_{q} {j \choose k}.
\]
Again, both results are easily proved by induction.
The second one may also be derived from the quantum binomial formula (proposition 2.14 in \cite{LeStumQuiros15}).

Actually, twisted Stirling numbers are uniquely determined by the following fundamental property:

\begin{prop} \label{FunStir}
Let $X, Y$ be two indeterminates and, for any $n \in \mathbb N$,
\[
X_n := \prod_{k=0}^{n-1}\left(X - (k)_{q}Y\right).
\]
Then, for all $n \in \mathbb N$, we have
\begin{equation}\label{StirId}
X_n = \sum_{k=0}^{n} s_{q}(n,k) X^{k}Y^{n-k}
\quad
\mathrm{and}
\quad
X^n= \sum_{k=0}^{n} S_{q}(n,k) X_kY^{n-k}.
\end{equation}
\end{prop}

\begin{proof}
This is shown by induction or derived from definition 99 in \cite{Ernst12}) which is the case $Y=1$.
\end{proof}

The fundamental notion for us is the following (recall that $q$ and $p$ are fixed):

\begin{dfn} \label{twistom}
If $n \in \mathbb N$, then the \emph{$n$th $\mathbb\Delta$-power} of an indeterminate $\omega$ (with respect to $x \in R$) is
\[
\omega^{(n)_{\mathbb\Delta}}(x) := \prod_{k=0}^{n-1}\left(\omega - (k)_{q^p}(q-1)x\right) \in R[\omega].
\]
\end{dfn}

We shall simply write $\omega^{(n)_{\mathbb\Delta}}$ when $x$ is understood from the context.
Starting at section \ref{secabs}, we shall concentrate on the case $x = q$ so that $\omega^{(n)_{\mathbb\Delta}} = \omega^{(n)_{\mathbb\Delta}}(q)$.
In general, we have
\[
\omega^{(0)_{\mathbb\Delta}} = 1, \quad \omega^{(1)_{\mathbb\Delta}} = \omega, \quad \omega^{(2)_{\mathbb\Delta}} = \omega^2-(q-1)x\omega, \quad \ldots
\]

\begin{rmks} \phantomsection \label{twpow}
\begin{enumerate}
\item
The notion of $n$th $\mathbb\Delta$-power is a special instance of
\[
\xi^{(n)_{q,y}} := \prod_{k=0}^{n-1}(\xi + (k)_qy) \in R[\xi]
\]
from \cite{GrosLeStumQuiros22a} where we replace $\xi$, $q$ and $y$ by $\omega$, $q^p$ and $(1-q)x$ respectively.
\item In the particular case $p=1$ (recall that we do not require $p$ to be prime at this point), we will simply say \emph{$q$-power}.
More generally, in this situation, we will systematically remove $p$ from the notations and write ``$q$-'' instead of ``$\mathbb\Delta$-''.
We shall then use $\xi$ and not $\omega$ for the indeterminate so that the $n$th $q$-power of $\xi$ (with respect to $x$) is
\[
\xi^{(n)_{q}} (x) := \prod_{k=0}^{n-1}\left(\xi - q^kx +x\right) \in R[\xi].
\]
This is consistent with \cite{GrosLeStumQuiros22a} when we set $y=(1-q)x$.
This is essentially the same thing as a \emph{generalized Pochammer symbol} of \cite{AnschuetzLeBras19a}: in their notation, we have $\xi^{(n)_{q}}(x) = (\xi + x, -x;q)_n$ and  conversely $(x,y;q)_n = (x+y)^{(n)_q}(-x)$.
\item Applying the previous case (where $p$ does not appear) with $q$ replaced with $q^p$, we get the $n$th $q^p$-power of $\xi$:
\[
\xi^{(n)_{q^p}}(x) := \prod_{k=0}^{n-1}\left(\xi - q^{pk}x + x\right) \in R[\xi].
\]
This should not be confused with the $n$th $\mathbb\Delta$-power (and using two different indeterminates $\xi$ and $\omega$ should help).
\item
$q^p$-powers and $\mathbb\Delta$-powers are not unrelated: it follows from lemma 1.4 in \cite{GrosLeStumQuiros22b} that blowing up
\[
R[\xi] \to R[\omega], \quad \xi \mapsto (p)_q\omega
\]
maps twisted powers as follows:
 \[
\xi^{(n)_{q^p}} \mapsto (p)_q^n\omega^{(n)_{\mathbb\Delta}}.
\]
\end{enumerate}
\end{rmks}

Twisted powers form an alternative basis for $R[\omega]$ as an $R$-module and there exists explicit formulas involving twisted Stirling numbers:

\begin{prop} \label{ernst}
For all $n > 0$, we have
\[
\omega^{(n)_{\mathbb\Delta}} = \sum_{k=1}^{n} s_{q^p}(n,k) (q -1)^{n-k}x^{n-k}\omega^{k}
\]
and
\[
\omega^{n}= \sum_{k=1}^{n} S_{q^p}(n,k) (q -1)^{n-k}x^{n-k}\omega^{(k)_{\mathbb\Delta}}.
\]
\end{prop}

\begin{proof}
Follows from the identities in proposition \ref{FunStir}.
\end{proof}

One may always consider the twisted completion
\[
R[[\omega]]_{\mathbb\Delta} := \varprojlim R[\omega]/\omega^{(n+1)_{\mathbb\Delta}}
\]
but this will not be necessary in practice thanks to following:

\begin{lem} \label{qlim}
If $R$ is complete for the $(q-1)$-adic topology, then
\[
R[[\omega]] := \varprojlim R[\omega]/\omega^{n+1} \simeq \varprojlim R[\omega]/\omega^{(n+1)_{\mathbb\Delta}} =: R[[\omega]]_{\mathbb\Delta}.
\]
\end{lem}

\begin{proof}
It follows from the formulas in proposition \ref{ernst} that
\[
\forall n \in \mathbb N, \quad \omega^{(2n)_{\mathbb\Delta}} \in (\omega^n) + ((q-1)^n) \quad \mathrm{and} \quad \omega^{2n} \in (\omega^{(n)_{\mathbb\Delta}}) + ((q-1)^n). \qedhere
\]
\end{proof}

\begin{dfn} \label{defdtw}
The ring of \emph{$\mathbb\Delta$-divided powers} (with respect to $x$) is the free $R$-module $R\langle \omega \rangle_{\mathbb\Delta}$ on the basis $(\omega^{\{n\}_{\mathbb\Delta}})_{n\in \mathbb N}$ with multiplication rule
\begin{linenomath}
\begin{align*}
&\omega^{\{n_{1}\}_{\mathbb\Delta}} \omega^{\{n_{2}\}_{\mathbb\Delta}} =
\\&\sum_{0\leq i \leq \min\{n_{1},n_{2}\}} q^{\frac{pi(i-1)}2}{n_{1} + n_{2} -i \choose n_{1}}_{q^{p}}{n_{1} \choose i}_{q^{p}} (q-1)^ix^i\omega^{\{n_{1}+n_{2}-i\}_{\mathbb\Delta} }.
\end{align*}
\end{linenomath}
\end{dfn}

\begin{rmks} \phantomsection \label{divpow}
\begin{enumerate}
\item
This was called the ring of \emph{twisted divided powers of level $-1$} in \cite{GrosLeStumQuiros22b}, where we assumed that $p$ is prime.

\item
The ring $R\langle \omega \rangle_{\mathbb\Delta}$ is a particular instance of the ring of twisted divided powers $R\langle \xi \rangle_{q,y}$ of \cite{GrosLeStumQuiros22a} where we replace $\xi$, $q$ and $y$ by $\omega$, $q^p$ and $(1-q)x$ respectively and use curly brackets instead of usual brackets.
\item In the particular case $p=1$, we will write $\xi$ instead of $\omega$, $q$ instead of $\mathbb\Delta$ and use standard brackets instead of curly ones.
Thus, the ring of $q$-divided powers is a free $R$-module $R\langle \xi \rangle_{q}$ with basis $(\xi^{[n]_{q}})_{n \in \mathbb N}$ and fancy multiplication similar to that in definition \ref{defdtw}.
Again, this is consistent with \cite{GrosLeStumQuiros22a} when we set $y=(1-q)x$.
\item Replacing $q$ with $q^p$ in the previous situation, we shall also consider the ring $R\langle \xi \rangle_{q^p}$ with basis $(\xi^{[n]_{q^p}})_{n \in \mathbb N}$ (and an explicit multiplication formula).
\item
It is shown in proposition 1.7 of \cite{GrosLeStumQuiros22b} that there exists a morphism of rings (blowing up)
\[
R\langle \xi \rangle_{q^p}\to R\langle \omega \rangle_{\mathbb\Delta}, \quad \xi^{[n]_{q^p}} \mapsto (p)_q^n\omega^{\{n\}_{\mathbb\Delta}}.
\]
\end{enumerate}
\end{rmks}

It follows from proposition 2.2 of \cite{GrosLeStumQuiros22a} that definition \ref{defdtw} provides indeed a ring structure such that $\omega^{\{0\}_{\mathbb\Delta}} = 1$ and we will write $\omega = \omega^{\{1\}_{\mathbb\Delta}}$.
Moreover, there exists a unique natural morphism of $R$-algebras
\begin{equation}
R[\omega] \to R\langle \omega \rangle_{\mathbb\Delta}, \quad \omega^{(n)_{\mathbb\Delta}} \mapsto (n)_{q^p}!\omega^{\{n\}_{\mathbb\Delta}}. \label{polcp}
\end{equation}
The free submodule $\omega^{\{>n\}_{\mathbb\Delta}}$ with basis $\omega^{(k)_{\mathbb\Delta}}$ for $k > n$ is an ideal of $R\langle \omega \rangle_{\mathbb\Delta}$.
In the case $n=0$, this is the augmentation ideal or, in other words, the kernel of the morphism of rings
\[
R\langle \omega \rangle_{\mathbb\Delta} \to R, \quad \omega^{(n)_{\mathbb\Delta}} \mapsto 0 \ \mathrm{for} \ n > 0.
\]
Be careful that the ideal $\omega^{\{>n\}_{\mathbb\Delta}}$ is not the same as the ideal generated by $\omega^{\{n+1\}_{\mathbb\Delta}}$ so that $(\omega^{\{n+1\}_{\mathbb\Delta}}) \subsetneq \omega^{\{>n\}_{\mathbb\Delta}}$ in general.
However, there exists an isomorphism of $R$-algebras
\[
R[\omega]/\omega^{(2)_{\mathbb\Delta}} \simeq R\langle \omega \rangle_{\mathbb\Delta}/\omega^{\{>1\}_{\mathbb\Delta}}.
\]

If we are given a morphism of commutative rings $R \to S$ and we still denote by $q$ and $x$ their images in $S$, then we have
\[
S \otimes_R R\langle \omega \rangle_{\mathbb\Delta} \simeq S\langle \omega \rangle_{\mathbb\Delta} \quad \mathrm{and} \quad S \otimes_R R\langle \omega \rangle_{\mathbb\Delta}/\omega^{\{>n\}_{\mathbb\Delta}} \simeq S\langle \omega \rangle_{\mathbb\Delta}/\omega^{\{>n\}_{\mathbb\Delta}}.
\]
It is in particular possible to reduce many assertions to simpler base rings such as $R=\mathbb Z[q,x]$ for example, as in the proof of the following statement (that will be used later):

\begin{lem} \label{estcong}
For all $n \in \mathbb N$, we have
\[
(x + (p)_q\omega) \omega^{\{n\}_{\mathbb\Delta}} \equiv q^{pn}x\omega^{\{n\}_{\mathbb\Delta}} \mod \omega^{\{n+1\}_{\mathbb\Delta}}.
\]
\end{lem}

\begin{proof}
Let us first show that
\[
\omega\omega^{\{n\}_{\mathbb\Delta}} - (n)_{q^p}(q-1)x\omega^{\{n\}_{\mathbb\Delta}}= (n+1)_q\omega^{\{n+1\}_{\mathbb\Delta}}.
\]
This is completely formal and we may therefore assume that $R = \mathbb Z[q,x]$.
After multiplication by $(n)_{q}!$, it is therefore sufficient, using the natural map \eqref{polcp}, to recall from definition \ref{twistom} that
\[
(\omega - (n)_{q^p}(q-1)x)\omega^{(n)_{\mathbb\Delta}} = \omega^{(n+1)_{\mathbb\Delta}}.
\]
We can now compute:
\begin{linenomath}
\begin{align*}
(x + (p)_q\omega) \omega^{\{n\}_{\mathbb\Delta}} &\equiv (x + (p)_q (n)_{q^p}(q-1)x) \omega^{\{n\}_{\mathbb\Delta}} \mod \omega^{\{n+1\}_{\mathbb\Delta}}
\\ &\equiv q^{pn}x\omega^{\{n\}_{\mathbb\Delta}} \mod \omega^{\{n+1\}_{\mathbb\Delta}}. \qedhere
\end{align*}
\end{linenomath}
\end{proof}

\begin{rmks}
\begin{enumerate}
\item
In the case $p=1$ (and then writing $\xi^{[n]_{q}}$ instead of $\omega^{\{n\}_{\mathbb\Delta}}$), lemma \ref{estcong} reads
\[
(x +\xi) \xi^{[n]_{q}} \equiv q^nx\xi^{[n]_{q}} \mod \xi^{[n+1]_{q}}.
\]
\item
Replacing now $q$ with $q^p$ provides a formula that we shall need later:
\begin{equation} \label{thetsig}
(x+ \xi)\xi^{[n]_{q^p}} \equiv q^{pn}x\xi^{[n]_{q^p}} \mod \xi^{[n+1]_{q^p}}. \qedhere
\end{equation}

\end{enumerate}
\end{rmks}

All of the above makes sense when $q=1$, in which case we fall onto the usual ring of divided polynomials $R\langle \omega \rangle$.
In other words $R\langle \omega \rangle$ is the free $R$-module with basis $(\omega^{[n]})_{n\in \mathbb N}$ and multiplication rule
\[
\omega^{[n_{1}]} \omega^{[n_{2}]} = {n_{1} + n_{2} \choose n_{1}} \omega^{[n_{1}+n_{2}] }.
\]

\begin{prop} \label{poweg2}
If $R$ is a $\mathbb Q$-algebra and $q^p-1$ belongs to the Jacobson radical of $R$, then
\[
R[\omega] \simeq R\langle \omega \rangle_{\mathbb\Delta}.
\]
Actually, we have
\begin{equation} \label{Stirall}
\omega^n= \sum_{k=1}^{n}(k)_{q^p}! S_{q^p}(n,k) (q-1)^{n-k} x^{n-k} \omega^{\{k\}_{\mathbb\Delta}}
\end{equation}
and
\[
\omega^{\{n\}_{\mathbb\Delta}} = \sum_{k=1}^{n} \frac 1{(n)_{q^p}!}s_{q^p}(n,k)(q-1)^{n-k} x^{n-k}\omega^k
\]
where $s_{q^p}(n,k)$ (resp.\ $S_{q^p}(n,k)$) denotes the $q^p$-Stirling numbers of the first (resp.\ second) kind.
\end{prop}

\begin{proof}
Both formulas from proposition \ref{ernst} when they make sense.
It is therefore sufficient to show that $(n)_{q^p} \in R^\times$ for all $n \neq 0$.
Since $q^p-1$ belongs to the Jacobson radical of $R$, this claim follows from the fact that $(n)_{q^p} \equiv n \mod q^p-1$ and that $n \in \mathbb Q^\times$.
\end{proof}

\begin{rmks}
\begin{enumerate}
\item
Jacobson condition in proposition \ref{poweg2} is satisfied for example when $R$ is $(q-1)$-adically complete or $(p)_q$-adically complete.
\item Formula \eqref{Stirall} obviously holds over any commutative ring $R$.
\item Since, under the hypothesis of proposition \ref{poweg2}, we have $R[\omega] = R\langle \omega \rangle$ (usual divided power polynomial ring), then also $R\langle \omega \rangle \simeq R\langle \omega \rangle_{\mathbb\Delta}$ with
\[
\omega^{[n]}= \sum_{k=1}^{n} \frac{(k)_{q^p}!}{n!} S_{q^p}(n,k) (q-1)^{n-k} x^{n-k} \omega^{\{k\}_{\mathbb\Delta}}
\]
and
\[
\omega^{\{n\}_{\mathbb\Delta}} = \sum_{k=1}^{n} \frac {k!}{(n)_{q^p}!}s_{q^p}(n,k)(q-1)^{n-k} x^{n-k}\omega^{[k]}.
\]
\end{enumerate}
\end{rmks}

Recall from definition 1.4 in \cite{LeStumQuiros15} that the \emph{$q$-characteristic} of a ring $R$ is the smallest positive integer $p$ such that $(p)_q = 0$ in $R$ (or $0$ otherwise).

\begin{prop} \label{poweg}
If $R$ is a $\mathbb Z_{(p)}$-algebra of prime $q$-characteristic $p$, then
\[
R\langle \omega \rangle \simeq R\langle \omega \rangle_{\mathbb\Delta}.
\]
Actually, we have
\[
\omega^{[n]} = \sum_{k=1}^n \frac {k!}{n!} S(n,k)(q-1)^{n-k}x^{n-k}\omega^{\{k\}_{\mathbb\Delta}}
\]
and
\[
\omega^{\{n\}_{\mathbb\Delta}} = \sum_{k=1}^n \frac {k!}{n!}s(n,k) (q-1)^{n-k}x^{n-k}\omega^{[k]}
\]
where $s(n,k)$ (resp.\ $S(n,k)$) denotes the (usual) Stirling number of the first (resp.\ second) kind.
\end{prop}

\begin{proof}
Thanks to proposition 1.15 of \cite{LeStumQuiros15}, we may assume that $R=\mathbb Z_{(p)}[\zeta][x]$ where $\zeta$ is a primitive $p$th root of unity and that $q = \zeta$.
Then, it follows from proposition \ref{ernst} that
\[
\omega^{n} = \sum_{k=1}^n k!S(n,k) (\zeta-1)^{n-k}x^{n-k}\omega^{\{k\}_{\mathbb\Delta}}.
\]
Therefore (making $p$ invertible):
\[
\omega^{[n]} = \sum_{k=1}^n \frac {k!}{n!}S(n,k) (\zeta-1)^{n-k}x^{n-k}\omega^{\{k\}_{\mathbb\Delta}} \in \mathbb Q_{p}(\zeta)[x]\langle \omega \rangle_{\mathbb\Delta}.
\]
However, we know from lemma 2.1 in \cite{Adelberg18} that, if we denote by $s(k)$ the sum of the digits of $k$ in its $p$-adic expansion, then
\[
v_p(S(n,k)) \geq \frac {s(k)-s(n)}{p-1}.
\]
It follows that
\begin{linenomath}
\begin{align*}
v_{p}\left(\frac {k!}{n!}S(n,k) (\zeta-1)^{n-k}\right) &\geq \frac {k-s(k)}{p-1} - \frac {n-s(n)}{p-1} +\frac {s(k)-s(n)}{p-1} + \frac {n-k}{p-1}
\\ &= 0,
\end{align*}
\end{linenomath}
so that we do not really need to make $p$ invertible.
This provides the first formula and the other one formally follows.
\end{proof}

\begin{rmks}
\begin{enumerate}
\item
Assume that $R$ is a complete bounded $\delta$-$\mathbb Z[q]_{(p,q-1)}$-algebra (with $p$ prime) so that we may consider the prism $(R,(p)_q)$.
Assume also that $A$ is a complete $R$-algebra for the $(p,q-1)$-adic topology and $x$ is a topologically \'etale coordinate on $A$.
We endow $A$ with the unique $\delta$-structure such that $\delta(x)=0$.
We know from proposition 5.7 of \cite{GrosLeStumQuiros22b} that the completion of $A\langle \omega \rangle_{\mathbb\Delta}$ is the prismatic envelope of the diagonal of $A$ over $R$.
Then, we recover proposition 4.3 in \cite{Tian23}: there exists a canonical isomorphism
\[
R/(p)_q \otimes^\mathrm{L}_R \widehat {A\langle \omega \rangle} \simeq R/(p)_q \otimes^\mathrm{L}_R \widehat{A\langle \omega \rangle}_{\mathbb\Delta}.
\]
\item There also exists an ``absolute'' analog of the result of Tian that we shall discuss in section \ref{reduced}.
In particular, we will recover lemma 2.13 of \cite{GaoMinWang23} in our situation.
\end{enumerate}
\end{rmks}

We shall end this section with a technical result whose main argument has been provided to us by Yu Min following a question about \cite{GaoMinWang22}.
This concerns the case where $p$ is prime.
The proof relies on the theory of $\delta$-rings and we first need to briefly explain how $\delta$-structures spread out to our rings (see section 2 of \cite{BhattScholze22} or section 1 of \cite{GrosLeStumQuiros23a} for an introduction to this notion).
Let us just recall that a $\delta$-ring always has a frobenius $\phi$ and that both notions are equivalent when the ring has no $p$-torsion.
Assume that $R$ is a $\delta$-ring such that both $q$ and $x$ have rank one, meaning that $\delta(x) = \delta(q) = 0$ (so that $\phi(q) = q^p$ and $\phi(x) = x^p$).
Then, the frobenius of $R$ extends uniquely to $R[\xi]$ in such a way that $\phi(x + \xi) = (x+ \xi)^p$.
More precisely, we have
\[
\phi(\xi) := (x+\xi)^p - x^p = \sum_{k=1}^p {p \choose k}x^{p-k}\xi^k.
\]
Similarly, the $\delta$-structure extends uniquely in a natural way to $R[\xi]$ through
\begin{equation} \label{delxi}
\delta(\xi) := \sum_{k=1}^{p-1} \frac 1p{p \choose k}x^{p-k}\xi^k.
\end{equation}

When $R$ is a general topological ring, we endow $R\langle \omega \rangle_{\mathbb\Delta}$ with the direct sum topology with respect to the basis $(\omega^{\{n\}_{\mathbb\Delta}})_{n\in \mathbb N}$.
This applies in particular to $R\langle \omega \rangle$.
Then, in order to improve on proposition \ref{poweg2}, we shall need the following:

\begin{prop} \label{fromMin}
Assume $R$ is a $p$-adic Tate ring\footnote{A topological commutative $\mathbb Q$-algebra that admits an open $p$-adic subring with $p$ prime.}.
If $(p)_q$ is nilpotent in $R$, then, the canonical map
\[
R\langle \omega \rangle \to R\langle \omega \rangle_{\mathbb\Delta}
\]
is continuous.
\end{prop}

\begin{proof}
We have to show that for some open $p$-adic subring $R_0$ of $R$, we have
\[
\exists N \in \mathbb N, \forall n \in \mathbb N, \quad p^N\omega^{[n]} \in R_0\langle \omega \rangle_{\mathbb\Delta}.
 \]
We may assume that $R = R_0[1/p]$ with $R_0 := \left(\mathbb Z_{(p)}[q]/(p)_q^m\right)[x]$.
We endow $\mathbb Z_{(p)}[q, x]$ with the unique $\delta$-structure such that both $q$ and $x$ have rank one.
Then, we endow $A := \left(\mathbb Z_{(p)}[q, x]\right)\langle \omega \rangle_{\mathbb\Delta}$ with the $\delta$-structure coming from proposition 1.12 in \cite{GrosLeStumQuiros22b} so that formula \eqref{delxi} holds for $\xi := (p)_q\omega$.
Assume for the time being that we have proved the following estimates\footnote{These estimates have been shown to us by Yu Min.}:
\begin{equation} \label{Min}
\forall k \in \mathbb N, \exists F_k \in A\left[\frac {(p)_q}{p^2}\right], \quad \omega^{p^k} = (-p)^{(k)_p}(pF_k + \delta^{k}(\omega)),
\end{equation}
where $(k)_p$ denotes the $p$-analog of $k$.
Since $v_p(p^k!) = (k)_p$, this will imply that $\omega^{[n]} \in A[(p)_q/p^2]$ when $n = p^k$, and the same therefore holds for all $n \in \mathbb N$.
As a consequence, we will have $p^N \omega^{[n]} \in R_0\langle \omega \rangle_{\mathbb\Delta}$ with $N = 2(m-1)$.

It remains to prove formula \eqref{Min}, but as an intermediate step, we shall check the following (see lemma 2.11 in \cite{GaoMinWang23}):
\[
\forall r \geq 0, s > 0, \quad \phi^s(\delta^r(\omega)) \equiv 0 \mod (p)_q.
\]
Note first that $((p)_q,p)$ is a regular sequence in $A$, and therefore Gauss lemma\footnote{Gauss lemma holds for $f,g \in A$ means:``$\forall h \in A, (f \mid gh \Rightarrow f \mid h)$ and $(g \mid fh \Rightarrow g \mid h$)''.} holds for $p$ and $(p)_q$.
Since $(p)_{q^{p^s}} \equiv p \mod (p)_q$, Gauss lemma also holds for $(p)_{q^{p^s}}$ and $(p)_q$.
Now, since $\xi \mid \phi(\xi)$, we have $\phi(\xi) \mid \phi^s(\xi)$.
Since $(p)_q \mid \phi(\xi)$ in $A$, it follows that $(p)_q \mid \phi^s(\xi) = \phi^s((p)_q\omega) = (p)_{q^{p^s}} \phi^s(\omega)$.
We may then apply Gauss lemma and this settles the case $r=0$.
Now, we have
\[
\phi^{s+1}(\delta^r(\omega)) = \phi^{s}(\delta^r(\omega))^p + p\phi^{s}(\delta^{r+1}(\omega)),
\]
so that, by induction on $r$, $(p)_q \mid p\phi^{s}(\delta^{r+1}(\omega))$ and we conclude with Gauss lemma again.

Now, we prove formula \eqref{Min} by induction on $k$, the case $k=0$ being fine.
We have
\[
\phi(\delta^k(\omega)) = \delta^k(\omega)^p + p \delta^{k+1}(\omega).
\]
From what we just proved, $(p)_q \mid \phi(\delta^k(\omega))$, and we can therefore write
\[
\delta^k(\omega)^p = p^2 \frac {(p)_q}{p^2}\frac{\phi(\delta^k(\omega))}{(p)_q} - p \delta^{k+1}(\omega).
\]
Now, by induction, we have
\begin{linenomath}
\begin{align*}
\omega^{p^{k+1}} &= \left((-p)^{(k)_p}(pF_k + \delta^{k}(\omega))\right)^p
\\ &=(-p)^{p(k)_p} \left((pF_k)^p + \sum_{i=1}^{p-1} {p \choose i} \delta^{k}(\omega)^{p-i}(pF_k)^i + \delta^{k}(\omega)^p\right)
\\ &=(-p)^{p(k)_p} \left(p^pF_k^p + \sum_{i=1}^{p-1}p^i {p \choose i} \delta^{k}(\omega)^{p-i}F_k^i + p^2 \frac {(p)_q}{p^2}\frac{\phi(\delta^k(\omega))}{(p)_q} - p \delta^{k+1}(\omega)\right)
\\ &=(-p)^{(k+1)_p}(pF_{k+1} + \delta^{k+1}(\omega))
\end{align*}
\end{linenomath}
with
\[
F_{k+1} = -p^{p-2}F_k^p - \sum_{i=1}^pp^{i-2}{p \choose i} \delta^{k}(\omega)^{p-i} F_k^i - \frac {(p)_q}{p^2}\frac{\phi(\delta^k(\omega))}{(p)_q}.\qedhere
\]
\end{proof}

\section{Absolute calculus} \label{secabs}

We shall present here a new\footnote{As far as we know.} theory that we call \emph{absolute calculus}.
This is similar to relative $q$-calculus (section 6 of \cite{GrosLeStumQuiros23a} for example) or relative $\mathbb\Delta$-calculus (section 3 of \cite{GrosLeStumQuiros22b}), but the parameter $q$ now coincides with the coordinate $x$.

We let $W$ be an adic ring (that we assume to be complete from the beginning for simplicity) with finitely generated ideal of definition $\mathfrak m_{W}$ and set $R := W[[q-1]]$ where $q$ is an indeterminate.
Any $R$-module $M$ will be endowed with the $\mathfrak m_{W} + (q-1)$-adic topology but we shall write $\overline M := M/(p)_qM$.
The map $W \to R$ is not adic and we need to be careful with continuity issues.

Note that $R$ is $q^r$-flat of $q^r$-characteristic zero (see definitions 1.4 and 1.9 of \cite{LeStumQuiros15}) for all $r >0$.
It simply means that, as long as $k \neq 0$ and $r \neq 0$, the $q^r$-analog $(k)_{q^r}$ is regular in $R$ and this is true because $(k)_{q^r}$ is a monic polynomial in $q$.
This property is the twisted analog of being $\mathbb Z$-torsion free and it will simplify many arguments.

We fix a positive integer $p$ (that need not be a prime for now).
If $r \in \mathbb Z$ (or more generally $r \in \mathbb Z_p$ if $p$ is topologically nilpotent in $W$), we shall denote by $g_r$ the unique continuous endomorphism of $R$ such that $g_r(q) = q^r$.
The $W$-algebra $R$ is turned into a twisted $W$-algebra in the sense of \cite{LeStumQuiros18b} by choosing the endomorphism $\gamma := g_{p+1}$ (so that $\gamma(q) = q^{p+1}$).
What we shall call \emph{absolute} could also be named after Mahler (see for example \cite{Roques21}) since \emph{Mahler functional equations} are associated to an endomorphism of the form $x \mapsto x^r$.
In our situation, we have $x=q$ and $r=p+1$.

We recall now from definition 4.3 of \cite{LeStumQuiros18b} the notion of a $\gamma$-derivation that we shall call here a $q^p$-derivation:

\begin{dfn}
If $M$ is an $R$-module, then a $W$-linear map $D : R \to M$ is a \emph{$q^p$-derivation} if 
\begin{equation} \label{der}
\forall \alpha,\beta \in R, \quad D(\alpha\beta) = D(\alpha)\beta +\gamma(\alpha)D(\beta).
\end{equation}
\end{dfn}

The $q^p$-derivations of $R$ (meaning from $R$ to itself) form an $R$-module $T_{R/W, q^p}$ (or $T_{q^p}$ for short).

\begin{lem} \label{Griff}
If $D : R \to M$ is a $q^p$-derivation, then
\[
\forall n \in \mathbb N, \quad D((q-1)^{n+1}R) \subset (q-1)^{n}M.
\]
\end{lem}

\begin{proof}
This is obtained by induction from the identity
\[
D((q-1)^{n+1}) = (q-1)^n D(q-1) + (p+1)_{q}(q-1) D((q-1)^n). \qedhere
\]
\end{proof}

\begin{rmks}
\begin{enumerate}
\item
As a consequence of lemma \ref{Griff}, a $q^p$-derivation is automatically continuous.
\item
One can also recall the general formula
\begin{equation} \label{derpo}
\forall \alpha \in R, \forall n \in \mathbb N, \quad D(\alpha^{n+1}) = \sum_{k=0}^n \gamma(\alpha)^{n-k}\alpha^kD(\alpha)
\end{equation}
 for a $\gamma$-derivation and obtain the lemma as a corollary.
\end{enumerate}
\end{rmks}

\begin{lem} \label{dersimp}
If $M$ is an $R$-module, then a $W$-linear map $D : R \to M$ is a $q^p$-derivation if and only if
\[
\forall k > 0, \quad D(q^{k}) = (k)_{q^p}q^{k-1}D(q).
\]
\end{lem}

\begin{proof}
First of all, by additivity and continuity, formula \eqref{der} reduces to
\[
\forall k,l \in \mathbb N, \quad D(q^{k+l}) = q^lD(q^k)+q^{(p+1)k}D(q^l).
\]
Now, the property is shown by induction:
\begin{linenomath}
\begin{align*}
D(q^{k+1}) & = qD(q^k)+q^{(p+1)k}D(q)
\\ & =q(k)_{q^p}q^{k-1}D(q)+q^{(p+1)k}D(q)
\\ & = ((k)_{q^p}+q^{pk})q^{k}D(q)
\\ & = (k+1)_{q^p}q^{k}D(q).
\end{align*}
\end{linenomath}
Conversely, we will have
\begin{linenomath}
\begin{align*}
q^lD(q^k)+q^{(p+1)k}D(q^l) & = q^l(k)_{q^p}q^{k-1}D(q)+q^{(p+1)k}(l)_{q^p}q^{l-1}D(q)
\\ & = ((k)_{q^p}+q^{pk}(l)_{q^p})q^{l+k-1}D(q)
\\ & = (k+l)_{q^p}q^{k+l-1}D(q)
\\ & = D(q^{k+l}).\qedhere
\end{align*}
\end{linenomath}
\end{proof}

\begin{prop} \phantomsection \label{basder}
\begin{enumerate}
\item
The $R$-module $T_{q^p}$ is free of rank one with generator $\partial_{q^p}$ given by
\[
\forall k > 0, \quad \partial_{q^p}(q^k) = (k)_{q^p}q^{k-1}.
\]
\item We have
\[
\forall n >1, \quad \partial_{q^p}((q-1)^n) = (n)_{(p+1)_q}(q-1)^{n-1}.
\]
\item If we set $\partial_{\mathbb\Delta} := (p)_q\partial_{q^p}$, then
\[
\forall k > 0, \quad \partial_{\mathbb\Delta}(q^k) =(pk)_{q}q^{k-1}.
\]
\item We have $\gamma = \mathrm{Id}_R + (q^2-q)\partial_{\mathbb\Delta}$.
\end{enumerate}
\end{prop}

\begin{proof}
The first assertion follows immediately from lemma \ref{dersimp} which provides $D=D(q)\partial_{q^p}$.
The second one is deduced from formula \eqref{derpo}:
\begin{linenomath}
\begin{align*}
\partial_{q^p}((q-1)^n) & = \sum_{k=0}^{n-1} (q^{p+1}-1)^{k}(q-1)^{n-1-k}
\\ & = \sum_{k=0}^{n-1} (p+1)_q^{k}(q-1)^{n-1}
\\ & = (n)_{(p+1)_q}(q-1)^{n-1}. 
\end{align*}
\end{linenomath}
The third statement comes from $\partial_{\mathbb\Delta}(q^k) = (p)_q(k)_{q^p}q^{k-1} =(pk)_{q}q^{k-1}$.
For the last assertion, it is sufficient, by continuity, to notice that
\[
\forall k > 0, \quad q^{(p+1)k} = q^k + (q^2-q)(pk)_qq^{k-1}.\qedhere
\]

\end{proof}


\begin{rmks} \phantomsection \label{rmksDqp}
\begin{enumerate}
\item The endomorphism $\partial_{\mathbb\Delta}$ is a kind of ``logarithmic version'' of $\partial_{q^p}$ (with a pole at $(p)_q$).
Il will be essential for the theory to focus on $\partial_{\mathbb\Delta}$ and consider $\partial_{q^p}$ as an intermediate tool.
Using the last formula of proposition \ref{basder}, we could as well directly set
\[
\partial_{\mathbb\Delta} := \frac{\gamma - \mathrm{Id}_R}{q^2-q}.
\]
Note also that, since $q$ is invertible, we would get an equivalent theory by using
\[
q\partial_{\mathbb\Delta} = \frac{\gamma - \mathrm{Id}_R}{q-1}
\]
(by adding an extra ``pole'' at $q$) instead of $\partial_{\mathbb\Delta}$.
Most formulas however would then be a little more complicated.
\item
Even if our absolute calculus has nothing to do with Buium's theory of arithmetic differential equations, it happens that our $q^p$-derivations fit into his general pattern.
More precisely, the endomorphism $\partial_{q^p}$ (resp.\ $\partial_{\mathbb\Delta}$) is a $\pi$-difference operator in his sense (see example (c) in \cite{Buium97}) with $\pi = (q^2-q)(p)_q$ (resp.\ $\pi = (q^2-q)$).
It means that we have the following symmetric formulas for $\alpha, \beta \in R$:
\[
\partial_{q^p}(\alpha\beta) = \partial_{q^p}(\alpha)\beta + \alpha \partial_{q^p}(\beta) + (q^2-q)(p)_q\partial_{q^p}(\alpha) \partial_{q^p}(\beta)
\]
and
\[
\partial_{\mathbb\Delta}(\alpha\beta) = \partial_{\mathbb\Delta}(\alpha)\beta + \alpha \partial_{\mathbb\Delta}(\beta) + (q^2-q)\partial_{\mathbb\Delta}(\alpha) \partial_{\mathbb\Delta}(\beta).
\]
\item For further use, note that
\begin{equation} \label{pqplun}
(p)_{q^{p+1}} = \lambda (p)_q \quad \mathrm{with} \quad \lambda = 1 + (q^2-q)\partial_{q^p}((p)_q) \in R^\times.
\end{equation}
This is easily checked: since $\gamma(q) = q^{p+1}$, we have
\[
(p)_{q^{p+1}} = \gamma((p)_q) = (p)_q + (q^2-q)\partial_{\mathbb\Delta}((p)_q) = (1 + (q^2-q)\partial_{q^p}((p)_q))(p)_q.
\]
\end{enumerate}
\end{rmks}

We now come to our main object of study which is a kind of logarithmic twisted derivation:

\begin{dfn} \label{qder1}
A \emph{$\mathbb\Delta$-derivation} on an $R$-module $M$ with respect to a $q^p$-derivation $D$ of $R$ is a $W$-linear map $D_M : M \to M$ such that
\[
\forall \alpha \in R, s \in M, \quad D_M(\alpha s) = (p)_qD(\alpha)s+\gamma(\alpha)D_M(s).
\]
\end{dfn}

There exists a variant without pole: a \emph{$q^p$-derivation} on an $R$-module $M$ with respect to a $q^p$-derivation $D$ is a $W$-linear map $D_M : M \to M$ such that
\[
\forall \alpha \in R, s \in M, \quad D_M(\alpha s) = D(\alpha)s+\gamma(\alpha)D_M(s).
\]
Actually, a $\mathbb\Delta$-derivation with respect to $D$ is the same thing as a $q^p$-derivation with respect to $(p)_qD$.
In particular, in definition \ref{qder1}, $D_M$ only depends on $(p)_qD$ (and not on $D$).
One also sees that $(p)_qD$ is uniquely determined by $D_M$ when $M$ is $R$-torsion free.

\begin{prop} \label{Griff}
\begin{enumerate}
\item 
A $\mathbb\Delta$-derivation $D_M$ satisfies \emph{Griffiths transversality} with respect to the $(q-1)$-adic filtration:
\[
\forall n \in \mathbb N, \quad  D_M((q-1)^{n+1}M) \subset (q-1)^{n}M.
\]
\item The $(p,q-1)$-adic filtration is stable under $D_M$:
\[
\forall n \in \mathbb N, \quad D_M((p,q-1)^{n}M) \subset (p,q-1)^{n}M.
\]
\end{enumerate}
\end{prop}

\begin{proof}
Both results follow from lemma \ref{Griff}.
\end{proof}

As an immediate consequence, we see that a $\mathbb\Delta$-derivation is automatically continuous.

\begin{rmks} \label{Nyg}
\begin{enumerate}
\item
Let us fix some notations for the $(q-1)$-adic filtration and the corresponding graduation:
\[
\forall n \in \mathbb N, \quad \mathrm{Fil}^n_{(q-1)}M := (q-1)^{n}M \quad \mathrm{and} \quad \mathrm{Gr}_{(q-1)}^nM := (q-1)^{n}M/(q-1)^{n+1}M.
\]
Then, a $\mathbb\Delta$-derivation $D_M$ on $M$ induces for each $n \in \mathbb N$ a $W$-linear map
\[
D_M : \mathrm{Gr}_{(q-1)}^{n+1}M \to \mathrm{Gr}_{(q-1)}^{n}M.
\]
\item
In the case $M = R$, then the $(q-1)$-adic filtration may also be called the \emph{Nygaard filtration} and, for each $n \in \mathbb N$, multiplication by $(q-1)^n$ induces an isomorphism $W \simeq \mathrm {Gr}_{(q-1)}^nR$.
Under these identifications, the induced map $\partial_{\mathbb\Delta} : \mathrm {Gr}_{(q-1)}^{n+1}R \to \mathrm {Gr}_{(q-1)}^{n}R$ corresponds to multiplication by $(p+1)^{n+1}-1$.
\end{enumerate}
\end{rmks}

There exists a unique $\mathbb\Delta$-derivation on $R$ itself with respect to a given $q^p$-derivation $D$, and this is simply $(p)_qD$.
Therefore, the $\mathbb\Delta$-derivations of $R$ form a submodule $T_{\mathbb\Delta} := (p)_qT_{q^p}$ of $T_{q^p}$ (we may also write $T_{R,\mathbb\Delta}$ or $T_{R/W,\mathbb\Delta}$ in order to specify the data).
Of course $T_{\mathbb\Delta}$ is free of rank one on the generator $\partial_{\mathbb\Delta}$.
When we simply say that $\partial_{M,\mathbb\Delta}$ is a \emph{$\mathbb\Delta$-derivation} of $M$, we always mean ``with respect to $\partial_{q^p}$'' and the condition reads
\[
\forall \alpha \in R, s \in M, \quad \partial_{M,\mathbb\Delta}(\alpha s) = \partial_{\mathbb\Delta}(\alpha)s+\gamma(\alpha) \partial_{M,\mathbb\Delta}(s).
\]
We shall also need the following variant of this formula:
\[
\forall \alpha \in R^\times, \forall s \in M, \quad \partial_{M,\mathbb\Delta}\left(\frac s \alpha\right) = \frac {\alpha \partial_{M,\mathbb\Delta}(s) - \partial_{\mathbb\Delta}(\alpha)s}{\alpha \gamma(\alpha)}.
\]

The condition in definition \ref{qder1} simplifies as follows:

\begin{lem} \label{simpq}
If $M$ is an $R$-module, then a $W$-linear map $D_M : M \to M$ is a $\mathbb\Delta$-derivation with respect to a $q^p$-derivation $D$ of $R$ if and only if
\[
\forall s \in M, \quad D_M(qs) = (p)_qD(q)s+q^{p+1}D_M(s).
\]
\end{lem}

\begin{proof}
By additivity and continuity, the condition of the definition is equivalent to
\[
\forall k \in \mathbb N, \forall s \in M, \quad D_M(q^ks) = (p)_qD(q^k)s+q^{(p+1)k}D_M(s).
\]
Now our assertion is then shown by induction:
\begin{linenomath}
\begin{align*}
D_M(q^{k+1}s) &= (p)_qD(q)q^ks+q^{p+1}D_M(q^ks)
\\ &= (p)_qD(q)q^ks+q^{p+1}((p)_qD(q^k)s+q^{(p+1)k}D_M(s))
\\ &= (p)_qD(q)q^ks+q^{p+1}((k)_{q^p}q^{k-1} (p)_qD(q)s+q^{(p+1)k}D_M(s))
\\ &= (1+q^p (k)_{q^p})q^kD(q)s+q^{(p+1)(k+1)}D_M(s)
\\ &= (k+1)_{q^p}q^k (p)_qD(q)s+q^{(p+1)(k+1)}D_M(s)
\\ & = (p)_qD(q^{k+1})s+q^{(p+1)(k+1)}D_M(s). \qedhere
\end{align*}
\end{linenomath}
\end{proof}

Actually, what we are really interested in is the following:

\begin{dfn} \label{nabR}
An \emph{action by $\mathbb\Delta$-derivations} on an $R$-module $M$ is an $R$-linear map $T_{q^p} \to \mathrm{End}_{W}(M)$ that sends $D \in T_{q^p}$ to a $\mathbb\Delta$-derivation $D_M$ with respect to $D$.
We will then call $M$ a \emph{$\nabla_{\mathbb\Delta}$-module} on $R$.
An $R$-linear map $M \to M'$ between $\nabla_{\mathbb\Delta}$-modules is said to be \emph{horizontal} if it is compatible with the actions.
\end{dfn}

We then obtain a category of $\nabla_{\mathbb\Delta}$-modules on $R$ and we shall simply denote by\footnote{Here and later, we shall simply write $\nabla$ instead of $\nabla_{\mathbb\Delta}$ when no confusion should arise.}  $\mathrm{Hom}_{\nabla}(M,M')$ the set of horizontal maps.

\begin{rmks}
\begin{enumerate}
\item
We may notice that an action by $\mathbb\Delta$-derivations on $M$ is equivalent to an $R$-linear map $T_{\mathbb\Delta} \to \mathrm{End}_{W}(M)$ that sends $(p)_qD$ to a $q^p$-derivation $D_M$ with respect to $(p)_qD$.
\item
One can show that the category of $\nabla_{\mathbb\Delta}$-modules on $R$ has all limits and colimits and they commute with the functor that forgets the action.
\item We could also consider the more restrictive notion of a $q^p$-$\nabla$-module which corresponds to an action by all $q^p$-derivations.
\item We can also consider the notion of (relative) $\nabla$-module which corresponds to an action by usual continuous derivations relative to $W$.
\end{enumerate}
\end{rmks}

Definition \ref{nabR} simplifies a lot:

\begin{prop} \label{acevd}
The category of $\nabla_{\mathbb\Delta}$-modules on $R$ is isomorphic to the category of $R$-modules $M$ endowed with a $W$-linear map $\partial_{M,\mathbb\Delta} : M \to M$ such that
\[
\forall s \in M, \quad \partial_{M,\mathbb\Delta}(qs) = (p)_qs+q^{p+1}\partial_{M,\mathbb\Delta}(s).
\]
\end{prop}

\begin{proof}
This is an immediate consequence of proposition \ref{basder}.1 and lemma \ref{simpq}.
\end{proof}

We define the \emph{cohomology} of a $\nabla_{\mathbb\Delta}$-module $M$ on $R$ as
\[
\mathrm R\Gamma_{\mathrm{dR},\mathbb\Delta}(M) := [M \overset {\partial_{M,\mathbb\Delta}} \longrightarrow M].
\]
Thus, we have $\mathrm H^i_{\mathrm{dR},\mathbb\Delta}(M) = 0$ for $i\neq 0,1$,
\[
\mathrm H^0_{\mathrm{dR},\mathbb\Delta}(M) = \ker \partial_{M,\mathbb\Delta} \quad \mathrm{and} \quad \mathrm H^1_{\mathrm{dR},\mathbb\Delta}(M) = \mathrm{coker} \ \partial_{M,\mathbb\Delta}.
\]

Cohomology is hard to compute in general.
Before running through some examples, let us show the following:

\begin{lem} \label{submod}
Let $M$ be a complete $\nabla_{\mathbb\Delta}$-module on $R$.
Assume $M$ is $(q-1)$-torsion free and  $M/(q-1)M$ is $\mathbb Z$-torsion free.
Then, the canonical map
\[
H^0_{\mathrm{dR},\mathbb\Delta}(M) \to M/(q-1)M
\]
is injective.
\end{lem}

\begin{proof}
We have to show that conditions $\partial_{M,\mathbb\Delta}(s) = 0$ and $s \in (q-1)M$ imply $s = 0$.
Since $M$ is complete, the $(q-1)$-adic filtration is separated, and it is therefore sufficient to show by induction that $s \in (q-1)^nM$ for all $n \in \mathbb N$.
Thus, we are reduced to proving that the induced map
\[
\partial_{M,\mathbb\Delta} : \mathrm {Gr}_{(q-1)}^nM \to \mathrm {Gr}_{(q-1)}^{n-1}M
\]
is injective.
But it follows from the second assertion of proposition \ref{basder} that it is given by
\[
\partial_{M,\mathbb\Delta}((q-1)^ns) = ((p+1)^n -1)(q-1)^{n-1}s \mod (q-1)^n.
\]
Our assertion therefore follows from our hypothesis.
\end{proof}

\begin{xmps}
\begin{enumerate}
\item \textit{Trivial.}
Unless otherwise specified, we always endow $R$ with $\partial_{\mathbb\Delta}$.
If $W$ is $\mathbb Z$-torsion free, then $\mathrm H^0_{\mathrm{dR},\mathbb\Delta}(R) = W$ (use lemma \ref{submod}) and we have a surjective map
\[
\mathrm H^1_{\mathrm{dR},\mathbb\Delta}(R) \twoheadrightarrow \overline R = R/(p)_q
\]
which is not an isomorphism.
There does not seem to exist an easy description of the kernel of this last map in general besides an identification with $H^1_{\mathrm{dR},q^p}(R)$.
\item \textit{Finite free of rank one.}
It is equivalent to give an action by $\mathbb\Delta$-derivations on a free $R$-module $F$ on one generator $s$ or an element $\alpha \in R$: they determine each other through $\partial_{F,\mathbb\Delta}(s) = \alpha s$.
As a specific example, we shall consider the action $\partial_{F_n,\mathbb\Delta}(s) = \alpha_n s$ for $n \in \mathbb N$ with
\[
\alpha_n := \sum_{k=1}^n {n \choose k} (q^2-q)^{k-1} \partial_{q^p}((p)_q)^k.
\]
We shall mainly be concerned with the case $n=1$, so that $\partial_{F_1,\mathbb\Delta}(s) = \partial_{q^p}((p)_q) s$, but also extend later to the case $n < 0$.
It is hard to compute $\mathrm H^0_{\mathrm{dR},\mathbb\Delta}(F)$, not to mention $\mathrm H^1_{\mathrm{dR},\mathbb\Delta}(F)$.
\item \textit{$(p)_q$-adic filtration.}
For $n \in \mathbb N$, the ideal $(p)_q^nR$ of $R$ is stable under the action by $\mathbb\Delta$-derivations (this will follow from the coming computations). 
It defines the \emph{$(p)_q$-adic filtration} on $R$.
Actually, with the notations of the previous example, we have $F_n \simeq (p)_q^nR$ with $s$ corresponding to $(p)_q^n$.
This is proved by induction.
One first notices that
\[
\alpha_{n+1} = \partial_{q^p}((p)_q) + \lambda \alpha_n,
\]
with $\lambda$ as in formula \eqref{pqplun} so that $(p)_{q^{p+1}} = \lambda (p)_q$.
It follows that
\begin{linenomath}
\begin{align*}
\partial_{\mathbb\Delta}((p)_q^{n+1}) &= (p)_q^n\partial_{\mathbb\Delta}((p)_q) +(p)_{q^{p+1}}\partial_{\mathbb\Delta}((p)_q^n)
\\ &= (p)_q^{n+1}\partial_{q^p}((p)_q) + \lambda (p)_q (p)_q^n\alpha_n 
\\ &= (p)_q^{n+1}\alpha_{n+1}.
\end{align*}
\end{linenomath}
\item \textit{Hodge-Tate.}
If $G$ is an $\overline R$-module (with $\overline R = R/(p)_q$), then an action by $\mathbb\Delta$-derivations on $G$ is simply given by an $\overline R$-linear map.
In the case $G = \overline R$ (with the trivial action), this is the zero map so that $\mathrm H^0_{\mathrm{dR},\mathbb\Delta}(\overline R) = \mathrm H^1_{\mathrm{dR},\mathbb\Delta}(\overline R) = \overline R$.
In order to give another interesting family of examples, we denote by $\zeta$ the class of $q$ in $\overline R$ and write
\[
(p)'_\zeta := 1 + 2\zeta + 3\zeta^2 + \cdots + (p-1)\zeta^{p-2} 
\]
(this will reappear in section \ref{generic}.
We may then consider the free $\overline R$-module $G_n$ on one generator $s$ with action given by $\partial_{G_n,\mathbb\Delta}(s) = a_n s$ where
\[
a_n := (n)_{p+1} (p)'_\zeta.
\]
Again, we shall be mainly concerned with the case $n=1$ in which case $\partial_{G_1,\mathbb\Delta}(s) = (p)'_\zeta s$.
\item \textit{$(p)_q$-adic graduation}.
The quotient $(p)_q^nR/(p)_q^{n+1}R$ is stable under the action by $\mathbb\Delta$-derivations.
With the above notations, we have
\[
G_n \simeq \overline F_n \simeq F_n/F_{n+1} \simeq (p)_q^nR/(p)_q^{n+1}R.
\]
Assume $W = W(k)$ is the ring of Witt vectors of a perfect field of positive characteristic $p$.
Then $\mathrm H^0_{\mathrm{dR},\mathbb\Delta}(G_n) = 0$ for all $n >0$ and $\mathrm H^1_{\mathrm{dR},\mathbb\Delta}(G_n) \simeq k[q]/(q-1)^{p-2}$ is a $k$-vector space of dimension $p-2$ when $p \nmid n$.
However, $\mathrm H^1_{\mathrm{dR},\mathbb\Delta}(G_2) \simeq W/4W$ and $\mathrm H^1_{\mathrm{dR},\mathbb\Delta}(G_2) \simeq  W/8W$ when $p=2$.
Also, $\mathrm H^1_{\mathrm{dR},\mathbb\Delta}(G_3) \simeq k^2 \oplus W/9W$ when $p=3$.
Everything becomes more complicated when either $p$ or $n$ grows.
\end{enumerate}
\end{xmps}

The fundamental example however is the \emph{Breuil-Kisin} action:

\begin{lem} \label{BKdef}
If $p$ belongs to the radical of $W$, then there exists a unique action by $\mathbb\Delta$-derivations on the free $R$-module $R\{1\}$ on one generator $e_R$ such that
\[
\partial_{R\{1\},\mathbb\Delta}((q-1)e_R) = \frac pq e_{R}.
\]
\end{lem}

\begin{proof}
The condition reads
\[
(p)_qe_R + (q^{p+1}-1)\partial_{R\{1\},\mathbb\Delta}(e_R) = \frac pq e_{R}
\]
or, equivalently,
\[
(q-1)(p+1)_q\partial_{R\{1\},\mathbb\Delta}(e_R) = \frac {p+1 -(p+1)_q}q e_{R}
\]
It is therefore sufficient to notice that $(p+1)_q \equiv p+1 \mod q-1$ and that $(p+1)_q$ is invertible since we assumed that $p$ belongs to the radical of $W$.
\end{proof}

\begin{rmks}
\begin{enumerate}
\item The following formula is worth mentioning:
\[
\partial_{R\{1\},\mathbb\Delta}(e_R) = \frac 1{q^2-q} \left(\frac {p+1}{(p+1)_q} -1 \right)e_R.
\]
\item In the notations of the examples above, we have $\overline {R\{1\}} \simeq \overline{(p)_qR} \simeq G_1$.
\item From a heuristic point of view, we have $e_R = \log(q)/(q-1)$.
\item One can also introduce the prismatic logarithm $\log_{\mathbb \Delta}(q^p) := (q-1)e_R$ so that
\[
\partial_{R\{1\},\mathbb\Delta}(\log_{\mathbb \Delta}(q^p)) = \frac p{q^2-q} \log_{\mathbb \Delta}(q^p)
\]
This corresponds to construction 3.2.2 of \cite{BhattLurie22} applied to $(R,(p)_q)$.
\end{enumerate}
\end{rmks}

Our terminology comes from the following compatibility with construction 2.2.11 in \cite{BhattLurie22}:

\begin{prop} \label{BKlim}
If $W$ is $p$-adically complete, then there exists an isomorphism of $\nabla_{\mathbb\Delta}$-modules
\[
R\{1\} \simeq \varprojlim_{/p}\ (p^r)_qR/(p^r)_q^2R.
\]
\end{prop}

\begin{proof}
Recall that, by definition, $q^{p^r} - 1= (q-1)(p^r)_q$, wich implies that $q^{p^r} \equiv 1 \mod (p^r)_q$ and then, also, that $(p)_{q^{p^r}} \equiv p\mod (p^r)_q$.
Thus, we see that
\[
(p^{r+1})_q = (pp^r)_q = (p)_{q^{p^r}}(p^r)_q \equiv p(p^r)_q \mod (p^r)_q^2.
\]
It follows that the map
\[
\frac 1p : (p^{r+1})_qR/(p^{r+1})_q^2R \to (p^r)_qR/(p^r)_q^2R, \quad (p^{r+1})_q \mapsto (p^{r})_q
\]
is well defined.
Moreover, we have
\[
\partial_{\mathbb\Delta}((q-1)(p^r)_q) = (p^{r+1})_qq^{p^r-1} \equiv \frac pq (p^r)_q \mod (p^r)_q^2
\]
which shows the compatibility with the formula in lemma \ref{BKdef}.
Finally, since $W$ is assumed to be $p$-adically complete, $R$ is $(p,q-1)$-adically complete and therefore
\[
R\{1\} \simeq \varprojlim R\{1\}/(p^r)_qR\{1\} \simeq \varprojlim_{/p}\ (p^r)_qR/(p^r)_q^2R, \quad e_R \mapsto ((p^r)_q)_{r \in \mathbb N}.\qedhere
\]
\end{proof}

The general notion of a twisted $R$-module is discussed in section 2 of \cite{LeStumQuiros18b} but we shall only need the following:

\begin{dfn}
A \emph{$\gamma$-module} is an $R$-module $M$ endowed with a continuous $\gamma$-linear endomorphism $\gamma_{M}$.
It is said to be \emph{inversive} if its linearization
\[
1 \ {}_{{}_{\gamma}\nwarrow}\!\!\otimes \gamma_M :  \gamma^*M := R\ {}_{{}_{\gamma}\nwarrow}\!\!\otimes_{R} M \to M
\]
is bijective.
\end{dfn}

The following will later help us make the link with Wach modules:

\begin{prop}
\begin{enumerate}
\item
If $M$ is a $\nabla_{\mathbb\Delta}$-module on $R$, then
\[
\gamma_M := \mathrm{Id} + (q^2-q)\partial_{M,\mathbb\Delta}
\]
is a continuous $\gamma$-linear endomorphism of $M$.
\item
This induces an isomorphism between the categories of $(q-1)$-torsion-free $\nabla_{\mathbb\Delta}$-modules and $(q-1)$-torsion-free $\gamma$-modules that become trivial modulo $q-1$.
\end{enumerate}
\label{gammeq}
\end{prop}

\begin{proof}
If $\alpha \in R$ and $s \in M$, then we have
\begin{linenomath}
\begin{align*}
\gamma_M(\alpha s) &= \alpha s + (q^2-q)\partial_{M,q^p}(\alpha s)
\\ &= \alpha s + (q^2-q)\partial_{\mathbb\Delta}(\alpha)s+(q^2-q)\gamma(\alpha)\partial_{M,q^p}(s)
\\ &= \gamma(\alpha)s+(q^2-q)\gamma(\alpha)\partial_{M,q^p}(s)
\\ &= \gamma(\alpha)\gamma_M(s).
\end{align*}
\end{linenomath}
Moreover, $\gamma_M$ is clearly continuous.
Finally, the inverse to our functor is provided by the formula
\[
\partial_{M,\mathbb\Delta} = \frac{\gamma_M - \mathrm{Id}_M}{q^2-q}.\qedhere
\]
\end{proof}

\begin{xmps}
\begin{enumerate}
\item With the above $\nabla_{\mathbb\Delta}$-module $F_n$ corresponding to the $(p)_q$-adic filtration, we get
\[
\gamma_{F_n}(s) = \lambda^ns \quad \left(= \frac {(p)_{q^{p+1}}^n}{(p)_q^n} s \right)
\]
with $\lambda$ as in formula \eqref{pqplun}.
\item With the above $\nabla_{\mathbb\Delta}$-module $G_n$ corresponding to the $(p)_q$-adic graduation, we get
\[
\gamma_{G_n}(s) = (p+1)^ns.
\]
\item For the Breuil-Kisin action, we obtain
\[
\gamma_{R\{1\}}(e_R) = \frac {p+1}{(p+1)_q}e_R \quad \mathrm{and} \quad   \gamma_{R\{1\}}(\log_{\mathbb \Delta}(q^p)) = (p+1)\log_{\mathbb \Delta}(q^p) .
\]
\end{enumerate}
\end{xmps}

\begin{rmks}
\begin{enumerate}
\item The trivial action on $R$ is inversive because the linearization of $\gamma$ is the identiy.
\item More generally, if $M$ is finite projective, then any action by $\mathbb\Delta$-derivations is inversive.
Writing $M$ as a direct summand of a free $R$-module, we may actually assume that $M$ is free.
Given a basis $\underline s$, we can write $\gamma_M(\underline s) = (I + (q^2-q)A)\underline s$ for some matrix $A$ with coefficients in $R$ and the matrix $I + (q^2-q)A$ is invertible.
\item Actually, thanks to derived Nakayama lemma (\cite[\href{https://stacks.math.columbia.edu/tag/0G1U}{Tag 0G1U}]{stacks-project}), it is sufficient to assume that $M$ is derived complete and completely flat (since $\gamma_M$ reduces to the identity modulo $q-1$).
\end{enumerate}
\end{rmks}

\begin{lem}\label{tenhom}
Let $M$ and $N$ be two $\nabla_{\mathbb\Delta}$-modules on $R$.
Then:
\begin{enumerate}
\item The formula
\[
\forall s \in M, t \in N, \quad \partial_{M\otimes N,\mathbb\Delta}(s \otimes t) = \partial_{M,\mathbb\Delta}(s) \otimes t + \gamma_M(s) \otimes \partial_{N,\mathbb\Delta}(t)
\]
defines an action by $\mathbb\Delta$-derivations on $M\otimes_R N$.
\item If $M$ is inversive, then there exists a unique action by $\mathbb\Delta$-derivations on $\mathrm{Hom}_R(M,N)$ such that
\[
\forall f :M\to N, \quad \partial_{\mathrm{Hom}(M,N),\mathbb\Delta}(f) \circ \gamma_M = \partial_{N,\mathbb\Delta} \circ f -f \circ \partial_{M,\mathbb\Delta}.
\]
\end{enumerate}
\end{lem}

\begin{proof}
We will not mention the modules as indices in this proof.

For the first assertion, we start with the following computation for $\alpha \in R$:
\begin{linenomath}
\begin{align*}
\partial_{\mathbb\Delta}(s \otimes \alpha t) &= \partial_{\mathbb\Delta}(s) \otimes \alpha t + \gamma(s) \otimes \partial_{\mathbb\Delta}(\alpha t)
\\ &= \partial_{\mathbb\Delta}(s) \otimes \alpha t + \gamma(s) \otimes \partial_{\mathbb\Delta}(\alpha) t + \gamma(s) \otimes \gamma(\alpha) \partial_{\mathbb\Delta}(t)
\\ &= (\partial_{\mathbb\Delta}(\alpha) \gamma(s) + \alpha \partial_{\mathbb\Delta}(s)) \otimes t + \gamma(\alpha) \gamma(s) \otimes \partial_{\mathbb\Delta}(t)
\\ &= \partial_{\mathbb\Delta}(\alpha s) \otimes t + \gamma(\alpha s) \otimes \partial_{\mathbb\Delta}(t).
\end{align*}
\end{linenomath}
This shows that the formula is well defined and we now show that this is indeed a $\mathbb\Delta$-derivation:
\begin{linenomath}
\begin{align*}
\partial_{\mathbb\Delta}(\alpha s \otimes t) &= \partial_{\mathbb\Delta}(\alpha s) \otimes t + \gamma(\alpha s) \otimes \partial_{\mathbb\Delta}(t)
\\ &= (\partial_{\mathbb\Delta}(\alpha)s+\gamma(\alpha)\partial_{\mathbb\Delta}(s)) \otimes t + \gamma(\alpha) \gamma(s) \otimes \partial_{\mathbb\Delta}(t)
\\ &= \partial_{\mathbb\Delta}(\alpha)s\otimes t +\gamma(\alpha)(\partial_{\mathbb\Delta}(s) \otimes t + \gamma(s) \otimes \partial_{\mathbb\Delta}(t))
\\ &= \partial_{\mathbb\Delta}(\alpha)s\otimes t +\gamma(\alpha)\partial_{\mathbb\Delta}(s \otimes t).
\end{align*}
\end{linenomath}

The second assertion requires more care.
Uniqueness is clear since $M$ is assumed to be inversive and we want to show that
\[
\partial_{\mathbb\Delta}(\alpha f) = \partial_{\mathbb\Delta}(\alpha) f + \gamma(\alpha) \partial_{\mathbb\Delta}(f)
\]
or, after composing on the right with $\gamma$, that
\[
\partial_{\mathbb\Delta} \circ \alpha f -\alpha f \circ \partial_{\mathbb\Delta} = \partial_{\mathbb\Delta}(\alpha) f \circ \gamma + \gamma(\alpha)\partial_{\mathbb\Delta} \circ f - \gamma(\alpha)f \circ \partial_{\mathbb\Delta}.
\]
Since $\partial_{\mathbb\Delta} \circ \alpha = \partial_{\mathbb\Delta}(\alpha) + \gamma(\alpha)\partial_{\mathbb\Delta}$, this is equivalent to
\[
\partial_{\mathbb\Delta}(\alpha)f -\alpha f \circ \partial_{\mathbb\Delta} = \partial_{\mathbb\Delta}(\alpha) f \circ \gamma - \gamma(\alpha)f \circ \partial_{\mathbb\Delta}
\]
and it is therefore sufficient (after simplification by $f$ on the left) to check that
\[
\partial_{\mathbb\Delta}(\alpha) -\alpha \partial_{\mathbb\Delta} = \partial_{\mathbb\Delta}(\alpha) \gamma - \gamma(\alpha) \partial_{\mathbb\Delta}.
\]
This follows from the equality
\[
\partial_{\mathbb\Delta}(\alpha) (\gamma - \mathrm{Id}) = (\gamma - \mathrm{Id})(\alpha)\partial_{\mathbb\Delta}
\]
which is readily seen to hold since $\gamma - \mathrm{Id} = (q^2-q)\partial_{\mathbb\Delta}$.
\end{proof}

The tensor product and Hom structures are compatible in various ways and we shall need in particular the following:

\begin{lem} \label{homtens}
If $M_1, M_2, M'_1, M'_2$ are four $\nabla_{\mathbb\Delta}$-modules on $R$ and $M_1, M_2$ are inversive, then the natural map
\[
\mathrm{Hom}_R(M_1, M'_1) \otimes_R \mathrm{Hom}_R(M_2, M'_2) \to \mathrm{Hom}(M_1 \otimes_R M_2, M'_1 \otimes_R M'_2)
\]
is horizontal.
\end{lem}

\begin{proof}
This is readily checked.
It is actually sufficient to show the analogous assertion for the action of $\gamma$ on the various $R$-modules.
This makes the computations easier (see remark \ref{tensor}.1 below).
\end{proof}

\begin{rmks} \phantomsection \label{tensor}
\begin{enumerate}
\item We will have
\[
\forall s \in M, t \in N, \quad \gamma_{M\otimes N}(s \otimes t) = \gamma_M(s) \otimes \gamma_{N}(t)
\]
\[
\left(\mathrm{resp}.\quad \forall f :M\to N, \quad \gamma_{\mathrm{Hom}(M,N)}(f) \circ \gamma_M = \gamma_{N} \circ f \right).
\]
\item
We can also consider for all $n \in \mathbb N$ the \emph{tensor power} $M^{\otimes n}$ of a $\nabla_{\mathbb\Delta}$-module $M$.
In the case of a free module $F$ of rank one with generator $s$ and action given by $\partial_{F,\mathbb\Delta}(s) = \alpha s$, one can check that
\[
\partial_{F^{\otimes n},\mathbb\Delta}(s^{\otimes n}) = \left( \sum_{k=1}^n {n \choose k} (q^2-q)^{k-1}\alpha^k\right)s^{\otimes n}
\]
and
\[
\gamma_{F^{\otimes n}}(s^{\otimes n}) = \beta^n s^{\otimes n} \quad \mathrm{with} \quad \beta := 1 + (q^2-q)\alpha.
\]
\item
We may also consider the \emph{dual} $\check{M}$ of $M$ when $M$ is an inversive $\nabla_{\mathbb\Delta}$-module.
With $F$ as before, we find
\[
\partial_{\check{F},\mathbb\Delta}(\check s) = - \frac \alpha{1 + (q^2-q)\alpha} \check s
\quad \mathrm{and} \quad
\gamma_{\check{F}}(\check s) = \frac 1{\beta}\check s.
\]
\item 
One can also generalize the definition of $M^{\otimes n}$ to all $n \in \mathbb Z$ by setting $M^{\otimes (-n)} := \check M^{\otimes n}$.
\item If $M$ and $M'$ are two $\nabla_{\mathbb\Delta}$-modules with $M$ inversive, then
\[
\mathrm H^0_{\mathrm{dR},\mathbb\Delta}(\mathrm{Hom}_{R}(M,M')) = \mathrm{Hom}_{\nabla}(M,M').
\]
\end{enumerate}
\end{rmks}

\begin{xmps}
\begin{enumerate}
\item With $n \in \mathbb N$ and $F_n$ as above corresponding to the $(p)_q$-adic filtration, we have $F_n \simeq F_1^{\otimes n}$.
If we set $F_{-n} := \check{F_n}$, then, in the case $n=1$, we have
\[
\partial_{F_{-1},\mathbb\Delta}(\check s) = - \frac 1\lambda \partial_{q^p}(p)_q \check s \quad \left(= - \frac {\partial_{\mathbb\Delta}((p)_q)}{(p)_{q^{p+1}}} \check s\right)
\]
with $\lambda$ as in formula \eqref{pqplun}.
Note that $F_{-n} \simeq (p)_q^{-r}R := \mathrm{Hom}_R((p)_q^rR, R)$.
\item If, for $n \in \mathbb Z$, we set $R\{n\} := R\{1\}^{\otimes n}$, then we will have
\[
\partial_{R\{n\},\mathbb\Delta}(e_R^{\otimes n}) = \frac 1{q^2-q} \left(\frac {(p+1)^n}{(p+1)_q^n} -1 \right)e_R^{\otimes n}
\]
and
\[
\gamma_{R\{n\}}(e_R^{\otimes n}) = \frac {(p+1)^n}{(p+1)_q^n} e_R^{\otimes n}.
\]
\item Since, for all $n \in \mathbb N$, $G_n$ is nothing but $R\{n\}$ (or $(p)_q^nR$) modulo $(p)_q$, we see that, if we set $G_{-n} := \check{G_n}$, then for all $n \in \mathbb Z$,
\[
\partial_{G_n,\mathbb\Delta}(s^{\otimes n}) = (n)_{p+1} (p)'_\zeta s^{\otimes n}
\]
and
\[
\gamma_{G_n}(s^{\otimes n}) = (p+1)^n s^{\otimes n}.
\]
\item
Even if computing cohomology is hard in general, we can do it for $\overline {R\{n\}}/p\overline {R\{n\}}$ ($\simeq G_n/pG_n$), with $n \in \mathbb Z$, when $W$ is the ring of Witt vectors of a perfect field $k$ of positive characteristic $p$. We then have
\[
(n)_{p+1} \equiv 0 \mod p \Leftrightarrow p \mid n \quad \mathrm{	and} \quad (p)'_\zeta \equiv (\zeta -1)^{p-2} \mod p
\]
 (compare with construction 2.4 in \cite{BhattMathew23}),
\[
\dim_k \mathrm H^0_{\mathrm{dR},\mathbb\Delta}(\overline R\{n\}/p\overline R\{n\}) =\dim_k \mathrm H^1_{\mathrm{dR},\mathbb\Delta}(\overline R\{n\}/p\overline R\{n\}) = p-1
\]
when $p \mid n$ and
\[
\dim_k \mathrm H^0_{\mathrm{dR},\mathbb\Delta}(\overline R\{n\}/p\overline R\{n\}) = 1,\quad \dim_k \mathrm H^1_{\mathrm{dR},\mathbb\Delta}(\overline R\{n\}/p\overline R\{n\}) = p-2
\]
when $p \nmid n$.
\end{enumerate}
\end{xmps}

We can twist a $\nabla_{\mathbb\Delta}$-module in various ways:

\begin{dfn}
If $M$ is is a $\nabla_{\mathbb\Delta}$-module on $R$ and $n \in \mathbb Z$, then
\begin{enumerate}
\item the $n$th \emph{$(p)_q$-adic twist} of $M$ is $(p)_q^nM$,
\item the $n$th \emph{Breuil-Kisin twist} of $M$ is $M\{n\} := M \otimes_R R\{n\}$.
\end{enumerate}
\end{dfn}

Note that they coincide modulo $(p)_q$.

We shall need later the following technical result:

\begin{lem} \label{injmod}
If $M$ and $N$ are two $\nabla_{\mathbb\Delta}$-modules with $M$ inversive and $N$ complete, $(q-1)$-torsion free with $N/(q-1)N$ $\mathbb Z$-torsion free, then the canonical map
\[
\mathrm{Hom}_\nabla(M, N) \to \mathrm{Hom}_W(M/(q-1)M, N/(q-1)N)
\]
is injective.
\end{lem}

\begin{proof}
Apply lemma \ref{submod} to $\mathrm{Hom}(M, N)$.
\end{proof}

Our theory is clearly functorial in $W$ but there also exists some functoriality in $R$:
\begin{prop} \label{foncr}
Let $M$ be a $\nabla_{\mathbb\Delta}$-module on $R$ and $M_r := g_r^*M$ its pullback under a morphism $g_r$ (given by $g_r(q) = q^r)$.
Then, there exists a unique action by $\mathbb\Delta$-derivations on $M_r$ such that
\[
\forall s \in M, \quad \partial_{M_r,\mathbb\Delta}(1 \otimes s) = (r)_qq^{r-1} \otimes \partial_{M,\mathbb\Delta}(s).
\]
\end{prop}

\begin{proof}
Again, we will not mention the modules as indices in the proof.
Uniqueness is clear and we will then have
\[
\forall s \in M, \quad \partial_{\mathbb\Delta}(r \otimes s) = \partial_{\mathbb\Delta}(r) \otimes s + \gamma(r)(r)_qq^{r-1} \otimes \partial_{\mathbb\Delta}(s).
\]
This is clearly a $\mathbb\Delta$-derivation if it is well defined: we will have on the one hand
\[
\partial_{\mathbb\Delta}(q^r \otimes s) = (p)_q(r)_{q^p}q^{r-1} \otimes s + q^{pr+r}(r)_qq^{r-1} \otimes \partial_{\mathbb\Delta}(s),
\]
and on the other hand
\begin{linenomath}
\begin{align*}
\partial_{\mathbb\Delta}(1\otimes q s)&=(r)_qq^{r-1} \otimes \partial_{\mathbb\Delta}(q s)
\\ &= (r)_qq^{r-1} \otimes ((p)_q s + q^{p+1} \partial_{\mathbb\Delta}(s))
\\ &= (r)_qq^{r-1}(p)_{q^r} \otimes s+ (r)_qq^{r-1} q^{pr+r} \otimes \partial_{\mathbb\Delta}(s).
\end{align*}
\end{linenomath}
The assertion therefore follows from the identity $ (p)_q(r)_{q^p} = (pr)_q = (r)_q(p)_{q^r}$.
\end{proof}

\begin{rmks}
\begin{enumerate}
\item The statement holds for any $r \in \mathbb Z$: the notion of $q$-analog of a negative integer makes sense whenever $q$ is invertible (see for example definition 1.1 in \cite{LeStumQuiros15}) and all usual formulas still hold.
\item Actually, when $p$ is topologically nilpotent in $W$, the statement extends to any $r \in \mathbb Z_p$.
\item
We will simply have $\gamma_{M_r}(1 \otimes s) = 1 \otimes \gamma_M(s)$.
\item
It makes sense to call a $g_r$-linear morphism $M \to M$ \emph{horizontal} whenever the corresponding $R$-linear morphism $M_r \to M$ is. 
\end{enumerate}
\end{rmks}

\section{Higher $\mathbb\Delta$-derivatives} \label{higherd}

In order to have a full theory, it will not be sufficient to consider higher powers of $\mathbb\Delta$-derivations, and we shall introduce higher $\mathbb\Delta$-derivatives as a byproduct of $\mathbb\Delta$-Taylor maps.

As before, $W$ denotes a complete adic ring and $R := W[[q-1]]$ where $q$ is an indeterminate.
We always use the $\mathfrak m_W+(q-1)$-adic topology where $\mathfrak m_W$ denotes an ideal of definition of $W$.
We fix some positive integer $p$ and we denote by $\xi$ and $\omega$ two indeterminates over $R$ (thinking of $\omega$ as $\xi/(p)_q$).

There exists a unique continuous isomorphism of $W$-algebras
\begin{equation} \label{isoq}
\xymatrix@R=0cm{
R \widehat \otimes_{W} R \simeq W[[q_{0}-1, q_{1}-1]] \ar[r]^-\simeq & R[[\xi]] \simeq W[[q-1,\xi]] 
\\q_0 \ar@{|->}[r] & q
\\ q_1 \ar@{|->}[r] & q + \xi.
}
\end{equation}

\begin{dfn} \label{Taylfinf}
The \emph{Taylor map (of infinite level)}
\[
\theta^{(\infty)} : R \to R[[\xi]]
\]
is the composition of the \emph{stupid Taylor map}
\[
p_2 : R \to R \widehat \otimes_{W} R, \quad \alpha \mapsto 1 \otimes \alpha
\]
with the isomorphism \eqref{isoq}.
\end{dfn}

In other words, $\theta^{(\infty)}$ is the unique continuous $W$-linear map such that $q \mapsto q + \xi$.
By symmetry, we shall denote by $\iota^{(\infty)} : R \to R[[\xi]]$ the canonical map which is the composite of
\[
p_1 : R \to R \widehat \otimes_{W} R, \quad \alpha \mapsto \alpha \otimes 1
\]
with the isomorphism \eqref{isoq}.
This is also the unique continuous $W$-linear map such that $q \mapsto q$.
We may then call \emph{action on the right} the action of $R$ through the Taylor map by contrast with the action through the canonical map, that we may then refer to as the \emph{action on the left}. 
This applies also to all canonical/Taylor maps that we shall encounter in the future.
When there is no risk of confusion, we will simply denote by $\iota$/$\theta$ \emph{any} canonical/Taylor map (and there will be many of them).
However, we may also write $\iota_R$/$\theta_R$ when other rings (or modules) are involved.
\begin{dfn} \label{fldef}
The \emph{flip map} on $R[[\xi]]$ is the unique semilinear (with respect to $\theta^{(\infty)}$) $R$-algebra automorphism
\[
\tau : R[[\xi]] \overset \sim \longrightarrow R[[\xi]], \quad \xi \mapsto -\xi.
\]
\end{dfn}

The flip map is a symmetry ($\tau \circ \tau = \mathrm{Id}$) that exchanges the left and right $R$-algebra structures of $R[[\xi]]$.
It corresponds to the continuous $W$-linear automorphism
\[
W[[q_0-1, q_1-1]] \overset \simeq \longrightarrow W[[q_0-1, q_1-1]], \quad q_0 \mapsto q_1, q_1 \mapsto q_0,
\]
or maybe better, to
\[
R \widehat \otimes_{W} R \overset \simeq \longrightarrow R \widehat \otimes_{W} R, \quad f \otimes g \mapsto g \otimes f.
\]

\begin{rmks} \phantomsection \label{classd}
\begin{enumerate}
\item Notice that, since $\tau$ exchanges $\theta$ and $\iota$, we have $\iota= \tau \circ \theta$ and $\theta= \tau \circ \iota$. Therefore, for $\alpha \in R$ we get $\alpha=\tau(\theta(\alpha))$ and $\theta(\alpha)=\tau(\alpha)$.
\item We shall freely use all terminology and notation from usual calculus (taking into account the topology of $R$).
We can consider for example the module of continuous differentials $\Omega_{R/W}$ (or $\Omega$ for short).
This is the kernel of the augmentation map $R[\xi]/\xi^2 \to R, \xi \mapsto 0$, and it comes with the universal derivation $\mathrm d : R \to \Omega$ induced by the difference $\theta_R^{(\infty)} - \iota_R^{(\infty)}$.
Of course, $\Omega$ is a free $R$-module of rank one with generator $\mathrm dq$.
One may then define a connection and so on.
\item If we compose the Taylor map $\theta^{(\infty)}$ with the blow up
\[
R[[\xi]] \to R[[\omega]], \quad \xi \mapsto (p)_q\omega,
\]
then we recover log-calculus with respect to the divisor $(p)_q$.
In particular, we can consider the kernel $\Omega_{\log}$ of the augmentation map $R[\omega]/\omega^2 \to R$ as well as log-connections.
Be careful however that the coordinate $q$ does not match the divisor $(p)_q$.
In particular, $\mathrm {dlog}(p)_q$ is not a generator of $\Omega_{\log}$ and we must use ``$\mathrm dq/(p)_q$'' instead.
We shall develop a variant of this theory in section \ref{generic}.
\end{enumerate}
\end{rmks}

We are actually interested in \emph{twisted} calculus and we recall from remark \ref{twpow}.3 and definition \ref{twistom} that we write\footnote{For a while, we shall do the $q^p$-case in parallel with the $\mathbb\Delta$-case even if only the latter will matter in the end.}
\[
\xi^{(n)_{q^p}} := \prod_{k=0}^{n-1}\left(\xi - q^{pk+1} + q\right) \quad \mathrm{and} \quad \omega^{(n)_{\mathbb\Delta}} := \prod_{k=0}^{n-1}\left(\omega - (k)_{q^p}(q^2-q)\right).
\]
As we noticed in remark \ref{twpow}.4, blowing up $\xi \mapsto (p)_q\omega$ sends $\xi^{(n)_{q^p}}$ to $(p)_q^n\omega^{(n)_{\mathbb\Delta}}$.
Let us recall that, according to lemma \ref{qlim}, we have
\[
R[[\xi]] \simeq \varprojlim R[\xi]/\xi^{(n+1)_{q^p}}.
\]

\begin{dfn} \label{Taylor}
The \emph{$q^p$-Taylor map} (resp.\ \emph{$\mathbb\Delta$-Taylor map}) of infinite level and finite order $n \in \mathbb N$ is the composite
\[
\theta^{(\infty)}_{n,q^p} : R \overset {\theta^{(\infty)}}\longrightarrow R[[\xi]] \twoheadrightarrow R[\xi]/\xi^{(n+1)_{q^p}}
\]
\[
\left(\mathrm{resp.} \quad \theta^{(\infty)}_{n,\mathbb\Delta} : R \overset {\theta^{(\infty)}_{n,q^p}}\longrightarrow R[\xi]/\xi^{(n+1)_{q^p}} \longrightarrow R[\omega]/\omega^{(n+1)_{\mathbb\Delta}}\right).
\]
\end{dfn}

This Taylor map sends $q$ to $q + \xi \mod \xi^{(n+1)_{q^p}}$ (resp.\ to $q + (p)_q\omega \mod \omega^{(n+1)_{\mathbb\Delta}}$).
By symmetry, we may also write $\iota^{(\infty)}_{n,q^p}$ and $\iota^{(\infty)}_{n,\mathbb\Delta}$ for the canonical maps.

We can now introduce \emph{higher $q^p$-derivatives} (resp.\ \emph{higher $\mathbb\Delta$-derivatives}) of infinite level:

\begin{lem} \label{tetaqn}
For each $k \in \mathbb N$, there exists a unique continuous $W$-linear map $\partial^{[k]}_{q^p} : R \to R$ (resp.\ $\partial^{[k]}_{\mathbb\Delta} : R \to R$) such that for all $n\in \mathbb N$,
\begin{equation} \label{dform}
\partial^{[k]}_{q^p}(q^n) = {n \choose k}_{q^p}q^{n-k} \quad \left(\mathrm{resp.} \quad \partial^{[k]}_{\mathbb\Delta}(q^n) = (p)_q^k{n \choose k}_{q^p}q^{n-k}\right)
\end{equation}
and we have for all $n \in \mathbb N$ and $\alpha \in R$,
\begin{equation} \label{tform}
\theta_{n,q^p}^{(\infty)}(\alpha) = \sum_{k=0}^n \partial^{[k]}_{q^p}(\alpha) \xi^{(k)_{q^p}} \quad \left(\mathrm{resp.} \quad \theta_{n,\mathbb\Delta}^{(\infty)}(\alpha) = \sum_{k=0}^n \partial^{[k]}_{\mathbb\Delta}(\alpha) \omega^{(k)_{\mathbb\Delta}}\right).
\end{equation}
\end{lem}

\begin{proof}
Uniqueness is clear.
Moreover, existence of continuous $W$-linear maps satisfying \eqref{tform} follows from the definitions of $\theta_{n,q^p}^{(\infty)}$ and $\theta_{n,\mathbb\Delta}^{(\infty)}$.
Finally, for $n \leq N$, we have
\[
\theta_{N,q^p}^{(\infty)}(q^n) = (q+\xi)^n = \sum_{k=0}^{n} {n \choose k}_{q^p}q^{n-k} \xi^{(k)_{q^p}}.
\]
This provides the formulas in \eqref{dform}.
\end{proof}

Note that these notations are compatible with those of section \ref{secabs} in the sense that
\[
\partial^{[0]}_{q^p} = \partial^{[0]}_{\mathbb\Delta} = \mathrm{Id}_R, \quad \partial^{[1]}_{q^p} = \partial_{q^p} \quad \mathrm{and} \quad \partial^{[1]}_{\mathbb\Delta} = \partial_{\mathbb\Delta}.
\]
Also, for all $k \in \mathbb N$, $\partial^{[k]}_{\mathbb\Delta} = (p)_q^k\partial^{[k]}_{q^p}$ (so one is a logarithmic version of the other).

We now recall the notion of module of $\gamma$-differentials from definition 2.1 of \cite{LeStumQuiros18} that we simply call here module of $q^p$-differentials:

\begin{dfn}
The \emph{module of $q^p$-differentials} of $R/W$ is
\[
\Omega_{R/W,q^p} := R[\xi]\xi/R[\xi]\xi^{(2)_{q^p}}
\]
(or $\Omega_{q^p}$ for short).
\end{dfn}

In other words, $\Omega_{q^p}$ is the kernel of the augmentation map $R[\xi]/\xi^{(2)_{q^p}} \to R$.
We may then consider the map $\mathrm d_{q^p} : R \to \Omega_{q^p}$ induced by difference
\[
\theta^{(\infty)}_{1,q^p} - \iota^{(\infty)}_{1,q^p} : R \mapsto R[\xi]/\xi^{(2)_{q^p}}, \quad q \mapsto \xi \mod \xi^{(2)_{q^p}}.
\]
Note that $\Omega_{q^p}$ is a free $R$-module on $\mathrm d_{q^p}q$ and
\begin{equation} \label{unider}
\forall \alpha \in R, \quad \mathrm d_{q^p}(\alpha) = \partial_{q^p}(\alpha)\mathrm d_{q^p}q.
\end{equation}

We can give a shorter proof of proposition 2.4 in \cite{LeStumQuiros18} in this setting:

\begin{lem} \label{dOmeg}
The map $\mathrm d_{q^p} : R \to \Omega_{q^p}$ is universal for $q^p$-derivations.
\end{lem}

\begin{proof}
First of all, formula \eqref{unider} shows that $\mathrm d_{q^p}$ is a $q^p$-derivation.
On the other hand, it follows from lemma \ref{dersimp} that a $q^p$-derivation $D : R \to M$ is uniquely determined by $D(q)$.
\end{proof}

As a consequence, we see that $T_{q^p}$ and $\Omega_{q^p}$ are naturally dual to each other.

\begin{dfn} \label{qconp}
A \emph{$\mathbb\Delta$-connection} on an $R$-module $M$ is a $W$-linear map $\nabla_M : M \to M \otimes_R \Omega_{q^p}$ such that
\[
\forall \alpha \in R, \forall s \in M, \quad \nabla_M(\alpha s) = (p)_q\, s\otimes \mathrm d_{q^p}(\alpha) +\gamma(\alpha) \nabla_M(s).
\]
\end{dfn}

\begin{prop} \label{miceq}
The category of $\nabla_{\mathbb\Delta}$-modules on $R$ is isomorphic to the category of $R$-modules endowed with a $\mathbb\Delta$-connection.
\end{prop}

\begin{proof}
This follows from proposition \ref{acevd} since a $\mathbb\Delta$-connection $\nabla_M$ is equivalent to a $\mathbb\Delta$-derivation $\partial_{M,\mathbb\Delta}$ with respect to $\partial_{q^p}$ via
\[
\nabla_M(s) = \partial_{M,\mathbb\Delta}(s) \otimes \mathrm d_{q^p}q. \qedhere
\]
\end{proof}

According to lemma \ref{simpq}, the condition of definition \ref{qconp} boils down to
\[
\forall s \in M, \quad \nabla_M(qs) = s\otimes \mathrm (p)_q\mathrm d_{q^p}q +q^{p+1}\nabla_M(s).
\]
Note also that
\[
\mathrm R\Gamma_{\mathrm{dR},\mathbb\Delta}(M) \simeq [M \overset {\nabla_M} \longrightarrow M \otimes_R \Omega_{q^p}]
\]
is the (cohomology of the) de Rham complex -- as it should.

We shall also consider a logarithmic version of $\Omega_{q^p}$:

\begin{dfn}
The \emph{module of $\mathbb\Delta$-differentials} of $R/W$ is
\[
\Omega_{R/W,\mathbb\Delta} = R[\omega]\omega/R[\omega]\omega^{(2)_{\mathbb\Delta}}
\]
(or $\Omega_{\mathbb\Delta}$ for short).
\end{dfn}

Again, this is the kernel of the augmentation map $R[\omega]/\omega^{(2)_{\mathbb\Delta}} \to R$.
There exists a canonical injective map $\Omega_{q^p} \hookrightarrow \Omega_{\mathbb\Delta}$ sending the class of $\xi$ to the class of $(p)_q\omega$ and the $q^p$-derivation
\[
R \overset {d_{q^p}} \longrightarrow \Omega_{q^p} \hookrightarrow \Omega_{\mathbb\Delta}
\]
(that we shall still denote by $\mathrm d_{q^p}$) is identical to the map induced by $\theta^{(\infty)}_{1,\mathbb\Delta} - \iota^{(\infty)}_{1,\mathbb\Delta}$.
On the other hand, we may consider the \emph{Bockstein\footnote{There is a twist here in the sense that the natural isomorphism is $\Omega_{q^p} \simeq (p)_q\Omega_{\mathbb\Delta}$} isomorphism} $B : \Omega_{q^p} \simeq \Omega_{\mathbb\Delta}$ sending the class of $\xi$ to the class of $\omega$, as well as and the composite map
\[
\mathrm d_{\mathbb\Delta} : R \overset {d_{q^p}} \longrightarrow \Omega_{q^p} \overset B \simeq \Omega_{\mathbb\Delta},
\]
so that $\mathrm d_{q^p} = (p)_q\mathrm d_{\mathbb\Delta}$.
In particular, $ \Omega_{\mathbb\Delta}$ is a free $R$-module of rank one on $\mathrm d_{\mathbb\Delta}q$.

\begin{prop}
The category of $\nabla_{\mathbb\Delta}$-modules on $R$ is isomorphic to the category of $R$-modules $M$ endowed with a $W$-linear map $\nabla_M : M \to M \otimes_R \Omega_{\mathbb\Delta}$ such that
\[
\forall \alpha \in R, \forall s \in M, \quad \nabla_M(\alpha s) = s\otimes \mathrm d_{\mathbb\Delta}(\alpha) +\gamma(\alpha) \nabla_M(s).
\]
\end{prop}

\begin{proof}
According to proposition \ref{miceq}, the equivalence is given by
\[
\nabla_M(s) = \partial_{M,\mathbb\Delta}(s) \otimes \mathrm d_{\mathbb\Delta}q. \qedhere
\]
\end{proof}

Be careful that this equivalence is not so natural in the sense that we use Bockstein isomorphism.

There exists yet another equivalent notion that needs to be introduced:

\begin{dfn}
A \emph{$\mathbb\Delta$-Taylor map of order $1$} on an $R$-module $M$ is a semilinear section ($R$-linear for the right structure)
\[
\theta_{M,1} : M \to M \otimes_R R[\omega]/\omega^{(2)_{\mathbb\Delta}}
\]
of the augmentation map.
\end{dfn}

The category of $R$-modules endowed with a $\mathbb\Delta$-Taylor map of order $1$ is defined in the obvious way.
The next statement is the first step in the process of lifting a $\mathbb\Delta$-derivation to what we shall call a $\mathbb\Delta$-Taylor map of infinite order:

\begin{lem} \label{Taylcon}
The category of $\nabla_{\mathbb\Delta}$-modules on $R$ is isomorphic to the category of $R$-modules endowed with a $\mathbb\Delta$-Taylor map of order $1$.
\end{lem}

\begin{proof}
We shall rely on proposition \ref{miceq}.
It is equivalent to give a $W$-linear map $\nabla_M : M \to M \otimes_R \Omega_{\mathbb\Delta}$ or a $W$-linear section $\theta_{M,1} : M \to M \otimes_R R[\omega]/\omega^{(2)_{\mathbb\Delta}}$ of the augmentation map.
This equivalence is given by
\[
\forall s \in M, \quad \theta_{M,1}(s) = s \otimes 1 + \nabla_M(s).
\]
Moreover, $\nabla_M$ defines a $\mathbb\Delta$-connection if and only if $\theta_{M,1}$ is semilinear.
More precisely, since $\theta_{R}$ acts as $\gamma$ on $ \Omega_{\mathbb\Delta}$ (see lemma \ref{congmod} below for example), we have for $\alpha \in R$ and $s \in M$,
\begin{linenomath}
\begin{align*}
\theta_{M,1}(\alpha s) = \theta_{R}(\alpha) \theta_{M,1}(s) &\Leftrightarrow \alpha s \otimes 1 + \nabla_M(\alpha s) = \theta_{R}(\alpha) (s \otimes 1 + \nabla_M(s))
\\ &\Leftrightarrow \nabla_M(\alpha s) = s \otimes \theta_{R}(\alpha) - s \otimes \alpha + \gamma(\alpha) \nabla_M(s)
\\ &\Leftrightarrow \nabla_M(\alpha s) = s \otimes \mathrm d_{q^p}(\alpha) + \gamma (\alpha) \nabla_M(s). \qedhere
\end{align*}
\end{linenomath}
\end{proof}

To summarize, it is equivalent to give an action on an $R$-module $M$ by $\mathbb\Delta$-derivations, a $\mathbb\Delta$-connection (with both flavors) or a $\mathbb\Delta$-Taylor map of order $1$.
Moreover, they all amount to giving a continuous $W$-linear map $\partial_{M,\mathbb\Delta} : M \to M$ such that
\[
\forall s \in M, \quad \partial_{M,\mathbb\Delta}(qs) = (p)_q\, s+q^{p+1}\partial_{M,\mathbb\Delta}(s).
\]
Moreover, these equivalences are compatible with the corresponding cohomology theories.

Recall that we introduced in remark \ref{divpow}.4 and definition \ref{defdtw} the twisted divided polynomial rings $R\langle \xi \rangle_{q^p}$ and $R\langle \omega \rangle_{\mathbb\Delta}$ and from remark \ref{divpow}.5 that there exists again a blowing up map
\[
R\langle \xi \rangle_{q^p} \to R\langle \omega \rangle_{\mathbb\Delta}, \quad \xi^{[n]_{q^p}} \mapsto (p)_q^n\omega^{\{n\}_{\mathbb\Delta}}.
\]

\begin{dfn} \label{finTayl}
The \emph{$q^p$-Taylor map (resp.\ the $\mathbb\Delta$-Taylor map) of \emph{order $n$}} on $R$ is the composite map
\[
\theta_{n,q^p} : R \overset{\theta^{(\infty)}_{n,q^p}}\longrightarrow R[\xi]/\xi^{(n+1)_{q^p}} \to R\langle \xi \rangle_{q^p}/\xi^{[>n]_{q^p}}
\]
\[
\left(\mathrm{resp.} \quad \theta_{n,\mathbb\Delta} : R \overset{\theta_{n,q^p}}\longrightarrow R\langle \xi \rangle_{q^p}/\xi^{[>n]_{q^p}} \to R\langle \omega \rangle_{\mathbb\Delta}/\omega^{\{>n\}_{\mathbb\Delta}})\right).
\]
\end{dfn}

This is again the unique continuous $W$-linear map such that $q \mapsto q + \xi \mod \xi^{[>n]_{q^p}}$ (resp.\ $q \mapsto q + (p)_q\omega \mod \omega^{\{>n\}_{\mathbb\Delta}}$).

There exists a statement similar to lemma \ref{tetaqn}: 

\begin{prop} \label{tetaqn2}
For each $k \in \mathbb N$, there exists a unique continuous $W$-linear map $\partial^{\langle k \rangle}_{q^p}: R \to R$ (resp.\ $\partial^{\langle k \rangle}_{\mathbb\Delta} : R \to R$) such that for all $n\in \mathbb N$,
\[
\partial^{\langle k \rangle}_{q^p}(q^n) = (k)_{q^p}!{n \choose k}_{q^p}q^{n-k} \quad \left(\mathrm{resp.} \quad \partial^{\langle k \rangle}_{\mathbb\Delta}(q^n) = (p)_q^k(k)_{q^p}!{n \choose k}_{q^p}q^{n-k}\right)
\]
and we have for all $n \in \mathbb N$ and $\alpha \in R$,
\[
\theta_{n,q^p}(\alpha) = \sum_{k=0}^n \partial^{\langle k \rangle}_{q^p}(\alpha) \xi^{[k]_{q^p}} \quad \left(\mathrm{resp.} \quad \theta_{n,\mathbb\Delta}(\alpha) = \sum_{k=0}^n \partial^{\langle k \rangle}_{\mathbb\Delta}(\alpha) \omega^{\{k\}_{\mathbb\Delta}}
\right).
\]
\end{prop}

\begin{proof}
This is easily deduced from lemma \ref{tetaqn}.
\end{proof}

\begin{rmks}
\begin{enumerate}
\item
We have
\[
\partial^{\langle 0 \rangle}_{q^p} = \partial^{\langle 0 \rangle}_{\mathbb\Delta} = \mathrm{Id}_R, \quad \partial^{\langle 1 \rangle}_{q^p} = \partial_{q^p} \quad \mathrm{and} \quad \partial^{\langle 1 \rangle}_{\mathbb\Delta} = \partial_{\mathbb\Delta}.
\]
Also, for all $k \in \mathbb N$, $\partial^{\langle k \rangle}_{\mathbb\Delta} = (p)_q^k\partial^{\langle k \rangle}_{q^p}$ so that, again, one is a logarithmic version of the other.
\item
Note however that, when $n \geq 2$, we have
\[
\partial^{\langle 2 \rangle}_{q^p}(q^n) = \partial^2_{q^p}(q^n) - q^{(p+1)(n-1)} \partial_{q^p}((n)_{q^p})
\]
and
\[
\partial^{\langle 2 \rangle}_{\mathbb\Delta}(q^n)= \partial^2_{\mathbb\Delta}(q^n) - q^{(p+1)(n-1)}\partial_{\mathbb\Delta}((pn)_{q}).
\]
It is therefore important not to confuse the $k$th higher derivative with the $k$th power of the derivation (unlike in the untwisted or relative case).
\end{enumerate}
\end{rmks}

Since
\[
R[\xi]/\xi^{(2)_{q^p}} \simeq R\langle \xi \rangle_{q^p}/\xi^{[>1]_{q^p}} \quad \left(\mathrm{resp.} \quad R[\omega]/\omega^{(2)_{\mathbb\Delta}} \simeq R\langle \omega \rangle_{q^p}/\omega^{\{>1\}_{\mathbb\Delta}}\right),
\]
we can interpret $\Omega_{q^p}$ (resp.\ $\Omega_{\mathbb\Delta}$) as the kernel of the augmentation map $R\langle \xi \rangle_{q^p}/\xi^{[>1]_{q^p}} \to R$ (resp.\ $R\langle \omega \rangle_{\mathbb\Delta}/\omega^{\{>1\}_{\mathbb\Delta}} \to R$).

The action of these Taylor maps on twisted divided powers is quite explicit (recall that $g_{pn+1}$ denotes the unique endomorphism of $R$ such that $g_{pn+1}(q) = q^{pn+1}$) :

\begin{lem} \label{congmod}
For all $n \in \mathbb N$ and $\alpha \in R$, we have
\[
\theta_{n,q^p}(\alpha)\xi^{[n]_{q^p}} \equiv g_{pn+1}(\alpha)\xi^{[n]_{q^p}} \mod \xi^{[n+1]_{q^p}}
\]
and
\[
\theta_{n,\mathbb\Delta}(\alpha)\omega^{\{n\}_{\mathbb\Delta}} \equiv g_{pn+1}(\alpha)\omega^{\{n\}_{\mathbb\Delta}} \mod \omega^{\{n+1\}_{\mathbb\Delta}}.
\]
\end{lem}

\begin{proof}
The second assertion follows by blowing up the first one which reduces to formula \eqref{thetsig} since both $\theta_{n,q^p}$ and $g_{pn+1}$ are continuous morphisms of $W$-algebras.
One may also derive the second assertion directly from lemma \ref{estcong}.
\end{proof}

Lemma \ref{congmod} shows a tight relation between twisted Taylor maps and the automorphisms $g_r$ for $r \equiv 1 \mod p$.
As a consequence of the lemma, the usual basis of a ring of twisted divided polynomials is also a basis for the \emph{right} structure.
In particular, the Taylor maps of finite order are finite free:

\begin{prop}
For all $n \in \mathbb N$, the $R$-algebra $R\langle \xi \rangle_{q^p}/\xi^{[>n]_{q^p}}$ (resp.\ $R\langle \omega \rangle_{{\mathbb\Delta}}/ {\omega^{\{>n\}_{{\mathbb\Delta}}}}$) is free on $\xi^{[n]_{q^p}}$ (resp.\ $\omega^{\{k\}_{{\mathbb\Delta}}}$) for $k \leq n$ for both the left and right structures.
\end{prop}

\begin{proof}
Only the case of the right structure needs a proof but it follows from lemma \ref{congmod} that the $R$-module $\xi^{[\geq n]_{q^p}}/\xi^{[>n]_{q^p}} $ (resp.\ ${\omega^{\{\geq n\}_{{\mathbb\Delta}}}}/ {\omega^{\{>n\}_{{\mathbb\Delta}}}}$) is free of rank one on $\xi^{[n]_{q^p}} $ (resp.\ $\omega^{\{n\}_{{\mathbb\Delta}}}$) for the right structure also.
Our assertion is then obtained by induction.
\end{proof}

We end this section with some preparation for section \ref{stratder}.
We will systematically extend the endomorphism $g_r$ of $R$ (given by $g_r(q) = q^r$) to the polynomial ring $R[\xi]$ by simply setting
\[
g_r(q+\xi) = q + \xi
\]
so that $g_r(\xi) = \xi + q - q^r$.
We may notice that
\begin{equation} \label{sigmaxi}
g_{kp+1}(\xi) = \xi + q - q^{kp+1}
\end{equation}
so that, with the notations of section \ref{twispow} and considering the case $x=q$, we have
\[
\xi^{(n)_{q^p}} = \prod_{k=0}^{n-1}g_{kp+1}(\xi).
\]
The $R$-linear endomorphism $g_r'$ of the polynomial ring $R[\xi]$ given by
\[
g_r'(q+\xi) = (q+\xi)^r
\]
will also show up.
The maps $g_r$ and $g_r'$ on $R[\xi]$ are exchanged by the flip map that interchanges $q$ and $q + \xi$.

The endomorphism $g_r$ of $R[\xi]$ extends uniquely to a continuous endomorphism of the ring $R[[\xi]]$.
This follows from the fact that
\[
g_r(\xi) = \xi - (r-1)_q(q^2-q) \in (q-1, \xi).
\]
Alternatively, the endomorphism $g_r$ of $R[[\xi]]$ corresponds to the unique continuous endomorphism of the ring $W[[q_{0}-1, q_{1}-1]] $ such that $q_0 \mapsto q_0^{r}$ and $q_1 \mapsto q_1$.
Or better, it corresponds to the endomorphism $g_r \widehat \otimes_W \mathrm{Id}_R$ of $R \widehat \otimes_{W} R$.

All the above applies in particular to $\gamma := g_{p+1}$ and we shall then write $\gamma' := g'_{p+1}$.
The endomorphism $\gamma$ of $R[\xi]$ extends naturally to an endomorphism of $R[\omega]$ via blowing up $\xi \mapsto (p)_q\omega$.
This is shown as follows.
From equality \eqref{sigmaxi}, we get
\begin{equation} \label{sigomav}
(p)_{q^{p+1}}\gamma(\omega) = (p)_q(\omega -q^2+q),
\end{equation}
and therefore
\begin{equation} \label{sigom}
\gamma(\omega) = \frac 1\lambda( \omega -q^2+q).
\end{equation}
with $\lambda$ as in formula \eqref{pqplun}.

It is much harder to extend $\gamma$ to $R\langle\xi\rangle_{q^p}$ or $R\langle \omega \rangle_{\mathbb\Delta}$ (but see proposition \ref{sigcont} below).

\section{The $p$-adic situation} \label{padsit}

We specialize now our situation to the case of a $p$-adic base ring where $p$ is a fixed prime.
With this extra requirement, we will be able to define the notion of a $\mathbb\Delta$-Taylor map of infinite order.
This is necessary in order to develop absolute hyper calculus in the next section.

So, we fix now a \emph{prime} $p$ and a $p$-adically complete ring $W$, and we set $R := W[[q-1]]$.
Unless otherwise specified, all $R$-modules (and in particular $R$-algebras) are endowed with the $(p,q-1)$-adic topology and completion is meant with respect to this topology.

In this situation, we have the following\footnote{Since $p$ is now a fixed prime, we use the letter $r$ for a generic integer that will soon be equal to $p$.}:

\begin{lem} \label{extqr}
For all $r \in \mathbb N$, there exists a canonical map
\[
R[[\xi ]] \to \widehat{R\langle \xi \rangle}_{q^r}.
\]
\end{lem}

\begin{proof}
Since $(p^s)_{q^r} \equiv 0 \mod (p,q-1)$ for all $s >0$, we have $(n)_{q^r}! \to 0$ when $n \to \infty$, and so does the image $(n)_{q^r}!\xi^{\{n\}_{q^r}}$ of $\xi^{(n)_{q^r}}$.
Our assertion therefore follows from lemma \ref{qlim} applied to $\xi$ (instead of $\omega$).
\end{proof}

Recall that we introduced in definition \ref{Taylfinf} the Taylor map of infinite level.

\begin{dfn} \label{defqpt}
The \emph{$q^r$-Taylor map} (resp. \emph{$\mathbb\Delta$-Taylor map}) of $R$ is the composite map
\[
\theta_{q^r} : R \overset{\theta^{(\infty)}}\longrightarrow R[[\xi ]] \longrightarrow \widehat{R\langle \xi \rangle}_{q^r}.
\]
\[
\left(\mathrm{resp.} \quad \theta_{\mathbb\Delta} : R \overset{\theta_{q^r}}\longrightarrow \widehat{R\langle \xi \rangle}_{q^p} \longrightarrow \widehat{R\langle \omega \rangle}_{\mathbb\Delta} \right)
\]
\end{dfn}

We shall again simply write $\theta$ when there is no risk of confusion and focus on the $\mathbb\Delta$ case.
The $\mathbb\Delta$-Taylor map is the unique continuous $W$-linear map such that $q \mapsto q + (p)_q\omega$.
We recover the $\mathbb\Delta$-Taylor map of finite order $n$ of definition \ref{finTayl} as the composition
\[
\theta_{n,\mathbb\Delta} : R \overset{\theta_{\mathbb\Delta}} \longrightarrow \widehat {R\langle \omega \rangle}_{\mathbb\Delta} \twoheadrightarrow \widehat {R\langle \omega \rangle}_{\mathbb\Delta}/\widehat {\omega^{\{>n\}_{\mathbb\Delta}}} \simeq R\langle \omega \rangle_{\mathbb\Delta}/\omega^{\{>n\}_{\mathbb\Delta}}.
\]
The last isomorphism follows from the fact that $R\langle \omega \rangle_{\mathbb\Delta}/\omega^{\{>n\}_{\mathbb\Delta}}$ is a finite projective (actually free) $R$-module.

In order to also extend below the flip map, we will need the following:

\begin{lem} \label{pqinv}
In $\widehat{R\langle \omega \rangle}_{\mathbb\Delta}$, we have $\theta((p)_q) = L(\omega)(p)_q$ with $L(\omega) \in \widehat{R\langle \omega \rangle}_{\mathbb\Delta}^\times$.
\end{lem}

\begin{proof}
It follows from the usual Taylor formula that
\[
\theta((p)_q) = \sum_{k=0}^{p-1} \partial^{[k]}((p)_q) (p)_q^{k}\omega^{k}.
\]
Therefore, the equality will hold with
\begin{equation} \label{uomega}
L(\omega) = 1 + \left( \sum_{k=1}^{p-1} \partial^{[k]}((p)_q) (p)_q^{k-1}\omega^{k-1}\right)\omega.
\end{equation}

This is an invertible element because $\omega$ is topologically nilpotent thanks to lemma \ref{qlim}.
\end{proof}

In order to go further, we shall need some results from the theory of prisms of Bhatt and Scholze \cite{BhattScholze22}.
A \emph{$\delta$-pair} is a couple $(B,J)$ where $B$ is a $\delta$-ring and $J$ is an ideal in $B$.
We shall then write $\overline B := B/J$.
The $\delta$-pair is said to be bounded if $\overline B$ has bounded $p^{\infty}$-torsion.
A \emph{bounded prism} is a bounded $\delta$-pair $(B,J)$ such that $J$ is invertible, $B$ is complete for the $(p) +J$-adic topology and $p \in J+\phi(J)B$.
When $J$ is understood from the context, we may simply say that $B$ is a bounded prism.
When $J = (d)$ is principal, we just write $(B,d)$ and call $B$ \emph{oriented}.
For example, if $W$ itself is a $\delta$-ring with bounded $p^{\infty}$-torsion and we set $\delta(q)=0$, then $(R, (p)_q)$ is an oriented bounded prism.
In general, the \emph{(bounded) prismatic envelope} of a bounded $\delta$-pair $(B,J)$, if it exists, is a bounded prism which is universal for morphisms of $\delta$-pairs from $(B,J)$ to bounded prisms.

The following result is the absolute variant of \cite{GrosLeStumQuiros22b}, proposition 5.7.

\begin{thm} \label{prismenv}
Let $W$ be a $p$-adically complete $\delta$-ring with bounded $p^{\infty}$-torsion and $R := W[[q-1]]$.
We endow $R$ (resp.\ $R[[\xi]]$) with the unique $\delta$-structure such that $\delta(q) = 0$ (resp.\ $\delta(q) = 0$ and $\delta(q+\xi) = 0$).
Then $\left(\widehat {R\langle \omega \rangle}_{\mathbb\Delta}, (p)_q\right)$ is the (bounded) prismatic envelope of $(R[[\xi]], \xi)$.
\end{thm}

\begin{proof}
We apply \cite{GrosLeStumQuiros22b}, theorem 5.2 to the case $A=R$ and $x=q$.
It means that $R\langle \omega \rangle_{\mathbb\Delta}$ is universal for morphisms of $\delta$-rings from $R[\xi]$ to a $(p)_q$-torsion free $\delta$-ring $B$ that send $\xi$ into $(p)_qB$.
Our assertion then follows from lemma \ref{extqr}, and corollary 4.5 of \cite{GrosLeStumQuiros23a} which implies that $\widehat {R\langle \omega \rangle}_{\mathbb\Delta}$ is bounded.
\end{proof}

\begin{rmks}
\begin{enumerate}
\item
We do not want to go into the details but there exists an analog of theorem \ref{prismenv} with $q^r$ instead of $\mathbb\Delta$: for any $r \in \mathbb N$, the $\delta$-pair $\left(\widehat {R\langle \xi \rangle}_{q^r}, (p)_{q^r}\right)$ is the $q^r$-PD-envelope of $(R[[\xi]], \xi)$.
This is a direct consequence of \cite{GrosLeStumQuiros23a}, theorem 3.5 (take $A=R$, replace $q$ with $q^r$ and $x$ with $q$).
\item
The prism $(A,I)$ introduced by Bhargav Bhatt and Jacob Lurie in \cite{BhattLurie22}, proposition 2.6.5 in order to describe prismatic logarithms, is exactly our $\left(\widehat {R\langle \xi \rangle}_{q}, (p)_{q}\right)$ when $W = \mathbb Z_p$.
In their notations, we have $u = 1+ \xi/q$.
\end{enumerate}
\end{rmks}
We can now extend the flip map that was introduced in definition \ref{fldef}:

\begin{prop} \label{extflip}
The flip map of $R[[\xi]]$ extends uniquely to a continuous semilinear (with respect to $\theta$) automorphism $\tau$ of the ring $\widehat{R\langle \omega \rangle}_{\mathbb\Delta}$ and we have $\tau \circ \tau = \mathrm{Id}_{\widehat{R\langle \omega \rangle}_{\mathbb\Delta}}$.
\end{prop}

\begin{proof}
This is completely formal and we may therefore assume that $W = \mathbb Z_p$.
Then, $W$ is a complete $\delta$-ring (for its unique $\delta$-structure) with no $p$-torsion (and in particular with bounded $p^\infty$-torsion).
We endow $R$ (resp.\ $R[[\xi]]$) with the unique $\delta$-structure such that $\delta(q) = 0$ (resp.\ $\delta(q) = 0$ and $\delta(q+\xi) = 0$).
We showed in lemma \ref{pqinv} that $\theta((p)_q) \in (p)_q \widehat{R\langle \omega \rangle}_{\mathbb\Delta}$.
Since the flip map exchanges $q$ and $q + \xi$, this is a morphism of $\delta$-rings.
It therefore follows from theorem \ref{prismenv} that the flip map extends uniquely to a $\theta$-semilinear endomorphism of the bounded prism $\left(\widehat{R\langle \omega \rangle}_{\mathbb\Delta}, (p)_q\right)$.
The last assertion also follows from the universal property.
\end{proof}

\begin{xmps}
\begin{enumerate}
\item
We have
\[
\theta((p)_q)\tau(\omega) = \tau((p)_q\omega) = \tau(\xi) = -\xi = -(p)_q\omega
\]
and therefore $\tau(\omega) = -L(\omega)^{-1}\omega$ with $L(\omega)$ as in formula \eqref{uomega}.
This shows that this flip map is not trivial at all.
In particular, it seems hard to provide an explicit formula for higher twisted powers as we did for example in \cite{GrosLeStumQuiros23a}, proposition 5.4 in level $0$ in the relative case.
\item Alternatively, $\tau$ is the unique natural continuous semilinear automorphism of the $R$-algebra $\widehat{R\langle \omega \rangle}_{\mathbb\Delta}$ such that $\tau(\omega) = -L(\omega)^{-1}\omega$.
\item
As a consequence of the above computation, we get $L(\omega)^{-1} = \tau(L(\omega))$ since
\[
\omega = \tau(\tau(\omega)) = \tau(-L(\omega)^{-1}\omega) =-\tau(L(\omega))^{-1}\tau(\omega) =\tau(L(\omega))^{-1}L(\omega)^{-1}\omega.
\]
\end{enumerate}
\end{xmps}

In order to go further, we need more estimates:

\begin{lem} \label{transan}
In $\widehat{R \langle \xi \rangle}_{q^p}$, we have
\[
\forall n,r \in \mathbb N, \quad \theta((n)_{q^{pr}}) \equiv 0 \mod (n)_{q^{pr}}.
\]
\end{lem}

\begin{proof}
If $p \nmid n$, then $(n)_{q^{pr}}$ is invertible and the assertion is trivially true.
We may therefore assume that $n=pm$ for some $m \in \mathbb N$ and we have
\[
(n)_{q^{pr}} = (pm)_{q^{pr}} = (p)_{q^{pr}}(m)_{q^{p^2r}}.
\]
By induction, we are reduced to the case $n=p$.
If $r,i \in \mathbb N$, we have
\[
(pri)_{q^p} = (r)_{q^p}(pi)_{q^{pr}} = (r)_{q^p}(p)_{q^{pr}}(i)_{q^{p^2r}}
\]
and therefore $(pri)_{q^p} \equiv 0 \mod (p)_{q^{pr}}$.
It follows that if $k \neq 0$, then
\[
(k)_{q^{p}}! {pri \choose k}_{q^{p}} \equiv 0 \mod (p)_{q^{pr}}.
\]
Therefore, using proposition \ref{tetaqn2}, we see that
\[
\theta(q^{pri}) = \sum_{k=0}^{pri} (k)_{q^{p}}! {pri \choose k}_{q^{p}}q^{pri-k} \xi^{[k]_{q^{p}}} \equiv q^{pri} \mod (p)_{q^{pr}}
\]
and finally
\[
\theta((p)_{q^{pr}}) = \sum_{i=0}^{p-1} \theta(q^{pri}) \equiv \sum_{i=0}^{p-1} q^{pri} = (p)_{q^{pr}} \equiv 0 \mod (p)_{q^{pr}}.\qedhere
\]
\end{proof}

\begin{rmks} \phantomsection \label{thetarmks}
\begin{enumerate}
\item
Of course, the result then also hods in $\widehat {R\langle \omega \rangle}_{\mathbb\Delta}$, but it is not true anymore in $\widehat{R \langle \xi \rangle}_{q}$.
For example, $\theta((2)_q) = (2)_q + \xi$ is not a multiple of $(2)_q$ if $p=2$.
For the same reason, there is no analog of lemma \ref{pqinv} in $\widehat{R\langle \xi \rangle}_q$ or even in $\widehat{R\langle \xi \rangle}_{q^p}$.
\item From lemmas \ref{pqinv} and \ref{transan}, one can deduce that
\[
\forall n \in \mathbb N, \quad \theta((pn)_q) \equiv 0 \mod (pn)_{q}
\]
in $\widehat{R \langle \omega \rangle}_{\mathbb\Delta}$.
\item As a consequence, if either $p \mid r$ or else, $r=1$ and $p \mid n$, then $\partial_{\mathbb\Delta}\left((n)_{q^{r}}\right) \in (n)_{q^{r}}R$.
This provides a whole collection of examples of free $\nabla_{\mathbb\Delta}$-modules of rank one.
This also shows that the \emph{twisted filtration} defined on $R$ by $(p^n)_qR$, $n \ge0$, is stable under the action of $\partial_{\mathbb\Delta}$.
\end{enumerate}
\end{rmks}

We can now prove the following analog to lemma \ref{pqinv}:

\begin{prop}
For all $n \in \mathbb N$, we have $\theta((n)_{q^p}) = u_n(\omega)(n)_{q^p}$ with $u_n(\omega) \in \widehat{R\langle \omega \rangle}_{\mathbb\Delta}^\times$.
\end{prop}

\begin{proof}
It follows from lemma \ref{transan} that $\theta((n)_{q^p}) = u_n(\omega)(n)_{q^p}$ for some $u_n(\omega) \in \widehat{R\langle \omega \rangle}_{\mathbb\Delta}$ and it remains to show that $ u_n(\omega)$ is invertible.
But then, using remark \ref{classd}.1, we have
\[
(n)_{q^p} = \tau(\theta((n)_{q^p})) = \tau(u_n(\omega)(n)_{q^p}) = \tau(u_n(\omega))\theta((n)_{q^p}) = \tau(u_n(\omega))u_n(\omega)(n)_{q^p}.
\]
Using the fact that $(n)_{q^p}$ is regular, we even get the formula $u_n(\omega)^{-1} = \tau(u_n(\omega))$.
\end{proof}


We need to go one step further and we consider now the isomorphism:
\[
\xymatrix@R=0cm{
R \widehat \otimes_{W} R \widehat \otimes_{W} R \ar[r]^\simeq & R[[\xi]] \widehat \otimes'_{R} R[[\xi]]
\\ || & ||
\\
W[[q_{0}-1, q_{1}-1, q_2-1]] \ar[r]^-\simeq & W[[q-1,\xi_1,\xi_2]] 
\\q_0 \ar@{|->}[r] & q
\\ q_1 \ar@{|->}[r] & q + \xi_1
\\ q_2 \ar@{|->}[r] & q + \xi_1 + \xi_2
}
\]
where $\otimes'$ indicates that we use the Taylor map as a structural map on the left.
The various maps
\[
p_{23}, p_{13}, p_{12} : R \widehat \otimes_{W} R \to R \widehat \otimes_{W} R \widehat \otimes_{W} R,
\]
that correspond geometrically to the projections that forget the first, middle and last component respectively, are taken through this isomorphism (respectively again) to
\[
p_2, \Delta, p_1 : R[[\xi]] \to R[[\xi]] \widehat \otimes'_{R} R[[\xi]],
\]
with
\[
\left\{\begin{array}{l} p_1(q) = q \otimes 1 \\ p_1(\xi) = \xi \otimes 1\end{array}\right., \quad \left\{\begin{array}{l} p_2(q) = \theta(q) \otimes 1 \\ p_2(\xi) = 1 \otimes \xi \end{array}\right. \quad \mathrm{and} \quad \left\{\begin{array}{l} \Delta(q) = q \otimes 1 \\ \Delta(\xi) = 1 \otimes \xi + \xi \otimes 1.\end{array}\right.
\]
In particular, both $p_1$ and $\Delta$ are $R$-linear with respect to the left action but $p_2$ is $R$-linear with respect to the right action.
We shall call $\Delta$ the \emph{comultiplication} map.

\begin{lem} \label{prismdiag}
Assume $W$ is a $\delta$-ring with bounded $p^{\infty}$-torsion.
Then, the ring $\widehat{R\langle\omega\rangle}_{\mathbb\Delta} \widehat \otimes'_{R} \widehat{R\langle\omega \rangle}_{\mathbb\Delta}$ is the (bounded) prismatic envelope of $(\xi \otimes 1, 1 \otimes \xi)$ in $R[[\xi]] \widehat \otimes'_{R} R[[\xi]]$.
\end{lem}

\begin{proof}
It follows from lemma \ref{pqinv} that $\widehat {R\langle\omega \rangle}_{\mathbb\Delta}$ is not only the prismatic envelope of $\xi$ in $R[[\xi]]$ when $R[[\xi]]$ is seen as an $R$-module via the canonical map, but also when it is seen as an $R$-module via the \emph{Taylor map} (which is a morphism of $\delta$-rings).
Our assertion therefore follows from the universal property of completed tensor product and \cite{GrosLeStumQuiros23a}, corollary 4.5 that insures us that $\widehat{R\langle\omega\rangle}_{\mathbb\Delta} \widehat \otimes'_{R} \widehat{R\langle\omega \rangle}_{\mathbb\Delta}$ is bounded.
\end{proof}

\begin{prop} \label{extcomult}
The comultiplication map $\Delta$ that sends $\xi$ to $1 \otimes \xi + \xi \otimes 1$ extends uniquely to a morphism of $R$-algebras
\[
\Delta : \widehat{R\langle\omega\rangle}_{\mathbb\Delta} \to \widehat{R\langle\omega\rangle}_{\mathbb\Delta} \widehat \otimes'_{R} \widehat{R\langle\omega \rangle}_{\mathbb\Delta}.
\]
\end{prop}

\begin{proof}
Uniqueness is automatic.
Concerning existence, we may assume that $W = \mathbb Z_p$ and our assertion therefore follows from theorem \ref{prismenv} and lemma \ref{prismdiag}.
\end{proof}

\begin{rmks}
\begin{enumerate}
\item As an example, we can compute
\begin{equation} \label{Delom}
\Delta(\omega) = L(\omega) \otimes \omega + \omega \otimes 1
\end{equation}
with $L(\omega)$ given by formula \eqref{uomega}.
Actually, comultiplication is the unique natural morphism of $R$-algebras such that equality \eqref{Delom} holds.
\item This comultiplication map $\Delta$ is not the same as the diagonal map of \cite{GrosLeStumQuiros22b} which simply sends $\omega$ to $1 \otimes \omega + \omega \otimes 1$.
\item
There is no general comultiplication map
\[
\Delta_{m,n} : R\langle\omega\rangle_{\mathbb\Delta}/\omega^{\{>n+m\}_{\mathbb\Delta}} \to R\langle\omega\rangle_{\mathbb\Delta}/\omega^{\{>n\}_{\mathbb\Delta}} \otimes'_{R} R\langle\omega \rangle_{\mathbb\Delta}/\omega^{\{>m\}_{\mathbb\Delta}}
\]
when $n+m >1$.
This is why composition of $\mathbb\Delta$-differential operators of finite order below will have infinite order in general.
For the same reason, there exists no such thing as a $\mathbb\Delta$-Taylor structure as in definition 3.8 of \cite{GrosLeStumQuiros22b}\footnote{The remark at the end of section 3 of that article stating that any Taylor structure extends to a stratification is not true in general, but this does not affect any other results.} on an $R$-module in general.
It will be essential to develop a hyper version of the theory.
\end{enumerate}
\end{rmks}

\section{Absolute hyper calculus} \label{stratder}

We develop now the formalism of absolute hyper calculus.
This will allow us to make the link with prismatic crystals but it is also necessary in order to obtain a full theory of absolute calculus.

We still fix a \emph{prime} $p$, a $p$-adically complete ring $W$ and write $R := W[[q-1]]$.

We are in a situation where one can apply the formalism of groupoids in order to develop hyper calculus, as explained for example in chapter 2 of \cite{Berthelot74}, or more recently in the appendix of \cite{Ogus22}.
More precisely, the formal scheme
\[
G := \mathrm{Spf}(R\langle\omega\rangle_{\mathbb\Delta})
\]
is a formal groupoid (it represents a presheaf of groupoids on the category of formal schemes) over $\mathrm{Spf}(R)$:
\begin{itemize}
\item
If $T$ is a formal scheme over $W$, then an object (resp.\ a morphism) in $G(T)$ is a morphism of formal schemes\footnote{In the notations of \cite{Ogus22}, $Y = \mathrm{Spf}(R)$ and $A=\mathrm{Spf}(\widehat{R\langle\omega\rangle}_{\mathbb\Delta})$.} $T \to \mathrm{Spf}(R)$ (resp.\ $T \to \mathrm{Spf}(R\langle\omega\rangle_{\mathbb\Delta})$).
\item
Codomain (resp.\ domain, resp.\ identity, resp.\ composition, resp. inverse) is obtained from the canonical (resp.\ Taylor, resp.\ augmentation, resp.\ comultiplication, resp. flip) map.
\end{itemize}
We shall however not rely on this formalism but, instead, will introduce successively the different notions that we need.

\begin{dfn} \label{hyperst}
A \emph{$\mathbb\Delta$-hyperstratification} on an $R$-module $M$ is an $\widehat{R\langle \omega \rangle}_{\mathbb\Delta}$-linear isomorphism
\[
\epsilon_M : \widehat{R\langle \omega \rangle}_{\mathbb\Delta} \widehat \otimes'_R M \simeq M \widehat \otimes_R \widehat{R\langle \omega \rangle}_{\mathbb\Delta} 
\]
(or $\epsilon$ for short) satisfying the cocycle condition 
\begin{equation} \label{trucoc}
(\epsilon_M \otimes \mathrm{Id}) \circ (\mathrm{Id}\otimes \epsilon_M) \circ (\Delta \otimes \mathrm{Id}_M) = (\mathrm{Id}_M \otimes \Delta) \circ \epsilon_M.
\end{equation}
\end{dfn}

The category of $R$-modules endowed with a $\mathbb\Delta$-hyperstratification is defined in the obvious way.
For $M$ and $N$ in this category, there exists the notion of tensor product of $\mathbb\Delta$-hyperstratifications given by
\[
\xymatrix{(\widehat{R\langle \omega \rangle}_{\mathbb\Delta} \widehat \otimes'_R M) \otimes_{\widehat{R\langle \omega \rangle}_{\mathbb\Delta}} (\widehat{R\langle \omega \rangle}_{\mathbb\Delta} \widehat \otimes'_R N)\ar[r]^-\simeq \ar[d]^{\epsilon_M \otimes \epsilon_N}_\simeq &
\widehat{R\langle \omega \rangle}_{\mathbb\Delta} \widehat \otimes'_R (M \otimes_R N)\ar[d]^{\epsilon_{M \otimes N}}_\simeq
\\ (M \widehat \otimes_R \widehat{R\langle \omega \rangle}_{\mathbb\Delta} \otimes_{\widehat{R\langle \omega \rangle}_{\mathbb\Delta}} (N \widehat \otimes_R \widehat{R\langle \omega \rangle}_{\mathbb\Delta}) \ar[r]^-\simeq &
(M \otimes_R N) \widehat \otimes_R \widehat{R\langle \omega \rangle}_{\mathbb\Delta}.
}
\]

When, moreover, $M$ is assumed to be finite projective, there also exists a $\mathbb\Delta$-hyperstratification on $\mathrm{Hom}_R(M, N)$, which is given by the following chain of isomorphisms
\[
\xymatrix{\widehat{R\langle \omega \rangle}_{\mathbb\Delta} \widehat \otimes'_R \mathrm{Hom}_R(M, N) \ar[r]^-\simeq_{\epsilon_{\mathrm{Hom}(M,N)}} \ar[d]^\simeq &
\mathrm{Hom}_R(M, N) \widehat \otimes_R \widehat{R\langle \omega \rangle}_{\mathbb\Delta} \ar[d]^\simeq
\\ \mathrm{Hom}_R(M, \widehat{R\langle \omega \rangle}_{\mathbb\Delta} \widehat \otimes'_R N)^\wedge \ar[d]^\simeq &
 \mathrm{Hom}_R(M , N \widehat \otimes_R \widehat{R\langle \omega \rangle}_{\mathbb\Delta})^\wedge \ar[d]^\simeq
\\ \mathrm{Hom}_{\widehat{R\langle \omega \rangle}_{\mathbb\Delta}}(\widehat{R\langle \omega \rangle}_{\mathbb\Delta} \widehat \otimes'_R M, \widehat{R\langle \omega \rangle}_{\mathbb\Delta} \widehat \otimes'_R N)^\wedge \ar[r]^-\simeq_-{\epsilon_M,\epsilon_N} &
 \mathrm{Hom}_{\widehat{R\langle \omega \rangle}_{\mathbb\Delta}}(M \widehat \otimes_R \widehat{R\langle \omega \rangle}_{\mathbb\Delta}, N \widehat \otimes_R \widehat{R\langle \omega \rangle}_{\mathbb\Delta})^\wedge
}
\]

As usual, we can consider a (apparently) weaker notion (we denote by $e : \widehat{R\langle\omega\rangle}_{\mathbb\Delta} \to R$ the augmentation map):

\begin{dfn} \label{TaylMod}
A \emph{$\mathbb\Delta$-Taylor map} on an $R$-module $M$ is a \emph{semilinear} map ($R$-linear for the right structure)
\[
\theta_M : M \to M \widehat \otimes_{R} \widehat {R\langle\omega\rangle}_{\mathbb\Delta}
\]
(or $\theta$ for short) such that
\begin{equation} \label{eqcoc}
(\mathrm{Id}_{M} \otimes e) \circ \theta_M = \mathrm{Id}_{M}
\quad \mathrm{and} \quad 
(\theta_M \otimes \mathrm{Id}_{\widehat {R\langle\omega\rangle}_{\mathbb\Delta}}) \circ \theta_M = (\mathrm{Id}_{M} \otimes \Delta) \circ \theta_M.
\end{equation}
\end{dfn}

We can make explicit the semilinearity condition: the right hand side is naturally a $\widehat {R\langle\omega\rangle}_{\mathbb\Delta}$-module and we have
\[
\forall \alpha \in R, \forall s \in M, \quad \theta_M(\alpha s) = \theta_R(\alpha)\theta_M(s)
\]
where $\theta_R$ denotes the $\mathbb\Delta$-Taylor map of $R$.

Again, $R$-modules endowed with a $\mathbb\Delta$-Taylor map form a category in the obvious way and there exists a tensor product (as well as a partial internal Hom, that we will not make explicit) given by
\begin{equation} \label{tenscp}
M \otimes_R N \to (M \widehat \otimes_R \widehat{R\langle \omega \rangle}_{\mathbb\Delta}) \otimes_{\widehat{R\langle \omega \rangle}_{\mathbb\Delta}} (N \widehat \otimes_R \widehat{R\langle \omega \rangle}_{\mathbb\Delta}) \simeq
(M \otimes_R N) \widehat \otimes_R \widehat{R\langle \omega \rangle}_{\mathbb\Delta}.
\end{equation}

Thanks to the existence of the flip map, we have the following:

\begin{prop} \label{eqHyperTayl}
The category of $R$-modules endowed with a $\mathbb\Delta$-hyperstratification is isomorphic to the category of $R$-modules endowed with a $\mathbb\Delta$-Taylor map.
\end{prop}

\begin{proof}
This is completely standard.
The $\mathbb\Delta$-Taylor map extends uniquely to 
a $\widehat{R\langle \omega \rangle}_{\mathbb\Delta}$-linear map
\[
\epsilon_M : \widehat{R\langle \omega \rangle}_{\mathbb\Delta} \widehat \otimes'_R M \to M \widehat \otimes_R \widehat{R\langle \omega \rangle}_{\mathbb\Delta} 
\]
that satisfies the cocycle condition and it remains to show that it is bijective.
For this purpose, we consider the ``flip map''
\begin{equation} \label{flipM}
\tau_M : M \widehat \otimes_R \widehat{R\langle \omega \rangle}_{\mathbb\Delta} \simeq \widehat{R\langle \omega \rangle}_{\mathbb\Delta} \widehat \otimes'_R M, \quad s \otimes \varphi \mapsto \tau(\varphi) \otimes s.
\end{equation}
The cocycle condition then implies that $\tau_M \circ \epsilon_M \circ \tau_M$ is an inverse for $\epsilon_M$.
\end{proof}

We may then rewrite the tensor product of two $\mathbb\Delta$-Taylor maps as
\[
M \otimes_R N \overset{\theta_M \otimes \mathrm{Id}_N} \longrightarrow
M \widehat \otimes_R \widehat{R\langle \omega \rangle}_{\mathbb\Delta} \otimes'_R N \overset{\mathrm{Id}_M \otimes \epsilon_N} \longrightarrow
(M \otimes_R N) \widehat \otimes_R \widehat{R\langle \omega \rangle}_{\mathbb\Delta}.
\]

There exists a down to earth description of these notions:

\begin{lem} \label{locTayl}
The category of $R$-modules endowed with a $\mathbb\Delta$-hyperstratification is isomorphic to the category of $R$-modules $M$ endowed with a family of continuous $W$-linear maps $\partial^{\langle k \rangle}_{M,\mathbb\Delta} : M \to M$ such that, for all $s \in M$ and $\alpha \in R$,
\begin{enumerate}
\item \label{zercond}
$\displaystyle \partial^{\langle 0 \rangle}_{M,\mathbb\Delta}(s) = s$,
\item \label{limcond}
$\displaystyle \partial^{\langle k \rangle}_{M,\mathbb\Delta}(s) \to 0\ \mathrm{as}\ k \to \infty$,
\item \label{cocond}$
\displaystyle \sum_{i,j=0}^{\infty} \left(\partial^{\langle i \rangle}_{M,\mathbb\Delta} \circ \partial^{\langle j \rangle}_{M,\mathbb\Delta}\right)(s) \otimes \omega^{\{i\}_{\mathbb\Delta}} \otimes \omega^{\{j\}_{\mathbb\Delta}} = \sum_{n=0}^{\infty} \partial^{\langle n \rangle}_{M,\mathbb\Delta}(s) \otimes \Delta\left(\omega^{\{n\}_{\mathbb\Delta}}\right),
$
\item \label{lincon}
$\displaystyle \partial^{\langle k \rangle}_{M,\mathbb\Delta}(\alpha s) = \sum_{0\leq i \leq j \leq k} q^{p{i \choose 2}}{k \choose j}_{q^{p}}{j \choose i}_{q^{p}} (q^2-q)^{i}\partial^{\langle k-j+i \rangle}_{\mathbb\Delta}(\alpha) \partial^{\langle j \rangle}_{M,\mathbb\Delta}(s)$.
\end{enumerate}
\end{lem}

\begin{proof}
This equivalence is provided by
\begin{equation} \label{Taylfor}
\forall s \in M, \quad \theta_M(s) = \sum_{k=0}^{\infty} \partial^{\langle k \rangle}_{M,\mathbb\Delta}(s) \otimes \omega^{\{k\}_{\mathbb\Delta}}.
\end{equation}
Condition \eqref{limcond} is equivalent to the existence of the map \eqref{Taylfor}.
Conditions \eqref{zercond} and \eqref{cocond} correspond to the cocycle conditions \eqref{eqcoc}.
Finally, condition \ref{lincon} is the explicit translation of right linearity.
\end{proof}

We will then call $\partial^{\langle k \rangle}_{M,\mathbb\Delta}$ the \emph{$k$th $\mathbb\Delta$-derivative} on $M$ and write $\partial_{M,\mathbb\Delta} := \partial^{\langle 1 \rangle}_{M,\mathbb\Delta}$.

\begin{xmps}
\begin{enumerate}
\item \textit{Trivial.} 
In the case
\[
M= R \quad \mathrm{and} \quad \epsilon = \mathrm{Id}_{\widehat {R\langle\omega\rangle}_{\mathbb\Delta}},
\]
the $\mathbb\Delta$-Taylor map is (resp.\ the higher $\mathbb\Delta$-derivatives are) the same as in definition \ref{defqpt} (resp.\ proposition \ref{tetaqn2}).
\item \textit{Right-regular.}
Comultiplication
\[
\Delta : \widehat{R\langle\omega\rangle}_{\mathbb\Delta} \to \widehat{R\langle\omega\rangle}_{\mathbb\Delta} \widehat \otimes'_{R} \widehat{R\langle\omega \rangle}_{\mathbb\Delta}
\]
is a $\mathbb\Delta$-Taylor map on $\widehat {R\langle\omega\rangle}_{\mathbb\Delta}$ seen as an $R$-module via the \emph{right} structure.
This follows from the formal identity
\begin{equation}
(\Delta \otimes \mathrm{Id}_{\widehat {R\langle\omega\rangle}_{\mathbb\Delta}}) \circ \Delta = (\mathrm{Id}_{\widehat {R\langle\omega\rangle}_{\mathbb\Delta}} \otimes \Delta) \circ \Delta.
\end{equation}
We will write
\[
\forall \varphi \in \widehat {R\langle\omega\rangle}_{\mathbb\Delta}, \quad \Delta(\varphi) = \sum_{k=0}^\infty L^{\langle k \rangle}_{\mathbb\Delta}(\varphi) \otimes \omega^{\{k\}_{\mathbb\Delta}}.
\]
The higher $\mathbb\Delta$-derivatives $L^{\langle k \rangle}_{\mathbb\Delta}$ are linear for the left structure.
Note that the general cocycle condition in lemma \ref{locTayl}.3 simplifies to
\begin{equation} \label{Lpart}
\sum_{i=0}^{\infty} \left(\partial^{\langle i \rangle}_{M,\mathbb\Delta} \circ \partial^{\langle j \rangle}_{M,\mathbb\Delta}\right)(s) \otimes \omega^{\{i\}_{\mathbb\Delta}}
= \sum_{n=0}^{\infty} \partial^{\langle n \rangle}_{M,\mathbb\Delta}(s) \otimes L^{\langle j \rangle}_{\mathbb\Delta}\left(\omega^{\{n\}_{\mathbb\Delta}}\right)
\end{equation}
for all $j \in \mathbb N$.
\item \textit{Left-regular.} There also exists a $\mathbb\Delta$-Taylor map for the \emph{left} structure of $\widehat {R\langle\omega\rangle}_{\mathbb\Delta}$, that we will denote again by $\theta$:
\[
\theta : \widehat {R\langle\omega\rangle}_{\mathbb\Delta} \to \widehat {R\langle\omega\rangle}_{\mathbb\Delta} \widehat \otimes_{R} \widehat {R\langle\omega\rangle}_{\mathbb\Delta}.
\]
It is obtained from $\Delta$ in the previous example by applying the flip map.
 More precisely,
 \[
 \theta = \left(\tau \otimes_{R} \mathrm{Id}_{\widehat {R\langle\omega\rangle}_{\mathbb\Delta}}\right) \circ \Delta \circ \tau.
 \]
 Alternatively, this is the unique continuous extension of the map
 \[
 R[[\xi]] \to R[[\xi]]\widehat \otimes_R R[[\xi]], \quad \left\{\begin{array} l \alpha \mapsto 1 \otimes \theta(\alpha) \\ \xi \mapsto \xi \otimes 1 - 1 \otimes \xi. \end{array}\right.
 \]
 In particular, the restriction to $R$ is exactly the $\mathbb\Delta$-Taylor map of definition \ref{defqpt}.
 We shall simply denote by
 \[
\partial^{\langle k \rangle}_{\mathbb\Delta}: \widehat {R\langle\omega\rangle}_{\mathbb\Delta} \to \widehat {R\langle\omega\rangle}_{\mathbb\Delta}
\]
the corresponding higher $\mathbb\Delta$-derivatives (which are linear for the right structure).
This notation is harmless in the sense that it is compatible with the trivial case.
Note that, by definition, we have
\[
\forall k \in \mathbb N, \quad \tau \circ \partial^{\langle k \rangle}_{\mathbb\Delta} = L^{\langle k \rangle}_{\mathbb\Delta} \circ \tau.
\]
\item \textit{Linearization.}
More generally, if $M$ is any $R$-module and we consider $M \otimes_R \widehat{R\langle\omega\rangle}_{\mathbb\Delta}$ as an $R$-module through the action on the \emph{right}, then there exists a $\mathbb\Delta$-Taylor map
\[
\mathrm{Id}_M \otimes \Delta : M \otimes_R \widehat{R\langle\omega\rangle}_{\mathbb\Delta} \to (M \otimes_R \widehat{R\langle\omega\rangle}_{\mathbb\Delta})\ \widehat \otimes'_{R} \widehat{R\langle\omega \rangle}_{\mathbb\Delta}.
\]
Using the flip map \eqref{flipM} again, we obtain a $\mathbb\Delta$-Taylor map on the \emph{linearization}
\[
L_{\mathbb\Delta}(M) := \widehat {R\langle\omega\rangle}_{\mathbb\Delta} \widehat \otimes'_R M
\]
of $M$ (now seen as an $R$-module via the action on the left).
We shall still denote by $L^{\langle k \rangle}_{\mathbb\Delta}$ and $\partial^{\langle k \rangle}_{\mathbb\Delta}$ the corresponding higher derivatives (without mentioning $M$ because it plays here a dummy role).
More precisely, we have
\[
\forall k \in \mathbb N, \forall \varphi \in \widehat {R\langle\omega\rangle}_{\mathbb\Delta}, \forall s \in M, \quad \left\{\begin{array}{c}L^{\langle k \rangle}_{\mathbb\Delta}(s \otimes \varphi) = s \otimes L^{\langle k \rangle}_{\mathbb\Delta}(\varphi), \\ \\\partial^{\langle k \rangle}_{\mathbb\Delta}(\varphi \otimes s) = \partial^{\langle k \rangle}_{\mathbb\Delta}(\varphi) \otimes s,\end{array}\right.
\]
and
\[
\forall k \in \mathbb N, \quad \tau_M \circ \partial^{\langle k \rangle}_{\mathbb\Delta} = L^{\langle k \rangle}_{\mathbb\Delta} \circ \tau_M.
\]
\end{enumerate}
\end{xmps}

Now comes the fundamental \emph{linearization lemma}:

\begin{lem} \label{maglem}
Il $M$ is a $\mathbb\Delta$-hyperstratified $R$-module and $N$ is any $R$-module, then $\epsilon_{M}^{-1}$ induces an isomorphism of $\mathbb\Delta$-hyperstratified $R$-modules
\[
M \widehat \otimes_R L_{\mathbb\Delta}(N) \simeq L_{\mathbb\Delta}(M \otimes_R N).
\]
\end{lem}

\begin{proof}
By linearity, we may assume that $N = R$.
From the very definition of the cocycle condition \eqref{trucoc}, there exists a commutative diagram
\[
\xymatrix{ \widehat {R\langle\omega\rangle}_{\mathbb\Delta} \widehat\otimes_R' M \ar[r]^-{\Delta \otimes \mathrm{Id}} \ar[d]^{\epsilon_M} & \widehat {R\langle\omega\rangle}_{\mathbb\Delta} \widehat\otimes_R'\widehat {R\langle\omega\rangle}_{\mathbb\Delta}\widehat\otimes_R' M \ar[r]^{ \mathrm{Id} \otimes \epsilon_M} & \widehat {R\langle\omega\rangle}_{\mathbb\Delta} \widehat\otimes_R' M \widehat\otimes_R \widehat {R\langle\omega\rangle}_{\mathbb\Delta} \ar[d]^{\epsilon_M \otimes \mathrm{Id}}
\\ M \widehat\otimes_R \widehat {R\langle\omega\rangle}_{\mathbb\Delta} \ar[rr]^-{\mathrm{Id} \otimes \Delta} && M \widehat\otimes_R \widehat {R\langle\omega\rangle}_{\mathbb\Delta} \widehat\otimes_R' \widehat {R\langle\omega\rangle}_{\mathbb\Delta}.
}
\]
Applying the flip maps \eqref{flipM} everywhere provides us with a commutative diagram (using in the top row the interpretation \eqref{tenscp} of tensor product)
\[
\xymatrix{M \widehat\otimes_R \widehat {R\langle\omega\rangle}_{\mathbb\Delta} \ar[rr]^-{\theta_M \otimes \theta} \ar[d]^{\epsilon^{-1}_M} && (M \widehat\otimes_R \widehat {R\langle\omega\rangle}_{\mathbb\Delta}) \widehat\otimes_R \widehat {R\langle\omega\rangle}_{\mathbb\Delta} \ar[d]^{\epsilon_M^{-1} \otimes \mathrm{Id}}
\\ \widehat {R\langle\omega\rangle}_{\mathbb\Delta} \widehat\otimes_R' M \ar[rr]^-{\theta} && (\widehat {R\langle\omega\rangle}_{\mathbb\Delta} \widehat\otimes_R' M) \widehat\otimes_R \widehat {R\langle\omega\rangle}_{\mathbb\Delta}.
}
\]
This shows that the isomorphism
\[
\epsilon_{M}^{-1} : M \widehat \otimes_R L_{\mathbb\Delta}(R) = M \widehat \otimes_R \widehat {R\langle\omega\rangle}_{\mathbb\Delta} \simeq \widehat {R\langle\omega\rangle}_{\mathbb\Delta} \widehat \otimes'_R M \simeq L_{\mathbb\Delta}(M)
\]
is compatible with the $\mathbb\Delta$-Taylor maps.
\end{proof}

In the case $N = R$, the lemma says that, with $\tau_M$ as in \eqref{flipM}, the map
\[
\tau_M \circ \theta_M : M \to L_{\mathbb\Delta}(M)
\]
is compatible with the $\mathbb\Delta$-hyperstratifications.
This may also be derived from equality \eqref{Lpart}.

We shall now make the link with our previous theory of absolute calculus:

\begin{prop} \label{inher}
If an $R$-module $M$ is endowed with a $\mathbb\Delta$-hyperstratification, then $\partial_{M,\mathbb\Delta} := \displaystyle \partial^{\langle 1 \rangle}_{M,\mathbb\Delta}$ is a $\mathbb\Delta$-derivation of $M$.
\end{prop}

\begin{proof}
It is sufficient to consider the composite map
\[
\theta_{M,1} : M \to M \widehat \otimes_{R} \widehat {R\langle\omega\rangle}_{\mathbb\Delta} \to M \otimes_R R[\omega]/\omega^{(2)_{\mathbb\Delta}}
\]
and to invoke lemma \ref{Taylcon} (or else use formula \ref{lincon} of lemma \ref{locTayl}).
\end{proof}

Formally, this proposition means that there exists a forgetful functor from the category of $\mathbb\Delta$-hyperstratified modules to the category of $\nabla_{\mathbb\Delta}$-modules.

\begin{xmps}
\begin{enumerate}
\item 
In the trivial case, we (fortunately) recover our original $\partial_{\mathbb\Delta}$ from proposition \ref{basder}.
\item In the right-regular case, the map
\[
L_{\mathbb\Delta} : \widehat {R\langle\omega\rangle}_{\mathbb\Delta} \to \widehat {R\langle\omega\rangle}_{\mathbb\Delta}
\]
is $R$-linear (for the left structure) and is a $\mathbb\Delta$-derivation for the \emph{right} structure.
It means that
\begin{equation} \label{rlin}
\forall \alpha \in R, \forall \varphi \in \widehat{R\langle\omega\rangle}_{\mathbb\Delta}, \quad L_{\mathbb\Delta}(\theta(\alpha)\varphi) = \theta(\partial_{\mathbb\Delta}(\alpha)) \varphi + \theta(\gamma(\alpha)) L_{\mathbb\Delta}(\varphi).
\end{equation}
\item In the left-regular case, we obtain a $\mathbb\Delta$-derivation (for the left structure)
 \[
\partial_{\mathbb\Delta}: \widehat {R\langle\omega\rangle}_{\mathbb\Delta} \to \widehat {R\langle\omega\rangle}_{\mathbb\Delta}
\]
that extends the standard $\mathbb\Delta$-derivation of $R$.
Moreover, it is linear for the \emph{right} structure, which means that
\[
\forall \alpha \in R, \forall \varphi \in \widehat{R\langle\omega\rangle}_{\mathbb\Delta}, \quad \partial_{\mathbb\Delta}(\theta(\alpha)\varphi) = \theta(\alpha)\partial_{\mathbb\Delta}(\varphi).
\]
\end{enumerate}
\end{xmps}

Recall that the original $\gamma$, which is defined on $R$ by $\gamma(q) = q^{p+1}$, is extended to $R[[\xi]]$ through $\gamma(q+\xi) = q+\xi$.

\begin{prop} \label{sigcont}
The endomorphism $\gamma$ of $R[[\xi]]$ extends uniquely to a continuous endomorphism of the ring
$\widehat {R\langle \omega \rangle}_{\mathbb\Delta}$ and we have
\[
\gamma(\omega) = \frac 1 \lambda (\omega -q^2+q)\quad \mathrm{with} \quad \lambda = 1 + (q^2-q)\partial_{q^p}((p)_q).
\]
\end{prop}

\begin{proof}
The first assertion is completely formal and we may assume that $W = \mathbb Z_p$.
It is then sufficient to notice that $\gamma$ is a morphism of $\delta$-rings and, using formula \eqref{pqplun}, that it sends both $(p)_q$ and $\xi$ inside the ideal generated by $(p)_q$.
We may then apply theorem \ref{prismenv}.
The last formula is simply a reminder of equality \eqref{sigom}.
\end{proof}

\begin{rmk}
The $R$-linear endomorphism of the polynomial ring $R[\xi]$ given by
\[
\gamma'(q+\xi) = (q+\xi)^{p+1} \quad (=\theta(q^{p+1}))
\]
also extends uniquely to an endomorphism of $\widehat{R\langle \omega \rangle}_{\mathbb\Delta}$ because $\gamma' = \tau \circ \gamma \circ \tau$.
We can give an explicit formula for $\gamma'(\omega)$.
We have
\[
\gamma'(\xi) = \theta(q^{p+1}) - q =\theta(q^{p+1}-q) + \xi = \theta(q^2-q)\theta((p)_q) + \xi
\]
and therefore
\[
\gamma'(\omega) = \omega + \theta(q^2-q)L(\omega)
\]
with $L(\omega)$ as in formula \eqref{uomega}.
\end{rmk}

We now show that the endomorphism $\gamma$ of proposition \ref{sigcont} is the same as the canonical endomorphism associated to the derivation $\partial_{\mathbb\Delta}$ on $\widehat{R\langle\omega\rangle}_{\mathbb\Delta}$ (and there is therefore no ambiguity in the notations):

\begin{lem} \label{sigdel}
$\forall \varphi \in \widehat{R\langle\omega\rangle}_{\mathbb\Delta}, \quad \gamma(\varphi) = \varphi + (q^2-q)\partial_{\mathbb\Delta}(\varphi).
$
\end{lem}

\begin{proof}
If we call $C(\varphi)$ this equality, then it is sufficient to prove
\[
(1) : C(1), \quad (2) : C(\varphi) \Leftrightarrow C(q \varphi) \quad \mathrm{and}\quad (3) : \quad C(\varphi) \Rightarrow C(\omega\varphi).
\]
Property (1) is trivial and $C(q\varphi)$ reads
\[
q^{p+1}\gamma(\varphi)= q\varphi + (q^2-q)(p)_q\varphi + (q^2-q)q^{p+1}\partial_{\mathbb\Delta}(\varphi).
\]
Since $q + (q^2-q)(p)_q = q^{p+1}$, we see that property (2) holds.
It remains to prove property (3).
Since $\partial_{\mathbb\Delta}$ is linear with respect to the right action, we have
\[
\partial_{\mathbb\Delta}((q+(p)_q\omega)\varphi) =(q+(p)_q\omega)\partial_{\mathbb\Delta}(\varphi).
\]
Since $\partial_{\mathbb\Delta}$ is a $\mathbb\Delta$-derivation, this may be rewritten as
\[
(p)_q\varphi + q^{p+1}\partial_{\mathbb\Delta}(\varphi)+\partial_{\mathbb\Delta}((p)_q)\omega\varphi + (p)_{q^{p+1}}\partial_{\mathbb\Delta}(\omega\varphi) =(q+(p)_q\omega)\partial_{\mathbb\Delta}(\varphi),
\]
or better:
\[
\lambda\partial_{\mathbb\Delta}(\omega\varphi) = (\omega -q^2+q)\partial_{\mathbb\Delta}(\varphi) -(1+\partial_{q^p}((p)_q)\omega)\varphi
\]
with $\lambda \in R^\times$ as in \eqref{pqplun} so that $(p)_{q^{p+1}}= \lambda(p)_q$.
We may now use equality \eqref{sigomav} and we obtain, if we assume that $C(\varphi)$ holds, that
\begin{linenomath}
\begin{align*}
\lambda\gamma(\omega\varphi) & = (\omega -q^2+q)(\varphi + (q^2-q)\partial_{\mathbb\Delta}(\varphi))
 \\ & = (q^2-q)(\omega -q^2+q)\partial_{\mathbb\Delta}(\varphi) + (\omega -q^2+q)\varphi
 \\ & = (q^2-q)((\omega -q^2+q)\partial_{\mathbb\Delta}(\varphi) - (1+\partial_{q^p}((p)_q)\omega)\varphi)
 \\ & + (q^2-q)(1+\partial_{q^p}((p)_q)\omega)\varphi + (\omega -q^2+q)\varphi
\\ & = (q^2-q) \lambda \partial_{\mathbb\Delta}(\omega\varphi) + (1+(q^2-q)\partial_{q^p}((p)_q))\omega\varphi
\\ & = \lambda (q^2-q) \partial_{\mathbb\Delta}(\omega\varphi) + \lambda \omega\varphi.
\end{align*}
\end{linenomath}
Thus, property (3) also holds.
\end{proof}

\begin{prop} \label{twder}
The endomorphism $\partial_{\mathbb\Delta}$ of $\widehat {R\langle\omega\rangle}_{\mathbb\Delta}$ is a twisted derivation of the \emph{ring} $\widehat {R\langle\omega\rangle}_{\mathbb\Delta}$ with respect to the \emph{ring} endomorphism $\gamma$ of $\widehat {R\langle\omega\rangle}_{\mathbb\Delta}$:
\[
\forall \varphi, \psi \in \widehat{R\langle\omega\rangle}_{\mathbb\Delta}, \quad \partial_{\mathbb\Delta}(\varphi\psi) = \partial_{\mathbb\Delta}(\varphi)\psi + \gamma(\varphi) \partial_{\mathbb\Delta}(\psi).
\]
\end{prop}

\begin{proof}
The assertion follows from lemma \ref{sigdel} since $\gamma$ is a ring morphism by definition.
\end{proof}

\begin{rmks}
\begin{enumerate}
\item One could have defined $\gamma$ on $\widehat {R\langle \omega \rangle}_{\mathbb\Delta}$ directly through the formula of lemma \ref{sigdel}.
It is however not clear at all that this would be a ring morphism and this is why we had to rely on proposition \ref{sigcont}
\item We can transfer our results to the right-regular case.
If $\varphi, \psi \in \widehat{R\langle\omega\rangle}_{\mathbb\Delta}$, then we have 
\begin{enumerate}
\item $\gamma'(\varphi) = \varphi + \theta(q^2-q)L_{\mathbb\Delta}(\varphi),
$
\item $
L_{\mathbb\Delta}(\varphi\psi) = L_{\mathbb\Delta}(\varphi)\psi + \gamma'(\varphi) L_{\mathbb\Delta}(\psi).
$
\end{enumerate}
\end{enumerate}
\end{rmks}

\begin{xmps}
\begin{enumerate}
\item 
We get
\[
\partial_{\mathbb\Delta}(\omega)= -\frac 1\lambda (1 + \partial_{q^p}((p)_q)\omega) \quad \mathrm{with} \quad \lambda = 1 + (q^2-q)\partial_{q^p}((p)_q)
\]
from equality \eqref{sigom} and lemma \ref{sigdel}.
\item
Applying formula \eqref{rlin} to the case $\alpha =q$ and $\varphi =1$ provides
\[
(p)_qL_{\mathbb\Delta}(\omega) = \theta((p)_q)
\]
and we can therefore write explicitly
\begin{equation} \label{Lqpom}
L_{\mathbb\Delta}(\omega) = 1 + \sum_{k=1}^{p-1} (p)_q^{k-1}\partial_{q^p}^{\langle k \rangle}((p)_q) \omega^{\{k\}_{\mathbb\Delta}}.
\end{equation}
Note that this is the same thing as $L(\omega)$ in lemma \ref{pqinv} but we now use its twisted expansion.
\end{enumerate}
\end{xmps}

It seems quite hard to compute $\partial_{\mathbb\Delta}$ or $L_{\mathbb\Delta}$ explicitly in general but the following estimate shall be enough for us:

\begin{lem} \label{estimates}
We have for all $n > 0$,
\[
L_{\mathbb\Delta}(\omega^{\{n\}_{\mathbb\Delta}}) \equiv P_{n-1}(\omega) \mod q-1
\]
with 
\[
P_n(\omega) = \sum_{i=0}^{p-1} {n+i \choose n} \alpha_{i}\omega^{\{n+i\}_{\mathbb\Delta}},
\]
\[
\alpha_{0} = 1 \quad \textrm{and} \quad \alpha_i = p^{i-1} i! \sum_{j=i}^{p-1} {j \choose i} \ \textrm{for}\ 0 < i < p.
\]
\end{lem}

\begin{proof}
This is completely formal and we may therefore assume that $W = \mathbb Z_p$.
We proceed by induction.
Since
\begin{linenomath}
\begin{align*}
\partial_{q^p}^{\langle i \rangle}((p)_q) &= \sum_{j=0}^{p-1} \partial_{q^p}^{\langle i \rangle}(q^j)
\\ &= \sum_{j=i}^{p-1} (i)_{q^p}!{j \choose i}_{q^p}q^{j-i}
\\ & \equiv i!\sum_{j=i}^{p-1} {j \choose i} \mod q-1,
\end{align*}
\end{linenomath}
it follows from formula \eqref{Lqpom} that $L_{\mathbb\Delta}(\omega) \equiv P_0(\omega) \mod q-1$ and the formula therefore holds when $n=1$.
Assume now that it holds up to some $n >0$ and recall from the very definition of $\omega^{\{n\}_{\mathbb\Delta}}$ that
\[
(n+1)\omega^{\{n+1\}_{\mathbb\Delta}} \equiv \omega^{\{n\}_{\mathbb\Delta}}\omega \mod q-1.
\]
Since $L_{\mathbb\Delta}$ is a usual derivation modulo $q-1$, we will have
\begin{linenomath}
\begin{align*}
(n+1)L_{\mathbb\Delta}(\omega^{\{n+1\}_{\mathbb\Delta}}) &\equiv L_{\mathbb\Delta}(\omega^{\{n\}_{\mathbb\Delta}})\omega + L_{\mathbb\Delta}(\omega)\omega^{\{n\}_{\mathbb\Delta}}
\\ & \equiv P_{n-1}(\omega)\omega + P_0(\omega)\omega^{\{n\}_{\mathbb\Delta}}
\\ &\equiv \sum_{i=0}^{p-1} {n+i-1 \choose n-1} \alpha_{i}\omega^{\{n+i-1\}_{\mathbb\Delta}}\omega + \sum_{i=0}^{p-1} \alpha_{i}\omega^{\{i\}_{\mathbb\Delta}}\omega^{\{n\}_{\mathbb\Delta}}
\\ &\equiv \sum_{i=0}^{p-1} \left({n+i-1 \choose n-1}(n+i) + {n+i \choose n}\right) \alpha_{i}\omega^{\{n+i\}_{\mathbb\Delta}}
\\ &\equiv \sum_{i=0}^{p-1} (n+1) {n+i \choose n} \alpha_{i}\omega^{\{n+i\}_{\mathbb\Delta}}
\\ & \equiv (n+1)P_n(\omega) \mod q-1.\qedhere
\end{align*}
\end{linenomath}
\end{proof}

We shall only need the formulas of lemma \ref{estimates} modulo $p$:

\begin{prop} \label{modp}
We have for all $n > 0$,
\[
L_{\mathbb\Delta}(\omega^{\{n\}_{\mathbb\Delta}}) \equiv \left\{\begin{array} {ll} \omega^{\{n-1\}_{\mathbb\Delta}} & \mathrm{if} \ p\ \mathrm{odd}\\ \omega^{\{n-1\}_{\mathbb\Delta}} + n \omega^{\{n\}_{\mathbb\Delta}} & \mathrm{if} \ p=2 \end{array}\right. \mod (p,q-1).
\]
\end{prop}

\begin{proof}
With the notations of lemma \ref{estimates}, we have
\[
\alpha_1\equiv \frac {p(p-1)}2 \equiv \left\{\begin{array} {ll}0 & \mathrm{if} \ p\ \mathrm{odd} \\1& \mathrm{if} \ p=2 \end{array}\right. \mod p
\]
and $\alpha_i \equiv 0 \mod p$ for $i > 1$.
It follows that
\[
P_n(\omega) \equiv \left\{\begin{array} {ll} \omega^{\{n\}_{\mathbb\Delta}} & \mathrm{if} \ p\ \mathrm{odd} \\\omega^{\{n\}_{\mathbb\Delta}} + (n+1) \omega^{\{n+1\}_{\mathbb\Delta}} & \mathrm{if} \ p=2 \end{array}\right. \mod p. \qedhere
\]
\end{proof}

\begin{prop}[Little Poincar\'e lemma] \label{little}
The sequences
\[
0 \longrightarrow R \longrightarrow \widehat{R\langle \omega \rangle}_{\mathbb\Delta} \overset {L_{\mathbb\Delta}} \longrightarrow \widehat{R\langle \omega \rangle}_{\mathbb\Delta} \longrightarrow 0
\]
and
\[
0 \longrightarrow R \overset \theta \longrightarrow \widehat{R\langle \omega \rangle}_{\mathbb\Delta} \overset {\partial_{\mathbb\Delta}} \longrightarrow \widehat{R\langle \omega \rangle}_{\mathbb\Delta} \longrightarrow 0
\]
are split exact.
\end{prop}

\begin{proof}
The second exact sequence is obtained from the first one by applying the flip map and we shall therefore only consider the first one.
This is equivalent to showing that the map $\widehat \omega ^{\{>0\}_{\mathbb\Delta}} \to \widehat{R\langle \omega \rangle}_{\mathbb\Delta}$ induced by $L_{\mathbb\Delta}$ is bijective (where $\widehat \omega^{\{>0\}_{\mathbb\Delta}}$ denotes the augmentation ideal).
Since this is a morphism of derived complete and completely flat $R$-modules, this may be shown modulo $(p,q-1)$.
When $p$ is odd, we saw in proposition \ref{modp} that
\[
L_{\mathbb\Delta}(\omega^{\{n+1\}_{\mathbb\Delta}}) \equiv \omega^{\{n\}_{\mathbb\Delta}} \mod (p,q-1).
\]
In the case $p=2$, this has to be refined a bit but we have
\begin{linenomath}
\begin{align*}
L_{\mathbb\Delta}(\omega^{\{n+1\}_{\mathbb\Delta}}-(n+1)\omega^{\{n+2\}_{\mathbb\Delta}}) &\equiv \omega^{\{n\}_{\mathbb\Delta}} -(n+1)(n+2)\omega^{\{n+2\}_{\mathbb\Delta}}
\\ &\equiv \omega^{\{n\}_{\mathbb\Delta}} \mod (2,q-1). \qedhere
\end{align*}
\end{linenomath}
\end{proof}

\begin{rmks}
\begin{enumerate}
\item
As a consequence of the little Poincar\'e lemma, we see that the $\mathbb\Delta$-Taylor map $\theta$, and therefore all higher $\mathbb\Delta$-derivatives $\partial_{\mathbb\Delta}^{\langle k \rangle}$ on $R$, are uniquely determined by the twisted derivation $\partial_{\mathbb\Delta}$ of $\widehat{R\langle \omega \rangle}_{\mathbb\Delta}$.
\item Both split exact sequences fit into a commutative diagram with exact rows and columns
\[
\xymatrix{
&0 \ar[d] &0 \ar[d] &0 \ar[d]
\\ 0 \ar[r] & \mathrm H^0_{\mathrm{dR}, \mathbb\Delta}(R)  \ar[r] \ar[d] & R \ar[r]^{\partial_{\mathbb\Delta}} \ar[d]_\theta & R \ar[d]_\theta \ar[r] & \mathrm H^1_{\mathrm{dR}, \mathbb\Delta}(R) \ar[r] & 0
\\0 \ar[r] & R \ar[r] \ar[d]^{\partial_{\mathbb\Delta}} & \widehat{R\langle \omega \rangle}_{\mathbb\Delta} \ar[r]^{L_{\mathbb\Delta}} \ar[d]^{\partial_{\mathbb\Delta}} & \widehat{R\langle \omega \rangle}_{\mathbb\Delta} \ar[r] \ar[d]^{\partial_{\mathbb\Delta}} & 0
\\ 0 \ar[r] & R \ar[r] \ar[d] & \widehat{R\langle \omega \rangle}_{\mathbb\Delta} \ar[r]^{L_{\mathbb\Delta}}\ar[d]& \widehat{R\langle \omega \rangle}_{\mathbb\Delta} \ar[r] \ar[d] & 0
\\ &\mathrm H^1_{\mathrm{dR}, \mathbb\Delta}(R) \ar[d] &0 &0
\\ & 0
}
\]
\end{enumerate}
\end{rmks}

\begin{cor} \label{fulfait}
The forgetful functor from the category of finite projective $R$-modules endowed with a $\mathbb\Delta$-hyperstratification to the category of $\nabla_{\mathbb\Delta}$-modules on $R$ is fully faithful.
\end{cor}

\begin{proof}
We start with a preliminary result.
Let $M$ be a complete $R$-module endowed with a $\mathbb\Delta$-hyperstratification and $s \in M$ such that $\partial_{\mathbb\Delta}(s) =0$.
Then equality \eqref{Lpart} in the case $j=1$ implies that $L_{\mathbb\Delta}(\theta(s)) = 0$.
Since the sequence
\[
0 \longrightarrow M \longrightarrow M \widehat \otimes_R \widehat{R\langle \omega \rangle}_{\mathbb\Delta} \overset {L_{\mathbb\Delta}} \longrightarrow M \widehat \otimes_R \widehat{R\langle \omega \rangle}_{\mathbb\Delta} \longrightarrow 0
\]
is split exact thanks to proposition \ref{little}, we see that $\theta(s) = s \otimes 1$.
Since the direct implication is clear, we obtain
\[
\forall s \in M, \quad \theta(s) = s \otimes 1 \Leftrightarrow \partial_{\mathbb\Delta}(s) =0.
\]

It is then sufficient to apply this intermediate result when $M$ is replaced with $\mathrm{Hom}_R(M,N)$ and $s$ with an $R$-linear map $f : M \to N$.
More precisely, it follows from remark \ref{tensor}.5 that $\partial_{\mathbb\Delta}(f) = 0$ if and only if $f$ is horizontal.
Also, the description of the $\mathbb\Delta$-hyperstratification on $\mathrm{Hom}_R(M,N)$ just after definition \ref{hyperst} shows that $\theta(f) = f$ if and only $\epsilon_N \circ (\mathrm{Id} \otimes f) = (f \otimes \mathrm{Id}) \circ \epsilon_M$.
\end{proof}

\begin{rmks}
\begin{enumerate}
\item
The first part of the proof shows that $0$th cohomology in the sense of $\mathbb\Delta$-hyperstratifications and in the sense of $\nabla_{\mathbb\Delta}$-modules coincide on complete $R$-modules.
\item It is not really necessary to assume that the $R$-modules are finite projective in our statement and the same proof works as well if we only assume for example that they are derived complete and completely flat.
\end{enumerate}
\end{rmks}

We shall finish the section with some explicit examples.

\begin{xmps}
\begin{enumerate}
\item 
The $\mathbb\Delta$-derivation $\partial_{F_1,\mathbb\Delta}(s) := \partial_{\mathbb\Delta}((p)_q)s$ comes from the Taylor map $\theta_{F_1}(s) = s \otimes L(\omega)$.
This is easily seen because $F_1 \simeq (p)_qR$ and $\theta((p)_q) = (p)_qL(\omega)$.
More generally, we have $\theta_{F_n}(s) = s \otimes L(\omega)^n$.
\item The $\mathbb\Delta$-derivation $\partial_{G_1,\mathbb\Delta}(s) = (p)'_\zeta s$ comes from the Taylor map $\theta_{G_1}(s) = s \otimes 1 + (p)'_\zeta s \otimes \omega$.
This follows from the fact that $G_1 \simeq (p)_qR/(p)^2_qR$.
This easily extends to $G_n$.
\end{enumerate}
\end{xmps}

In order to give the last example, we need to discuss a bit the question of the logarithm whose definition we recall below (formula \eqref{logdef}).
It is good to keep in mind the heuristic
\[
\displaystyle \log_q(u) := \frac {q-1}{\log(q)}\log(u)
\]
and that (whenever it makes sense)
\begin{equation} \label{qplog}
p\log_{q^p}(u) = (p)_q \log_q(u).
\end{equation}

\begin{lem} \label{logq}
With $A = R[[x-1]]$, we have in $\widehat {A\langle \omega \rangle}_{\mathbb\Delta}$
\begin{equation} \label{logq_eq1}
\log_{q}\left(1 + \frac {(p)_q}x \omega\right) = p\sum_{k=1}^\infty (-1)^{k-1} (k-1)_{q^p}!(p)_q^{k-1} q^{-p{k \choose 2}} x^{-k} \omega^{\{k\}_{\mathbb\Delta}}.
\end{equation}
When specializing at $x=q$, we get in $\widehat{R\langle \omega \rangle}_{\mathbb\Delta}$
\begin{equation} \label{logq_eq2}
\log_q\left(1 + \frac {(p)_q} q\omega\right) = p\sum_{k=1}^\infty(-1)^{k-1} (k-1)_{q^p}! (p)_q^{k-1} q^{-p{k \choose 2}-k} \omega^{\{k\}_{\mathbb\Delta}}.
\end{equation}
\end{lem}

\begin{proof}
We use definition 4.5 of \cite{AnschuetzLeBras19a}:
\begin{equation} \label{logdef}
\log_{q}(u) = \sum_{k=1}^\infty (-1)^{k-1} q^{-{k \choose 2}} \frac {(u,-1;q)_k}{(k)_q},
\end{equation}
where $(u,-1;q)_k$ denotes their generalized Pochammer symbol.
In the case $u = 1 + \xi/x$, we have
\[
(u,-1;q)_k = \frac {\xi^{(k)_q}}{x^k}
\]
and the formula may be rewritten
\[
\log_{q}(1 + \xi/x) = \sum_{k=1}^\infty (-1)^{k-1} q^{-{k \choose 2}} \frac {\xi^{(k)_q}}{(k)_q x^k} = \sum_{k=1}^\infty (-1)^{k-1} (k-1)_q! q^{-{k \choose 2}} \frac {\xi^{[k]_q}}{x^k}.
\]
After replacing $q$ with $q^p$, we obtain
\[
\log_{q^p}(1 + \xi/x) = \sum_{k=1}^\infty (-1)^{k-1} (k-1)_{q^p}! q^{-p{k \choose 2}} \frac {\xi^{[k]_{q^p}}}{x^k}
\]
or, equivalently,
\[
\log_{q^p}(1 + \frac{(p)_q}x \omega) = \sum_{k=1}^\infty (-1)^{k-1} (k-1)_{q^p}!(p)_q^k q^{-p{k \choose 2}} \frac {\omega^{\{k\}_{\mathbb\Delta}}}{x^k}.
\]
We then deduce formula \eqref{logq_eq1} from identity \eqref{qplog}, and \eqref{logq_eq2} follows immediately.
\end{proof}

\begin{prop} \label{BKthet}
The Breuil-Kisin action comes from a $\mathbb\Delta$-Taylor map such that
\[
 \theta((q-1)e_R) = \left(q-1 + \log_q\left(1 + \frac {(p)_q} q\omega\right)\right)e_R.
\]
\end{prop}

\begin{proof}
Recall first that for all $r \in \mathbb N$, we have $q^{p^r} \equiv 1 \mod (p^r)_q$ and $(p)_{q^{p^r}} \equiv p\mod (p^r)_q$.
It follows that
\[
(p^{r+1})_q = (pp^r)_q = (p)_{q^{p^r}}(p^r)_q \equiv p(p^r)_q \mod (p^r)_q^2.
\]
Also, when $0 \leq i \leq p^r$,
\[
(p^{r+1}-ip)_q = (p^{r+1})_q-q^{p^{r+1}-ip}(ip)_q \equiv -q^{-ip}(ip)_q \mod (p^r)_q.
\]
Therefore, whenever $0 \leq k \leq p^r$, we have
\begin{linenomath}
\begin{align*}
(k-1)_{q^p}!{p^r-1 \choose k-1}_{q^p} (p)_q^{k-1}& = \prod_{i=1}^k (p^{r+1}-ip)_q
\\ & \equiv \prod_{i=1}^k (-q^{-ip}(ip)_q) \mod (p^r)_q
\\ & \equiv (-1)^{k-1} (k-1)_{q^p}! q^{-p{k\choose 2}} (p)_q^{k-1} \mod (p^r)_q.
\end{align*}
\end{linenomath}
Now, it follows from lemma \ref{transan} that the Taylor map of $R$ induces a Taylor map on $(p^r)_qR$ and therefore also on $(p^r)_qR/(p^r)_q^2R$.
Then, for all $k>0$,
\begin{linenomath}
\begin{align*}
\partial^{\langle k \rangle}_{\mathbb\Delta}((q-1)(p^r)_q) &= \partial^{\langle k \rangle}_{\mathbb\Delta}(q^{p^r} -1)
\\ & = (k)_{q^p}!{p^r \choose k}_{q^p}q^{p^r-k}(p)_q^k
\\ & = (p^{r+1})_q(k-1)_{q^p}!{p^r-1 \choose k-1}_{q^p}q^{p^r-k}(p)_q^{k-1}
\\ & \equiv (-1)^{k-1} p (k-1)_{q^p}! q^{-p{k\choose 2}-k } (p)_q^{k-1} (p^r)_q \mod (p^r)_q^2.
\end{align*}
\end{linenomath}
Our assertion therefore follows from formula \eqref{logq_eq2} in lemma \ref{logq} and proposition \ref{BKlim}.
\end{proof}

We can recover the formula of proposition \ref{BKthet} from proposition 2.6.10 in \cite{BhattLurie22}: in their notations, we have
\[
A = R\langle \xi \rangle_q \quad \mathrm{with} \quad u = 1 + \xi/q =1 + \frac{(p)_q\omega}q.
\]
By definition, $\log_{\mathbb \Delta}(q^p) := (q-1)e_R$ and they prove that $\log_{\mathbb \Delta}(u^p) = \log_q(u)e_R$.
It follows that
\begin{linenomath}
\begin{align*}
&\theta((q-1)e_R) = \theta( \log_{\mathbb \Delta}(q^p)) =\log_{\mathbb \Delta}(\theta(q^p)) = \theta( \log_{\mathbb \Delta}((q+\xi)^p)) = \theta( \log_{\mathbb \Delta}(q^pu^p))
\\
&= \theta( \log_{\mathbb \Delta}(q^p) + \log_{\mathbb \Delta}(u^p)) = \log_{\mathbb \Delta}(q^p) + \theta(\log_{\mathbb \Delta}(u^p)) = (q-1)e_R + \log_q(u) e_R.
\end{align*}
\end{linenomath}

\section{Absolute differential operators} \label{diffop}

We introduce now the notion of a $\mathbb\Delta$-differential operator.
This will allow us to lift a $\mathbb\Delta$-derivation to a $\mathbb\Delta$-hyperstratification.

As before, we fix a \emph{prime} $p$, a $p$-adically complete ring $W$ and write $R := W[[q-1]]$.

\begin{dfn} \label{defdifo}
Let $M$ and $N$ be two $R$-modules.
A \emph{$\mathbb\Delta$-differential operator} from $M$ to $N$ is an $R$-linear map (for the left structure)
\[
u: L_{\mathbb\Delta}(M) := \widehat{R\langle \omega \rangle}_{\mathbb\Delta} \widehat\otimes'_{R} M \to N.
\]
\end{dfn}

Even if it makes sense in general, we shall mostly apply this definition to complete $R$-modules.

\begin{lem}
If $u$ is a $\mathbb\Delta$-differential operator from an $R$-module $M$ to a torsion-free $R$-module $N$, then $u$ is uniquely determined by its restriction
\[
u_0 : M \to N, \quad s \mapsto u(1 \otimes s).
\]
\end{lem}

\begin{proof}
The $R$-linear map $u$ is uniquely determined by the image of $\omega^{\{n\}_{\mathbb\Delta}} \otimes s$ for $n \in \mathbb N$ and $s \in M$.
In fact, since the $q^r$-analogs of integers are always regular in $R$, and $N$ is assumed to be torsion free, $u$ is uniquely determined by the image of $\xi^{(n)_{q^p}} \otimes s$, which in turn is uniquely determined by the images of $\xi^n \otimes s$ for various $n$.
Now, since $\xi = \theta(q)-q$, we have
\[
\xi^n \otimes s =  \sum_{k=0}^n (-1)^k{n \choose k} q^{k}\theta(q^{n-k}) \otimes s = \sum_{k=0}^n (-1)^k{n \choose k} q^{k} (1 \otimes q^{n-k} s)
\]
and therefore
\[
u(\xi^n \otimes s) = \sum_{k=0}^n (-1)^k{n \choose k} q^{k} u_0(q^{n-k}s). \qedhere
\]
\end{proof}

We will denote by
\[
\mathrm{HDiff}_{\mathbb\Delta}(M, N) := \mathrm{Hom}_R(L_{\mathbb\Delta}(M), N)
\]
the $R$-module of all $\mathbb\Delta$-differential operators from $M$ to $N$ and write $\mathrm{HDiff}_{\mathbb\Delta}(M)$ when $M=N$.
In the case $M=N=R$, we will write $\mathrm{HD}_{\mathbb\Delta}$.
Since $W$ is complete, $\mathrm{HD}_{\mathbb\Delta}$ is simply the dual of $R\langle \omega \rangle_{\mathbb\Delta}$:
\[
\mathrm{HD}_{\mathbb\Delta} \simeq \mathrm{Hom}_R(R\langle \omega \rangle_{\mathbb\Delta}, R).
\]

\begin{dfn}
A $\mathbb\Delta$-differential operator $u$ from $M$ to $N$ has \emph{finite order} at most $n \in \mathbb N$ if it factors through $R\langle \omega\rangle_{\mathbb\Delta}/ {\omega^{\{>n\}_{\mathbb\Delta}}} \otimes'_{R} \widehat M$.
\end{dfn}

For example, for all $k \in \mathbb N$, the operator $\partial_{\mathbb\Delta}^{\langle k \rangle}$ on $R$ lifts (uniquely) to a differential operator $\widetilde \partial_{\mathbb\Delta}^{\langle k \rangle}$ of order at most (and, actually, exactly) $k$ on $R$ and we have
\[
\widetilde \partial_{\mathbb\Delta}^{\langle k \rangle}\left(\omega^{\{l\}_{\mathbb\Delta}}\right) = \left\{\begin{array} l 1 \ \mathrm{if} \ k=l, \\ 0\ \mathrm{otherwise}.\end{array}\right.
\]

It follows that $\left( \widetilde \partial_{\mathbb\Delta}^{\langle k \rangle}\right)_{k \in \mathbb N}$ is a formal basis for $\mathrm{HD}_{\mathbb\Delta}$ in the sense that
\[
\mathrm{HD}_{\mathbb\Delta} = \left\{\sum_{k=0}^\infty \alpha_k \widetilde \partial_{\mathbb\Delta}^{\langle k \rangle}, \quad \alpha_k \in R\right\}.
\]

\begin{dfn} \label{deflin}
The \emph{$R$-linearization} $L_{\mathbb\Delta}(u) : L_{\mathbb\Delta}(M) \to L_{\mathbb\Delta}(N)$ of a $\mathbb\Delta$-differential operator $u$ from $M$ to $N$ is the composite map
\[
L_{\mathbb\Delta}(u) : \widehat{R\langle \omega \rangle}_{\mathbb\Delta} \widehat\otimes'_{R} M \overset {\Delta \otimes \mathrm{Id}}\longrightarrow \widehat{R\langle \omega \rangle}_{\mathbb\Delta} \widehat\otimes'_{R} \widehat{R\langle \omega \rangle}_{\mathbb\Delta} \widehat\otimes'_{R} M \overset {\mathrm{Id} \otimes u} \longrightarrow \widehat{R\langle \omega \rangle}_{\mathbb\Delta} \widehat\otimes'_{R} N.
\]
\end{dfn}

\begin{xmps}
\begin{enumerate}
\item
We recover
\[
L_{\mathbb\Delta}\left(\widetilde \partial^{\langle k \rangle}_{\mathbb\Delta}\right) = L^{\langle k \rangle}_{\mathbb\Delta} : \widehat{R\langle \omega \rangle}_{\mathbb\Delta} \to \widehat{R\langle \omega \rangle}_{\mathbb\Delta}
\]
(which shows compatibility between the different notations).
\item When $u$ is a differential operator of order $0$ so that $u$ is the composition of the augmentation map $e : L_{\mathbb\Delta}(M) \twoheadrightarrow M$ with some $R$-linear map $u_0 : M \to N$, then $L_{\mathbb\Delta}(u) = \mathrm{Id}_{ \widehat{R\langle \omega \rangle}_{\mathbb\Delta}} \otimes_R u_0$.
\end{enumerate}
\end{xmps}

Unlike $\mathbb\Delta$-differential operators of \emph{finite} order, $\mathbb\Delta$-differential operators (of infinite order) are stable under composition.

\begin{dfn}
If $u : L_{\mathbb\Delta}(M) \to N$ and $v : L_{\mathbb\Delta}(P) \to M$ are two $\mathbb\Delta$-differential operators, then their \emph{composition} is
\[
u \widetilde \circ v := u \circ L_{\mathbb\Delta}(v) : L_{\mathbb\Delta}(P) \to N.
\]
\end{dfn}

With this composition, $\mathrm{HDiff}_{\mathbb\Delta}(M)$ becomes a (non commutative) $R$-algebra.
In particular, $\mathrm{HD}_{\mathbb\Delta}$ is an $R$-algebra.

We must insist on the fact that composition of two $\mathbb\Delta$-differential operators of finite order will have infinite order in general.
We were not able to provide an explicit formula for the composition of two $\mathbb\Delta$-differential operators of the form $\widetilde \partial_{\mathbb\Delta}^{\langle n \rangle}$.

Recall from the last example after lemma \ref{locTayl} that if $M$ is any $R$-module, then $L_{\mathbb\Delta}(M)$ comes naturally with a $\mathbb\Delta$-hyperstratification inherited from the left-regular $\mathbb\Delta$-hyperstratification of $\widehat{R\langle \omega \rangle}_{\mathbb\Delta}$.

\begin{lem} \label{linext}
The \emph{$R$-linearization} $L_{\mathbb\Delta}(u) : L_{\mathbb\Delta}(M) \to L_{\mathbb\Delta}(N)$ of a $\mathbb\Delta$-differential operator is compatible with the $\mathbb\Delta$-hyperstratifications on both sides:
\[
\forall k \in \mathbb N, \quad \partial^{\langle k \rangle}_{\mathbb\Delta} \circ L_{\mathbb\Delta}(u) = L_{\mathbb\Delta}(u) \circ \partial^{\langle k \rangle}_{\mathbb\Delta}.
\]

\end{lem}

\begin{proof}
The point is to check that the diagram 
\[
\xymatrix{
L_{\mathbb\Delta}(M) \ar[rr]^{L_{\mathbb\Delta}(u)} \ar[d]^{\theta_M}&& L_{\mathbb\Delta}(N) \ar[d]^{\theta_N} \\ L_{\mathbb\Delta}(M) \widehat\otimes_R \widehat{R\langle \omega \rangle}_{\mathbb\Delta} \ar[rr]^{L_{\mathbb\Delta}(u)\otimes\mathrm {Id}} &&L_{\mathbb\Delta}(N)\widehat\otimes_R \widehat{R\langle \omega \rangle}_{\mathbb\Delta}
}
\]
is commutative.
It is sufficient to insert in the diagram the map obtained by tensoring the Tayor map $\theta : \widehat{R\langle \omega \rangle}_{\mathbb\Delta} \to \widehat{R\langle \omega \rangle}_{\mathbb\Delta} \otimes \widehat{R\langle \omega \rangle}_{\mathbb\Delta}$ corresponding to the left regular structure with the identity of $L_{\mathbb\Delta}(M)$ (and then switching factors):
\[
\xymatrix{
L_{\mathbb\Delta}(M) \ar[r]^-{\Delta \otimes \mathrm{Id}} \ar[d]^{\theta_M}& \widehat{R\langle \omega \rangle}_{\mathbb\Delta} \widehat\otimes'_{R} L_{\mathbb\Delta}(M) \ar[r]^-{\mathrm{Id} \otimes u} \ar[d]^{\theta \otimes \mathrm{Id}} & L_{\mathbb\Delta}(N) \ar[d]^{\theta_N}
\\ L_{\mathbb\Delta}(M) \widehat\otimes_R \widehat{R\langle \omega \rangle}_{\mathbb\Delta} \ar[r]^-{\Delta \otimes \mathrm{Id}} & \widehat{R\langle \omega \rangle}_{\mathbb\Delta} \widehat\otimes'_{R} L_{\mathbb\Delta}(M) \widehat\otimes_R \widehat{R\langle \omega \rangle}_{\mathbb\Delta} \ar[r]^-{\mathrm{Id} \otimes u \otimes\mathrm {Id}}&L_{\mathbb\Delta}(N)\widehat\otimes_R \widehat{R\langle \omega \rangle}_{\mathbb\Delta}.
}\qedhere
\]
\end{proof}

\begin{rmks}
\begin{enumerate}
\item One can prove a little more: linearization defines a functor from the category of $R$-modules and $\mathbb\Delta$-differential operators to the category of $\mathbb\Delta$-hyperstratified modules.
\item
If $u$ is a $\mathbb\Delta$-differential operator, then $\ker L_{\mathbb\Delta}(u)$ inherits a $\mathbb\Delta$-stratification.
This will play a fundamental role.
\end{enumerate}
\end{rmks}

\begin{lem}
If $M$ is a complete $\nabla_{\mathbb\Delta}$-module on $R$, then $\partial_{M,\mathbb\Delta}$ extends to a ${\mathbb\Delta}$-differential operator $\widetilde \partial_{M,\mathbb\Delta}$ of order at most one given by
\[
\widetilde \partial_{M,\mathbb\Delta}\left(\sum_{k=0}^{\infty} \omega^{\{k\}_{\mathbb\Delta}} \otimes s_k\right) = \partial_{M,\mathbb\Delta}(s_0) + \gamma_M(s_1).
\]
\end{lem}

\begin{proof}
We know from lemma \ref{Taylcon} that $\partial_{M,\mathbb\Delta}$ extends to a $\mathbb\Delta$-Taylor map of order $1$: 
\[
M \to M \otimes_R R[\omega]/\omega^{(2)_{\mathbb\Delta}}, \quad s \mapsto s \otimes 1 + \partial_{M,\mathbb\Delta}(s) \otimes \omega
\]
(which is $R$-linear for the right structure).
This map extends uniquely to an $R[\omega]/\omega^{(2)_{\mathbb\Delta}}$-linear map that we can then compose on the right with $\mathrm{Id}_M \otimes \widetilde \partial_{\mathbb\Delta}$ and obtain an $R$-linear map (for the left structure)
\[
\widetilde \partial_{M,\mathbb\Delta} : R[\omega]/\omega^{(2)_{\mathbb\Delta}} \otimes'_R M \to M \otimes_R R[\omega]/\omega^{(2)_{\mathbb\Delta}} \to M.
\]
Now, we compute for $s \in M$,
\[
\widetilde \partial_{M,\mathbb\Delta}(1 \otimes s) = (\mathrm{Id}_M \otimes \widetilde \partial_{\mathbb\Delta})(s \otimes 1 + \partial_{M,\mathbb\Delta}(s) \otimes \omega) = \partial_{M,\mathbb\Delta}(s)
\]
and
\begin{linenomath}
\begin{align*}
\widetilde \partial_{M,\mathbb\Delta}(\omega \otimes s) &= (\mathrm{Id}_M \otimes \widetilde \partial_{\mathbb\Delta})(s \otimes \omega + \partial_{M,\mathbb\Delta}(s) \otimes \omega^2)
\\ &= (\mathrm{Id}_M \otimes \widetilde \partial_{\mathbb\Delta})(s \otimes \omega + \partial_{M,\mathbb\Delta}(s) \otimes \left((2)_{q^p}\omega^{\{2\}} + (q^2-q)\omega\right)
\\ &= s + (q^2-q)\partial_{M,\mathbb\Delta}(s)
\\ & = \gamma_M(s). \qedhere
\end{align*}
\end{linenomath}
\end{proof}

\begin{rmks}
\begin{enumerate}
\item
Be careful that there may exist other ways to lift $\partial_{M,\mathbb\Delta}$ to a $\mathbb\Delta$-differential operator $u$ when $M$ is not torsion-free.
As a toy example, assume that $M$ is reduced in the sense that $(p)_qM = 0$.
It then happens that a $\mathbb\Delta$-connection $\partial_{M,\mathbb\Delta}$ on $M$ is the same thing as an $R$-linear map and we can lift it to a $\mathbb\Delta$-differential operator of order $0$:
\[
u\left(\sum_{k=0}^{\infty} \omega^{\{k\}_{\mathbb\Delta}} \otimes s_k\right) = \partial_{M,\mathbb\Delta}(s_0). 
\]
\item Note however that, in the case of the trivial $\mathbb\Delta$-connection on $R$, the lifting is unique and in particular, our notations are compatible.
\end{enumerate}
\end{rmks}

\begin{dfn}
The \emph{linearized de Rham complex} of a $\nabla_{\mathbb\Delta}$-module $M$ on $R$ is
\[
\left[
\xymatrix{ L_{\mathbb\Delta}(M) \ar[rr]^{L_{M,\mathbb\Delta}} && L_{\mathbb\Delta}(M)}
\right]
\]
with $L_{M,\mathbb\Delta} := L_{\mathbb\Delta}(\widetilde \partial_{M,\mathbb\Delta})$.
\end{dfn}

\begin{rmks}
\begin{enumerate}
\item If $M$ is a complete $\nabla_{\mathbb\Delta}$-module on $R$, then
\begin{linenomath}
\begin{align*}
L_{M,\mathbb\Delta} (\varphi \otimes s) &= ((\mathrm{Id} \otimes \widetilde \partial_{M,\mathbb\Delta}) \circ (\Delta \otimes \mathrm{Id}_M))(\varphi \otimes s)
\\ &= \sum_{k=0}^\infty L^{\langle k \rangle}_{\mathbb\Delta}(\varphi) \otimes \widetilde \partial_{M,\mathbb\Delta}(\omega^{\{k\}_{\mathbb\Delta}} \otimes s)
\\ &= \varphi \otimes \partial_{M,\mathbb\Delta}( s) + L_{\mathbb\Delta}(\varphi) \otimes \gamma_M(s).
\end{align*}
\end{linenomath}
Alternatively, this formula may be rewritten
\begin{equation} \label{LLM}
L_{M,\mathbb\Delta} (\varphi \otimes s) = L_{\mathbb\Delta}(\varphi) \otimes s + \gamma'(\varphi) \otimes \partial_{M,\mathbb\Delta}( s).
\end{equation}
\item 
This last computation shows that $L_{M,\mathbb\Delta}$ is a tensor product $\mathbb\Delta$-derivation (see lemma \ref{tenhom}) on $L_{\mathbb\Delta}(M)$.
Be careful however that $L_{M,\mathbb\Delta}$ is an $R$-linear map which is a $\mathbb\Delta$-derivation for the right structure.
This (right) $\mathbb\Delta$-derivation $L_{M,\mathbb\Delta}$ should not be confused with the (left) $\mathbb\Delta$-derivation $\partial_{\mathbb\Delta}$ of $L_{\mathbb\Delta}(M)$ that does not depend on the $\mathbb\Delta$-connection but only on the underlying $R$-module $M$ (and there also exists an $L_{\mathbb\Delta}$ on $M \otimes_R \widehat{R\langle \omega \rangle}_{\mathbb\Delta}$ which does not depend on the $\mathbb\Delta$-connection either).
\item The linearized de Rham complex is functorial: if $M$ and $N$ are two $\nabla_{\mathbb\Delta}$-modules and $\varphi : M \to N$ is a horizontal map, then the diagram
\[
\xymatrix{L_{\mathbb\Delta}(M) \ar[rr]^{\mathrm{Id} \otimes \varphi} \ar[d]^{L_{M,\mathbb\Delta}} && L_{\mathbb\Delta}(N) \ar[d]^{L_{N,\mathbb\Delta}} \\ L_{\mathbb\Delta}(M) \ar[rr]^{\mathrm{Id} \otimes \varphi} && L_{\mathbb\Delta}(N)
}
\]
is commutative.
\item
It follows from lemma \ref{linext} that if $M$ is a $\nabla_{\mathbb\Delta}$-module on $R$, then
\[
\partial_{\mathbb\Delta} \circ L_{M,\mathbb\Delta} = L_{M,\mathbb\Delta} \circ \partial_{\mathbb\Delta}.
\]
\item
Lemma \ref{linext} also implies that, if $M$ is a $\nabla_{\mathbb\Delta}$-module on $R$, then $\ker L_{M,\mathbb\Delta}$ is naturally endowed with a $\mathbb\Delta$-hyperstratification.
This is functorial in $M$.
\item
As a consequence of the last remark, if $M,M'$ are two $\nabla_{\mathbb\Delta}$-modules on $R$, then there exists a morphism
\[
\mathrm{Hom}_{\nabla}(M,M') \to \mathrm{Hom}_{\nabla}(\ker L_{M,\mathbb\Delta},\ker L_{M',\mathbb\Delta})
\]
or, equivalently (see the remark \ref{tensor}.5), when both $M$ and $\ker L_{M,\mathbb\Delta}$ are inversive,
\begin{equation}\label{H0Hom}
\mathrm H^0_{\mathrm{dR},\mathbb\Delta}(\mathrm{Hom}_{R}(M,M')) \to \mathrm H^0_{\mathrm{dR},\mathbb\Delta}(\mathrm{Hom}_{R}(\ker L_{M,\mathbb\Delta},\ker L_{M',\mathbb\Delta}).
\end{equation}
\end{enumerate}
\end{rmks}

We shall need later the following comparison lemma:

\begin{lem} \label{H0dr}
If $M$ is a complete $\nabla_{\mathbb\Delta}$-module on $R$, then there exists a natural isomorphism
\[
\mathrm H^0_{\mathrm{dR},\mathbb\Delta}(M) \simeq \mathrm H^0_{\mathrm{dR},\mathbb\Delta}(\ker L_{M,\mathbb\Delta}).
\]
\end{lem}

\begin{proof}
Thanks to little Poincar\'e lemma (proposition \ref{little}), the map \eqref{H0Hom} applied to the case where $M$ and $M'$ are replaced with $R$ (which is trivially inversive) and $M$ respectively provides a morphism
\[
\mathrm H^0_{\mathrm{dR},\mathbb\Delta}(M) \to \mathrm H^0_{\mathrm{dR},\mathbb\Delta}(\ker L_{M,\mathbb\Delta}).
\]
By construction, this map is clearly injective and it remains to show that it is surjective.
Thus we start with some $\varphi \in L_{\mathbb\Delta}(M)$ such that both $L_{M,\mathbb\Delta}(\varphi) = 0$ and $\partial_{\mathbb\Delta}(\varphi) = 0$.
It follows from little Poincar\'e lemma again that there exists an exact sequence (since $M$ is assumed to be complete):
\[
0 \longrightarrow M \longrightarrow L_{\mathbb\Delta}(M) \overset{\partial_{\mathbb\Delta}}\longrightarrow L_{\mathbb\Delta}(M) \longrightarrow 0.
\]
Therefore, condition $\partial_{\mathbb\Delta}(\varphi) = 0$ means that $\varphi = 1 \otimes s$ for some $s \in M$.
Now, condition $L_{M,\mathbb\Delta}(\varphi) = 0$ boils down to $\partial_{M,\mathbb\Delta}(s) = 0$ which shows that $s \in \mathrm H^0_{\mathrm{dR},\mathbb\Delta}(M)$.
\end{proof}

\begin{prop} \label{exacM}
If a complete $R$-module $M$ is endowed with a $\mathbb\Delta$-hyperstratification, then there exists a split exact sequence
\[
0 \longrightarrow M \overset{\tau_M \circ \theta_M} \longrightarrow L_{\mathbb\Delta}(M) \overset{L_{M,\mathbb\Delta}}\longrightarrow L_{\mathbb\Delta}(M) \longrightarrow 0.
\]
\end{prop}

\begin{proof}
It follows from little Poincar\'e lemma (proposition \ref{little}) that there exists a split exact sequence
\[
0 \longrightarrow M \longrightarrow M \widehat\otimes_R \widehat{R\langle \omega \rangle}_{\mathbb\Delta} \overset {\mathrm{Id} \otimes L_{\mathbb\Delta}} \longrightarrow M \widehat\otimes_R \widehat{R\langle \omega \rangle}_{\mathbb\Delta} \longrightarrow 0.
\]
We may then use lemma \ref{maglem} (case $N = R$) which states that there exists an isomorphism
\[
\epsilon_M^{-1} : M \widehat\otimes_R \widehat{R\langle \omega \rangle}_{\mathbb\Delta} \simeq L_{\mathbb\Delta}(M)
\]
compatible with the $\mathbb\Delta$-hyperstratifications and obtain a split exact sequence as in the statement.
It remains to identify both maps.
The first one is obtained by composing the inclusion with $\epsilon_M^{-1} = \tau_M \circ \epsilon_M \circ \tau_M$ and is therefore equal to $\tau_M \circ \theta_M$ as asserted.
Thus, we are left with showing that
\[
\epsilon_M \circ L_{M,\mathbb\Delta} =(\mathrm{Id}_M \otimes L_{\mathbb\Delta}) \circ \epsilon_M.
\]
On an element of the form $1 \otimes s$ with $s \in M$, we can use equality \eqref{Lpart} in the case $j=1$.
We have
\begin{linenomath}
\begin{align*}
\theta_M(\partial_{M,\mathbb\Delta}(s)) & =\sum_{k=0}^{\infty} \partial^{\langle k \rangle}_{M,\mathbb\Delta}(\partial_{M,\mathbb\Delta}(s))) \otimes \omega^{\{k\}_{\mathbb\Delta}}
\\ &
= \sum_{k=0}^{\infty} \partial^{\langle k \rangle}_{M,\mathbb\Delta}(s) \otimes L_{\mathbb\Delta}\left(\omega^{\{k\}_{\mathbb\Delta}}\right)
\\ &=(\mathrm{Id}_M \otimes L_{\mathbb\Delta})\left(\sum_{k=0}^{\infty} \partial^{\langle k \rangle}_{M,\mathbb\Delta}(s) \otimes \omega^{\{k\}_{\mathbb\Delta}}\right)
\\ &=(\mathrm{Id}_M \otimes L_{\mathbb\Delta})(\theta_M(s)).
\end{align*}
\end{linenomath}
In general, using equality \eqref{LLM}, we have for $\varphi \otimes s \in L_{\mathbb\Delta}(M)$,
\begin{linenomath}
\begin{align*}
\epsilon_M(L_{M,\mathbb\Delta}(\varphi \otimes s)) & = \epsilon_M(L_{\mathbb\Delta}(\varphi) \otimes s + \gamma'(\varphi) \otimes \partial_{M,\mathbb\Delta}( s))
\\ & = L_{\mathbb\Delta}(\varphi)\theta_M(s) + \gamma'(\varphi) \theta_M(\partial_{M,\mathbb\Delta}( s))
\\ &= L_{\mathbb\Delta}(\varphi)\theta_M(s) + \gamma'(\varphi) (\mathrm{Id}_M \otimes L_{\mathbb\Delta})(\theta_M(s))
\\ &= \sum_{k=0}^{\infty} \partial^{\langle k \rangle}_{M,\mathbb\Delta}(s) \otimes L_{\mathbb\Delta}(\varphi)\omega^{\{k\}_{\mathbb\Delta}} 
\\&+ \sum_{k=0}^{\infty} \partial^{\langle k \rangle}_{M,\mathbb\Delta}(s) \otimes \gamma'(\varphi) L_{\mathbb\Delta}(\omega^{\{k\}_{\mathbb\Delta}})
\\ &= \sum_{k=0}^{\infty} \partial^{\langle k \rangle}_{M,\mathbb\Delta}(s) \otimes L_{\mathbb\Delta}(\varphi\omega^{\{k\}_{\mathbb\Delta}}) 
\\ & = (\mathrm{Id}_M \otimes L_{\mathbb\Delta}) (\epsilon_M(\varphi \otimes s)).
\end{align*}
\end{linenomath}
We have used the fact that $L_{\mathbb\Delta}$ is a $\mathbb\Delta$-derivation for the right structure so that $L_{\mathbb\Delta}(\varphi\psi) = L_{\mathbb\Delta}(\varphi)\psi + \gamma'(\varphi) L_{\mathbb\Delta}(\psi)$.
\end{proof}

In order to characterize explicitly the essential image of the forgetful functor from the category of finite projective $\mathbb\Delta$-hyperstratified modules to the category of $\nabla_{\mathbb\Delta}$-modules, we introduce the following general notion:

\begin{dfn}
Let $M$ be a module over an adic ring $R$.
A map $u : M \to M$ is said to be \emph{topologically quasi-nilpotent} (resp.\ \emph{weakly quasi-nilpotent}) if for all (resp.\ if there exists an) ideal of definition $I$ such that
\[
\forall s \in M, \exists k \in \mathbb N, \quad u^k(s) \in IM.
\]
\end{dfn}

When the topology is trivial, then both conditions are equivalent and one usually says \emph{quasi-nilpotent}.

\begin{rmks}
\begin{enumerate}
\item There exists the stronger notion of a \emph{topologically nilpotent} (resp.\ \emph{weakly nilpotent}) endomorphism, which is obtained by reversing the quantifiers:
\[
\exists k \in \mathbb N, \forall s \in M, \quad u^k(s) \in IM.
\]
Again, in the case of the trivial topology, both conditions are equivalent and it means that $u$ is \emph{nilpotent}.
\item When $u$ is $R$-linear, the epithets ``topologically'' and ``weakly'' are equivalent.
If moreover, $M$ is finite, then all four conditions are equivalent.
\item Since we will also need to consider topological rings that are not adic, let us mention that the notion of a \emph{topologically quasi-nilpotent} map $u : M \to M$ makes sense whenever $(M,0)$ is any pointed topological space: it simply means that
\[
\forall s \in M, \quad \lim_{k\to \infty} u^k(s) = 0.
\]
\end{enumerate}
\end{rmks}

\begin{dfn}
A $\nabla_{\mathbb\Delta}$-module $M$ is said to be \emph{topologically/weakly (quasi-) nilpotent} if this is the case for the endomorphism $\partial_{M,\mathbb\Delta}$ when $p$ is odd and for $\partial_{M,\mathbb\Delta}^2 - \partial_{M,\mathbb\Delta}$ when $p=2$.
\end{dfn}

\begin{xmps}
\begin{enumerate}
\item
The trivial $\nabla_{\mathbb\Delta}$-module is weakly nilpotent because $\partial_{\mathbb\Delta} \equiv 0 \mod (p)_q$.
The same conclusion holds for our running examples $F_n \simeq (p)_q^nR$, $G_n \simeq (p)_q^nR/(p)_q^{n+1}R$ as well as $R\{n\}$.
\item
If $p$ is odd, then a $\nabla_{\mathbb\Delta}$-module $M$ is weakly quasi-nilpotent if and only if
\[
\forall s \in M, \exists k \in \mathbb N, \quad \partial^{k}_{M,\mathbb\Delta}(s) \equiv 0 \mod (p,q-1)
\]
because $(p,q-1)$ is the biggest ideal of definition.
When $p=2$, the condition reads
\[
\forall s \in M, \exists k \in \mathbb N, \quad (\partial_{M,\mathbb\Delta}^2 - \partial_{M,\mathbb\Delta})^k(s) \equiv 0 \mod (p,q-1).
\]
\item For a $\nabla_{\mathbb\Delta}$-module $M$, the following are equivalent:
\begin{enumerate}
\item $M$ is weakly quasi-nilpotent,
\item $M$ is topologically quasi-nilpotent modulo $(p)_q$,
\item $M$ is topologically quasi-nilpotent modulo $(q-1)$,
\item $M$ is topologically quasi-nilpotent modulo $p$.
\end{enumerate}
\item If a \emph{finite} $\nabla_{\mathbb\Delta}$-module $M$ is weakly quasi-nilpotent, then it is actually weakly nilpotent (and we shall therefore remove the prefix in this case).
\item The property for a $\nabla_{\mathbb\Delta}$-module of being (weakly) quasi-nilpotent is stable under tensor product and internal Hom.
\end{enumerate}
\end{xmps}

\begin{lem} \label{techlem}
A finite (resp. finite projective) $\nabla_{\mathbb\Delta}$-module $M$ is weakly nilpotent if and only if the composite map
\begin{equation} \label{canmap}
\ker\left( L_{M,\mathbb\Delta}\right) \hookrightarrow L_{\mathbb\Delta}(M) \overset {e} \twoheadrightarrow M
\end{equation}
is surjective (resp. bijective).
\end{lem}

\begin{proof}
Let us first assume that $p$ is odd.

If
\[
\varphi := \sum_{k=0}^\infty \omega^{\{k\}}_{\mathbb\Delta} \otimes s_k \in L_{\mathbb\Delta}(M),
\]
then it follows from proposition \ref{modp} that
\begin{align} \label{modeq}
L_{M,\mathbb\Delta} (\varphi) &\equiv \sum_{k=0}^\infty \left(\omega^{\{k\}}_{\mathbb\Delta} \otimes \partial_{M,\mathbb\Delta}( s_k) + \omega^{\{k-1\}}_{\mathbb\Delta} \otimes s_k \right) \mod (p,q-1) \nonumber
\\ &\equiv \sum_{k=0}^\infty \omega^{\{k\}}_{\mathbb\Delta} \otimes (\partial_{M,\mathbb\Delta}( s_k) + s_{k+1}) \mod (p,q-1).
\end{align}
Therefore,
\begin{align} \label{modeq2}
&L_{M,\mathbb\Delta}(\varphi) \equiv 0 \mod (p,q-1) \nonumber
\\ \Leftrightarrow\ &\forall k \geq 0, \quad s_{k+1} \equiv -\partial_{M,\mathbb\Delta}( s_k) \mod (p,q-1) \nonumber
\\ \Leftrightarrow\ &\forall n \geq 0, \quad s_n \equiv (-1)^n\partial^n_{M,\mathbb\Delta}(s_0) \mod (p,q-1).
\end{align}
If $M$ is finite, it follows from Nakayama lemma that the map \eqref{canmap} that sends $\varphi$ to $s_0$ is surjective if and only if $\partial_{M,\mathbb\Delta}$ is weakly nilpotent.

Assume now that the previous equivalent conditions hold and that $M$ is finite projective.
We can \emph{choose} an $R$-linear section of the surjective map \eqref{canmap} and then a complement $N$ for $M$ inside $L_{\mathbb\Delta}(M)$ (with respect to this section).
It is then sufficient to show that the map $N \to L_{\mathbb\Delta}(M)$ induced by $L_{M,\mathbb\Delta}$ is bijective.
Since $N$ is completely flat, we can use derived Nakayama lemma and check our assertion modulo $(p,q-1)$.
It is therefore sufficient to show that the sequence
\[
0 \longrightarrow M \longrightarrow L_{\mathbb\Delta}(M) \overset {L_{M,\mathbb\Delta}} \longrightarrow L_{\mathbb\Delta}(M) \longrightarrow 0
\]
is exact modulo $(p,q-1)$.
Left exactness follows from equivalence \eqref{modeq2}.
Right exactness is obtained from congruence \eqref{modeq} as follows.
Given
\[
\psi := \sum_{k=0}^\infty \omega^{\{k\}}_{\mathbb\Delta} \otimes t_k \in L_{\mathbb\Delta}(M),
\]
one sets $s_k := \sum_{i=0}^{k-1} \partial^i_{M,\mathbb\Delta}(t_{k-i-1})$ and $\varphi$ is then well defined since we assumed that the $\mathbb\Delta$-connection is weakly nilpotent.

When $p=2$, the formulas have to be modified a bit and we fall onto the following equivalence:
\[
L_{M,\mathbb\Delta} (\varphi) \equiv \sum_{k=0}^\infty \omega^{\{k\}}_{\mathbb\Delta} \otimes (\partial_{M,\mathbb\Delta}( s_k) + s_{k+1} +ks_k) \mod (2,q-1),
\]
so that
\begin{linenomath}
\[
\begin{aligned}
&L_{M,\mathbb\Delta}(\varphi) \equiv 0 \mod (2,q-1)
\\ \Leftrightarrow\ &\forall k \geq 0, \quad s_{k+1} \equiv (\partial_{M,\mathbb\Delta} - k)(s_k)\mod (2,q-1)
\\ \Leftrightarrow\ &\forall n \geq 0, \quad 
\left\{\begin{array} {ll} s_{2n} &\equiv (\partial_{M,\mathbb\Delta}^2 - \partial_{M,\mathbb\Delta})^n(s_0) \\ s_{2n+1} &\equiv (\partial_{M,\mathbb\Delta}^2 - \partial_{M,\mathbb\Delta})^n\partial_{M,\mathbb\Delta}(s_0)\end{array}\right. \mod (2,q-1). \qedhere
\end{aligned}
\]
\end{linenomath}
\end{proof}

\begin{rmk}\phantomsection \label{p2remark}
The difference between the cases $p$ odd and $p=2$ is a special instance of a standard phenomenon in the study of Breuil-Kisin prisms.
It corresponds to the fundamental group (see example 9.6 of \cite{BhattLurie22b}) being respectively additive or multiplicative.
\end{rmk}

\begin{thm} \label{invlem}
Let $W$ be a $p$-adically complete ring and $R := W[[q-1]]$.
Then, the category of finite projective $R$-modules endowed with a $\mathbb\Delta$-hyperstratification is isomorphic to the category of finite projective weakly nilpotent $\nabla_{\mathbb\Delta}$-modules on $R$.
\end{thm}

\begin{proof}
Let us first show that if a complete $R$-module $M$ is endowed with a $\mathbb\Delta$-hyperstratification, then the natural map $\ker L_{M,\mathbb\Delta} \to M$ is an isomorphism compatible with the $\mathbb\Delta$-hyperstratifications.
Thanks to proposition \ref{exacM}, it is sufficient to prove that the composite map
\[
\tau_M \circ \theta_M : M \longrightarrow L_{\mathbb\Delta}(M)
\]
is compatible with the stratifications.
After a flip, this is equivalent to the cocycle condition.
In particular, if $M$ is finite projective, then it follows from lemma \ref{techlem} that $M$ is weakly nilpotent.
We also know from corollary \ref{fulfait} that the forgetful functor is fully faithful.
In order to conclude, it is therefore sufficient to show that the functor $M \mapsto \ker L_{M,\mathbb\Delta}$ is fully faithful on weakly nilpotent finite projective $\nabla_{\mathbb\Delta}$-modules.
In other words, we have to show that, if $M$ and $M'$ are two weakly nilpotent finite projective $\nabla_{\mathbb\Delta}$-modules, then the canonical map
\[
\mathrm{Hom}_{\nabla}(M, M') \to \mathrm{Hom}_{\nabla}(\ker L_{M,\mathbb\Delta}, \ker L_{M',\mathbb\Delta})
\]
is bijective.
This will follow from lemma \ref{H0dr} applied to $\mathrm{Hom}_{R}(M, M')$ once we have shown that there exists a canonical horizontal isomorphism
\[
\ker L_{\mathrm{Hom}_R(M,M'),\mathbb\Delta} \simeq \mathrm{Hom}_R(\ker L_{M,\mathbb\Delta}, \ker L_{M',\mathbb\Delta}).
\]
In general, if $M$ and $M'$ are two finite projective $R$-modules, then there exists an obvious map which is automatically horizontal (for $\partial_{\mathbb\Delta}$):
\begin{equation}\label{homL}
\xymatrix@R=0cm{
L_{\mathbb\Delta}(\mathrm{Hom}_R(M,M')) \ar[r] & \mathrm{Hom}_{\widehat {R\langle \omega \rangle}_{\mathbb\Delta}}(L_{\mathbb\Delta}(M),L_{\mathbb\Delta}(M'))
\\ \varphi \otimes u \ar@{|->}[r] & \varphi \otimes u : \psi \otimes s \mapsto \varphi\psi \otimes u(s).
}
\end{equation}
Let us now consider the canonical map
\begin{equation} \label{maplem}
\mathrm{Hom}_R(\widehat {R\langle \omega \rangle}_{\mathbb\Delta},\widehat {R\langle \omega \rangle}_{\mathbb\Delta}) \otimes'_R \mathrm{Hom}_R(M,M') \to \mathrm{Hom}_{R}(\widehat {R\langle \omega \rangle}_{\mathbb\Delta} \otimes'_R M,\widehat {R\langle \omega \rangle}_{\mathbb\Delta} \otimes_R' M')
\end{equation}
where $\widehat {R\langle \omega \rangle}_{\mathbb\Delta}$ is seen as an $R$-module via the Taylor map.
If $M$ and $M'$ are both endowed with a $\mathbb\Delta$-connection and we endow $\widehat {R\langle \omega \rangle}_{\mathbb\Delta}$ with $L_{\mathbb\Delta}$, it follows from lemma \ref{homtens} that the map \eqref{maplem} is horizontal.
As a consequence, the induced map \eqref{homL} is also horizontal for the corresponding $\mathbb\Delta$-connections coming from the connections on $M$ and $M'$ and using $L_{\mathbb\Delta}$ on $\widehat {R\langle \omega \rangle}_{\mathbb\Delta}$.
In particular, it induces a map between the $H^0_{\mathrm{dR},\mathbb\Delta}$ (for these $\mathbb\Delta$-connections) on both sides.
In other words, if we assume that $L_{\mathrm{Hom}_R(M,M'),\mathbb\Delta}(\varphi \otimes u) = 0$ then the diagram
\[
\xymatrix{L(M) \ar[r]^{\varphi \otimes u} \ar[d]^{L_{M,\mathbb\Delta}} & L(M') \ar[d]^{L_{M',\mathbb\Delta}} \\L(M) \ \ar[r]^{\varphi \otimes u} & L(M') }
\]
is commutative and $\varphi \otimes u$ therefore induces a map $\ker L_{M,\mathbb\Delta} \to \ker L_{M',\mathbb\Delta}$.
Thus, we see that the map \eqref{homL} induces a horizontal (for $\partial_{\mathbb\Delta}$) map on the kernels:
\[
\ker L_{\mathrm{Hom}_R(M,M'),\mathbb\Delta} \to \mathrm{Hom}_R(\ker L_{M,\mathbb\Delta}, \ker L_{M',\mathbb\Delta}).
\]
It remains to show that this last map is bijective when $M$ and $M'$ are weakly nilpotent.
But, in that case, it follows from lemma \ref{techlem} that the canonical maps $\ker L_{M,\mathbb\Delta} \to M$ and $\ker L_{M',\mathbb\Delta} \to M'$ are bijective.
We are therefore reduced to checking that
\[
\ker L_{\mathrm{Hom}_R(M,M'),\mathbb\Delta} \simeq \mathrm{Hom}_R(M, M').
\]
But this follows again from lemma \ref{techlem} since $\mathrm{Hom}_R(M, M')$ is then also finite projective and weakly nilpotent.
\end{proof}

We shall finish this section with a technical result that will be needed in the next one.
Let us write $\sigma := g_{-1}$ so that $\sigma$ is the unique endomorphism of the ring $R$ such that $\sigma(q) = q^{-1}$.
Recall that it extends uniquely to $R[[\xi]]$ if we require that $\sigma(q+\xi) = q+\xi$ and that
\begin{equation} \label{sigf}
\sigma(\xi) = \xi + q-q^{-1} = \xi + q^{-1}(2)_q(q-1).
\end{equation}
Assume now that $p=2$ (otherwise the next assertion is \emph{not} true).
One can then show as in proposition \ref{sigcont} that $\sigma$ extends uniquely to an endomorphism of $\widehat{R\langle \omega \rangle}_{\mathbb\Delta}$ and that  $\sigma(\omega) = q\omega + q -1$.

\begin{lem} \label{p2case}
Assume $p=2$.
If $M$ is a weakly nilpotent finite projective $\nabla_{\mathbb\Delta}$-module, then there exists a unique $\sigma$-linear horizontal map on $M$ that reduces to the identity modulo $q-1$.
\end{lem}

\begin{proof}
Uniqueness follows from lemma \ref{injmod}.
Now, by construction, $\sigma$ and $\gamma$ commute on $\widehat{R\langle \omega \rangle}_{\mathbb\Delta}$.
In other words, $\sigma$ is an endomorphism of $\gamma$-module or, equivalently, a horizontal map.
On the other hand, since $\sigma$ is linear for the right structure of $\widehat{R\langle \omega \rangle}_{\mathbb\Delta}$, we may consider the endomorphism $\sigma \otimes \mathrm{Id}_M$ of $L_{\mathbb\Delta}(M)$.
By construction, it is horizontal and commutes with $L_{M,\mathbb\Delta}$.
It induces therefore a $\sigma$-linear horizontal endomorphism of $\ker(L_{M,\mathbb\Delta})$ and we conclude with lemma \ref{techlem} which tells us that $\ker(L_{M,\mathbb\Delta}) \simeq M$.
\end{proof}

\section{Application to prismatic crystals} \label{prissec}

We will prove here our main results that relate absolute prismatic vector bundles to absolute calculus.

We let $k$ be a perfect field of positive characteristic $p$ and denote by $W$ its ring of Witt vectors.
We set $R := W[[q-1]]$ so that $\overline R := R/(p)_q = W[\zeta]$ where $\zeta$ (the class of $q$) is a primitive $p$th root of unity.
The ring $R$, endowed with the unique $\delta$-structure such that $\delta(q) = 0$ and with the ideal generated by $(p)_q$, is a bounded prism.

We recall that the absolute prismatic site\footnote{Unlike Bhatt and Scholze, we follow the classical notations for ``cristalline'' sites and topos.} $\mathbb \Delta$ is the opposite to the category of all bounded prisms $(B,J)$ endowed with the formally flat topology\footnote{A morphism is a covering if it is faithfully flat modulo any ideal of definition.}.
Then, the prismatic site $\mathbb \Delta(\overline R)$ is the fiber over $\overline R$ of the forgetful functor $(B,J) \mapsto \overline B := B/J$ from $\mathbb \Delta$ to the (opposite) category of $p$-adic rings.
In other words, an object of $\mathbb \Delta(\overline R)$ is a prism $(B,J)$ endowed with a morphism $\overline R \to \overline B$ .
We will write $\overline R_{\mathbb \Delta}$ for the corresponding topos.
We shall also consider later the prismatic site $\mathbb \Delta(/(R, (p)_q))$ over $(R, (p)_q)$ made of bounded prisms $(B,J)$ endowed with a morphism of prisms $(R, (p)_q) \to (B,J)$. For simplicity, we will denote it by $\mathbb \Delta(/R)$ and write $/R_{\mathbb \Delta}$ for the corresponding topos.

Recall that if $(B,d)$ is an oriented bounded prism and $q \in B$, then we introduced in definition 3.2 of \cite{GrosLeStumQuiros23b} the ring of prismatic polynomials $B[\omega]^{\delta}_d$, which is universal for the property $\delta(q+d\omega) =0$.
It is important to remember that, in the case $B=R$ and $d = (p)_q$, there is no need to introduce this object because\footnote{This is implicit in the statement of \cite{GrosLeStumQuiros22b}, theorem 5.2.} $R[\omega]^{\delta}_d \simeq R\langle \omega \rangle_{\mathbb\Delta}$.

The next statement is the absolute analog of \cite{GrosLeStumQuiros23b}, corollary 3.6:

\begin{prop} \label{cova}
If $(B,d)$ is an oriented bounded prism over $\overline R$ and $q \in B$ is a lifting of the image of $\zeta$ in $\overline B$, then the completion $\widehat{B[\omega]}^{\delta}_d$ of the ring of prismatic polynomials is the product of $(R,(p)_q)$ and $(B,d)$ in $\mathbb \Delta(\overline R)$.
\end{prop}

\begin{proof}
It follows from proposition 3.3 in \cite{GrosLeStumQuiros23b} that $\widehat{B[\omega]}^{\delta}_d$ is a bounded prism.
Now, the universal property of the Witt vectors implies that the composite map $W \to \overline R \to \overline B$ lifts uniquely to a morphism of rings $W \to B$.
We set $Q := q + d \omega \in B[\omega]^{\delta}_d$ so that $Q$ is a lift of $\zeta$ in $B[\omega]^{\delta}_d$.
Since we always have $(Q-1)^p \in (p, (p)_{Q})$, and $(p)_{Q} \in dB[\omega]^{\delta}_d$ by construction, the element $Q-1$ is topologically nilpotent in $B$.
It follows that the composite map $\overline R \to \overline B \to \overline {B[\omega]}^{\delta}_d$ lifts to
\begin{equation} \label{TaylB}
\theta : R \to \widehat{B[\omega]}^{\delta}_d, \quad q \mapsto Q=q+d\omega.
\end{equation}
Since, by definition, $\delta(Q)= 0$, this is a morphism of $\delta$-rings.
Using the fact that $(p)_{Q} \in dB[\omega]^{\delta}_d$ again, we see that this is a morphism of prisms.
The universal property of the product then follows from the universal properties of $B[\omega]^{\delta}_d$ and completion.
\end{proof}

We shall again call the map \eqref{TaylB} a Taylor map and write $\otimes'_R$ when we use this structure to consider $\widehat{B[\omega]}^{\delta}_d$ as an $R$-algebra.
This is compatible with our previous notation when $B = R$.
Note that the above choice of $q \in B$ also defines a map
\[
R \to B \to \widehat{B[\omega]}^{\delta}_d
\]
but this is \emph{not} a morphism of $\delta$-rings.

The next statement is the analog of \cite{GrosLeStumQuiros23b}, proposition 3.7:

\begin{cor} \label{covabs}
The bounded prism $(R, (p)_q)$ is a covering of (the final object of) $\mathbb \Delta(\overline R)$.
\end{cor}

\begin{proof}
First of all, a bounded prism is always locally oriented.
Moreover it is shown in proposition 3.3 of \cite{GrosLeStumQuiros23b} that, if $(B,d)$ is an oriented bounded prism over $\overline R$, then $\widehat{B[\omega]}^{\delta}_d$ is formally faithfully flat over $B$.
Our assertion is therefore a formal consequence of proposition \ref{cova}.
\end{proof}

\begin{cor} \label{prodpr}
In the category $\mathbb \Delta(\overline R)$, we have for all $n \in \mathbb N$,
\[
\underbrace{(R, (p)_q) \times \cdots \times (R, (p)_q)}_{n+1\ \mathrm{times}} = \left(\underbrace{\widehat {R\langle\omega\rangle}_{\mathbb\Delta} \otimes'_R \cdots \otimes'_{R} \widehat {R\langle\omega\rangle}_{\mathbb\Delta}}_{n\ \mathrm{times}}, (p)_q\right)
\]
\end{cor}

\begin{proof}
This follows from theorem \ref{prismenv} and lemma \ref{prismdiag} (or equivalently from theorem 5.2 in \cite{GrosLeStumQuiros22b} and proposition \ref{cova}).
\end{proof}

\begin{dfn}
A \emph{prismatic crystal} on $\overline R$ is a cartesian section of the fibered category over $\mathbb \Delta(\overline R)$ whose fiber over $(B,J)$ is the category of derived complete $B$-modules.
\end{dfn}

A locally finite free module on a ringed site is also called a \emph{vector bundle}.
If we endow $\mathbb \Delta(\overline R)$ with the sheaf of rings $\mathcal O_{\mathbb \Delta(\overline R)}$ that sends $(B,J)$ to $B$, then any vector bundle on $\mathbb \Delta(\overline R)$ is a prismatic crystal.

Recall that we introduced in definition \ref{hyperst} the notion of a $\mathbb\Delta$-hyperstratification.

\begin{thm} \label{eqTayl}
Let $W$ be the ring of Witt vectors of a perfect field $k$ of characteristic $p >0$ and $\zeta$ a primitive $p$th root of unity.
Then there exists an equivalence between the category of prismatic crystals on $W[\zeta]$ and the category of complete $W[[q-1]]$-modules endowed with a $\mathbb\Delta$-hyperstratification.
\end{thm}

\begin{proof}
This formally follows from corollaries \ref{covabs} and \ref{prodpr}.
\end{proof}

\begin{xmps}
\begin{enumerate}
\item
If we let $I_{\mathbb \Delta}(B,J) := J$, then $I_{\mathbb \Delta}^n$ corresponds to $F_n \simeq (p)_q^nR$.
\item
The prismatic crystal $I_{\mathbb \Delta}^n/I_{\mathbb \Delta}^{n+1}$ corresponds to $G_n \simeq (p)_q^nR/ (p)_q^{n+1}R$.
\item
The Breuil-Kisin prismatic crystal $\mathcal O_{\mathbb \Delta}\{1\}$ of example 2.3 in \cite{Bhatt23} corresponds to our $R\{1\}$ (and its $n$th power to $R\{n\}$ if $n \in \mathbb Z$).
\end{enumerate}
\end{xmps}

We shall now extend to prismatic crystals the linearization process introduced in section \ref{stratder} for $R$-modules.
This is an absolute version of \cite{GrosLeStumQuiros23b}, section 4.

The prismatic site $\mathbb \Delta(/R)$ of bounded prisms over $(R,(p)_q)$ also comes with a structural ring $\mathcal O_{\mathbb \Delta(/R)}$.
There exists a final morphism of ringed topoi
\[
e_{R} : \left( /R_{\mathbb{\Delta}}, \mathcal O_{\mathbb \Delta(/R)}\right) \to ( \mathbb S\mathrm{ets}, R)
\]
and a localization morphism of topoi 
\[
j_{R} : /R_{\mathbb{\Delta}} \to \overline R_{\mathbb{\Delta}}.
\]

\begin{dfn}
If $M$ is an $R$-module, then the \emph{linearization} of $M$ is the $\mathcal O_{\mathbb \Delta(\overline R)}$-module
\[
\mathcal L(M) := j_{R*}e_{R}^*(M).
\]
\end{dfn}

It follows from the second part of the next statement that this teminology is compatible with that of section \ref{stratder} in the sense that $\mathcal L(M)_R \simeq L_{\mathbb\Delta}(M)$.

\begin{prop} \label{linprop}
Let $M$ be a finite free $R$-module.
Then,
\begin{enumerate}
\item
$\mathrm R\Gamma(\overline R_{\mathbb{\Delta}}, \mathcal L(M)) = M$.
\item \label{tens1}
If $(B,d)$ is an oriented bounded prism on $\overline R$, then
\[
\mathcal L(M)_{B} = \widehat{B[\omega]}^{\delta}_d \otimes_{R}' M.
\]
\item \label{tens2}
 If $B \to C$ is a morphism of bounded prisms over $\overline R$, then
\[
\mathcal L(M)_{C} = C \widehat \otimes_{B}\mathcal L(M)_{B} = C \widehat \otimes^\mathrm{L}_{B}\mathcal L(M)_{B}.
\]
\end{enumerate}
\end{prop}

\begin{proof}
It follows from proposition \ref{cova} that, if $(B,d)$ is an oriented bounded prism on $\overline R$, then $j_{R}^{-1}(B) =\widehat{B[\omega]}^{\delta}_d$ is the completion of the prismatic polynomial ring.
Our assertions are then shown exactly as in \cite{GrosLeStumQuiros23b}, corollary 4.5 and lemma 4.6.
\end{proof}

One can then state the linearization lemma (see also lemma \ref{maglem}):

\begin{lem} \label{getout}
If $E$ is a vector bundle on $\mathbb \Delta(\overline R)$ and $M$ is a finite free $R$-module, then
\[
\mathcal L(E_{R} \otimes_{R} M) \simeq E \otimes_{\mathcal O_{\mathbb \Delta(\overline R)}} \mathcal L(M).
\]
\end{lem}

\begin{proof}
This is shown exactly as in \cite{GrosLeStumQuiros23b}, lemma 4.7.
\end{proof}

If $(B,d)$ is an oriented bounded prism on $\overline R$, then there exists a canonical \emph{comultiplication} map
\[
\Delta_{B} : \widehat{B[\omega]}^{\delta}_d \to \widehat{B[\omega]}^{\delta}_d \widehat \otimes'_{R} \widehat{R\langle\omega\rangle}_{\mathbb\Delta}.
\]
In term of prisms, this is the map $B \times R \times R \to B \times R$ that forgets the middle term.
Of course, this is compatible with our previous notion of comultiplication and we need to generalize a bit definition \ref{deflin}:

\begin{dfn}
The \emph{linearization} on an oriented bounded prism $(B,d)$ over $\overline R$ of a $\mathbb\Delta$-differential operator $u$ between finite free $R$-modules $M$ and $N$ is the $B$-linear map
\[
\xymatrix{\mathcal L(M)_{B} \ar[rr]^{\mathcal L(F)_{B}} \ar@{=}[d]&& \mathcal L(N)_{B}\ar@{=}[d]\\
\widehat{B[\omega]}^{\delta}_d \otimes'_{R} M \ar[rd]^{\Delta_{B} \otimes \mathrm{Id}_{M}} && \widehat{B[\omega]}^{\delta}_d \otimes'_{R} N \\ & \widehat{B[\omega]}^{\delta}_d \widehat \otimes_{R}' \widehat{R\langle\omega\rangle}_{\mathbb\Delta} \otimes'_{R} M \ar[ru]^{\mathrm{Id} \otimes u}
}
\]
\end{dfn}

\begin{prop} \label{linfunc}
\begin{enumerate}
\item
Linearization provides a functor from the category of finite free $R$-modules and $\mathbb\Delta$-differential operators to the category of vector bundles on $\mathbb \Delta(\overline R)$.
\item
 If $D$ is a $\mathbb\Delta$-derivation, then
\[
\mathrm R\Gamma(\overline R_{\mathbb \Delta}, \mathcal L(\widetilde D)) =D.
\]
\end{enumerate}
\end{prop}

\begin{proof}
It follows from assertions \ref{tens1} and \ref{tens2} of proposition \ref{linprop} that the construction is compatible with completed base extension and composition.
The first assertion follows.
The other one results from the definition and functoriality in proposition \ref{linprop}.1.
\end{proof}

\begin{prop}[Prismatic Poincar\'e lemma] \label{poinc}
If $E$ is a vector bundle on $\mathbb \Delta(\overline R)$ and $M$ denotes the corresponding $\mathbb\Delta$-hyperstratified module on $R$, then the sequence
\[
0 \longrightarrow E \longrightarrow \mathcal L(M) \overset{\mathcal L(\partial_{M,\mathbb\Delta})} \longrightarrow \mathcal L(M) \longrightarrow 0
\]
is exact.
\end{prop}

\begin{proof}
Using little Poincar\'e lemma (proposition \ref{little}), this is derived from the previous results exactly as in \cite{GrosLeStumQuiros23b}, proposition 6.1.
\end{proof}

\begin{thm} \label{prisdR}
Let $W$ be the ring of Witt vectors of a perfect field $k$ of characteristic $p >0$ and $\zeta$ a primitive $p$th root of unity.
Then there exists an equivalence between the category of prismatic vector bundles $E$ on $W[\zeta]$ and the category of finite free weakly nilpotent $\nabla_{\mathbb\Delta}$-modules $M$ on $R$ and we have
\[
\mathrm R \Gamma(W[\zeta]_{\mathbb{\Delta}}, E) \simeq \mathrm R\Gamma_{\mathrm{dR},\mathbb\Delta}(M).
\]
\end{thm}

\begin{proof}
The equivalence follows from theorem \ref{eqTayl} and theorem \ref{invlem}.
The isomorphism in cohomology follows from the prismatic Poincar\'e lemma (proposition \ref{poinc}) as in \cite{GrosLeStumQuiros23b}, theorem 6.2.
\end{proof}

\begin{rmk}
We recover theorem 4.10 of \cite{GaoMinWang23} in the case $\mathcal O_K = W[\zeta]$: there is no cohomology in degree $> 1$.
\end{rmk}

In order to later make the link with the theory of ($\varphi$-$\Gamma$)-modules, let us describe how absolute calculus is related to the cyclotomic tower in our situation.
For $n \in \mathbb N \cup \{\infty\}$, we set $K_n := \mathrm{Frac}(W)[\mu_{p^n}]$ and $\Gamma_n = \mathrm{Gal}(K_{\infty}/K_n)$.
We shall simply write $K := K_1$, so that $\mathcal O_K = \overline R$, and $\Gamma := \Gamma_1$.
We let $\Gamma$ act continuously on $R$ via $g(q) = q^{\chi(g)}$ where $\chi$ denotes the cyclotomic character.
A \emph{$\Gamma$-module (on $R$)} is an $R$-module endowed with a continuous semilinear action of $\Gamma$.

A $\Gamma$-module $M$ is in particular a $\gamma$-module.
In the case $M$ is $(q-1)$-torsion free, then, according to proposition \ref{gammeq}, this is the same thing as a $\nabla_{\mathbb\Delta}$-module.
We shall then call $M$ topologically/wealky (quasi-) nilpotent if the corresponding $\nabla_{\mathbb\Delta}$-module is.

\begin{prop} \label{nabGam}
The category of weakly nilpotent finite free $\Gamma$-modules on $R$ that become trivial modulo $q-1$ is isomorphic to the category of weakly nilpotent finite free $\nabla_{\mathbb\Delta}$-modules on $R$. 
\end{prop}

\begin{proof}
Assume first that $p$ is odd.
The cyclotomic character induces an isomorphism $\Gamma \simeq 1 + p \mathbb Z_p$ (as well as $\Gamma_0 \simeq \mathbb Z_p^\times$).
Moreover, the logarithm induces an isomorphism $1 + p \mathbb Z_p \simeq pZ_p$.
We have $\log(1+p) \equiv p \mod p^2$ and it follows that $p+1$ is a topological generator of $1 + p\mathbb Z_p$.
In other words, $\gamma$ is a topological generator of $\Gamma$.
Our statement therefore follows from the second assertion of proposition \ref{gammeq}.

In the case $p=2$, the argument needs some refinement.
We have $K = K_0$ and (with an argument similar to the above)
\[
\Gamma = \Gamma_0 \simeq \mathbb Z_2^\times = \langle 3 \rangle \times \langle -1 \rangle \simeq \mathbb Z_2 \times \mathbb Z/2\mathbb Z.
\]
It follows that a continuous action of $\Gamma$ is equivalent to a continuous action of $\gamma = g_3$ and a commuting action of $\sigma := g_{-1}$ whose square is trivial.
Then, according to proposition \ref{gammeq} again, we see that a continuous action of $\Gamma$ is equivalent to a $\mathbb\Delta$-connection together with a $\sigma$-linear horizontal symmetry.
Our assertion therefore follows from lemma \ref{p2case}.
\end{proof}

\begin{rmk}
\begin{enumerate}
\item
The nilpotence condition is only necessary to deal with the case $p=2$.
When $p$ is odd, the category of $(q-1)$-torsion-free $\Gamma$-modules that become trivial modulo $q-1$ is isomorphic to the category of $(q-1)$-torsion-free $\nabla_{\mathbb\Delta}$-modules.
\item In the case $p=2$, the category of $(q-1)$-torsion-free $\Gamma$-modules that become trivial modulo $q-1$ is equivalent to the category of $(q-1)$-torsion-free $\nabla_{\mathbb\Delta}$-modules endowed with a semi-linear horizontal symmetry that reduces to identiy modulo $q-1$.
\item
Assume $p$ is odd and consider the action of $e \in \mathbb F_p^\times$ on $R$ given by $g_{[e]} : q \mapsto q^{[e]}$ where $[e]$ denotes the Teichmüller character.
Then the category of $(q-1)$-torsion-free $\Gamma_0$-modules that become trivial modulo $q-1$ is equivalent to the category of $(q-1)$-torsion-free $\nabla_{\mathbb\Delta}$-modules with an extra semilinear action of $\mathbb F_p^\times$ which becomes trivial modulo $q-1$.
\item For the last argument of the proof, we can also rely on theorem \ref{prisdR}: since\footnote{This is not true for odd $p$.} $\sigma$ is an automorphism of the object $R$ of the prismatic site of $W[\zeta]$, if we denote by $E$ the prismatic crystal associated to $M$, there is an isomorphism
\[
\mathrm{Hom}_\nabla(\sigma^*M, M) \simeq \mathrm{Hom}(E, E)
\]
and the identity of $E$ therefore provides the expected symmetry.
\end{enumerate}
\end{rmk}

\section{de Rham prismatic crystals} \label{generic}

When $p$ is invertible, one can rely on usual powers instead of twisted powers.
We will describe the process, which is completely analogous to the one we developed so far, but much simpler, and discuss the link with absolute calculus.
We will also indicate how our strategy relates to previous work of other authors.  For example, the main result of this section, proposition \ref{Bdreq}, recovers (in our situation) one of the equivalences in  \cite{GaoMinWang22} , theorem 5.1 (see also corollary 7.17 of \cite{BhattScholze21}). Nevertheless, our approach provides, as a bonus, explicit formulas for the equivalence (proposition \ref{invfor}) between twisted and ``classical'' operators.

We let $k$ be a perfect field of positive characteristic $p$, we denote by $W$ its ring of Witt vectors, we choose a primitive $p$th root of unity $\zeta$, we write $K := \mathrm{Frac}(W)(\zeta)$ and finally, we set $S := K[[q-\zeta]]$.
We shall also write $\mathcal O_K := W[\zeta]$ and $\mathcal O_S := \mathcal O_K[[q-\zeta]]$.

As a byproduct of calculus on log-schemes and on formal schemes, one can develop a theory of calculus on formal log-schemes but we will only need it here in a very specific situation.
With $S$ as above, we are interested in the formal log-scheme $\mathrm{Spf}(S)^{\log}$ where both the topology and the logarithmic structure come from the special fiber $\mathrm{Spec}(K)$. 
Concretely, we first formally blow up $S \widehat \otimes_K S$ along the special fiber and obtain the complete $K$-algebra $P_{\log}$.
There exists a \emph{canonical map} $\iota : S \to P_{\log}$ coming from the action on the left and a \emph{log-Taylor map} $\theta : S \to P_{\log}$ coming from the action on the right.
There also exists a \emph{flip map} $\tau : P_{\log} \simeq P_{\log}$ that exchanges both strutures.
We will use, as we always do, the notation $\otimes'_S$ in order to insist on the fact that $S$ acts via the log-Taylor map on the left.
We shall also consider below the comultiplication map $\Delta : P_{\log} \to P_{\log} \widehat \otimes_S P_{\log}$ which is induced as usual by the map $f \otimes g \mapsto f \otimes 1 \otimes g$.
We denote by $I_{\log}$ the kernel of the augmentation map $e : P_{\log} \to S$ and we define for each $n \in \mathbb N$, the ring of \emph{log-principal parts of order $n$} $P_{\log}^{n} := P_{\log}/I_{\log}^{n+1}$ and the module of \emph{log-differentials} $\Omega_{\log} := I_{\log}/I_{\log}^{2}$.
It comes with its universal \emph{log-derivation}
\[
\left(\mathrm {dlog} : S\setminus \{0\} \to \Omega_{\log},\quad \mathrm {d} : S \to \Omega_{\log}\right).
\]
A \emph{log-connection} on an $S$-module $M$ is a $K$-linear map $\nabla_M : M \to M \otimes_S \Omega_{\log}$ that satisfies the usual Leibniz rule.
We shall then call $M$ a \emph{$\nabla_{\log}$-module}.
A \emph{log-differential operator\footnote{Note that the approach slightly differs from definition \ref{defdifo} since the lift will be unique here.} of order at most $n$} is a $K$-linear map $u : M \to N$ between $S$-modules which extends (automatically uniquely) to an $S$-linear map $\widetilde u : P_{\log}^{n} \otimes' M \to N$.
The composition of two log-differential operators $u$ and $v$ of respective order at most $n$ and $m$ is a log-differential operator of order at most $n+m$.
Actually, the comultiplication map induces for all $m, n \in \mathbb N$
\[
\Delta_{n,m} : P_{\log}^{n+m} \to P_{\log}^{n} \otimes_S' P_{\log}^{m}
\]
and we always have
\[
\widetilde u\ \widetilde \circ\ \widetilde v := \widetilde {u\circ v} = \widetilde u \circ (\mathrm{Id} \otimes \widetilde v) \circ (\Delta_{n,m} \otimes \mathrm{Id}).
\]
In particular, log-differential operators of finite order from $S$ to $S$ form a ring
\[
\mathrm D_{S,\log} \simeq \varinjlim_n \mathrm{Hom}_S(P_{\log}^{n}, S).
\]
There also exists the notions of a \emph{log-stratification} (resp.\ a \emph{log-Taylor structure}) on an $S$-module $M$ which is a compatible family of $S$-linear isomorphisms (resp.\ maps)
\[
\epsilon_{n} : P_{\log}^{n}\otimes'_S M \simeq M \otimes_{S} P_{\log}^{n} \quad (\mathrm{resp.}\ \theta_{n} : M \to M \otimes_{S} P_{\log}^{n})
\]
satisfying the usual cocycle conditions.
It is equivalent to give a log-connection, a $\mathrm D_{S,\log}$-module structure, a log-stratification or a log-Taylor structure on an $S$-module $M$.

In order to describe these objects explicitly, we need to choose a coordinate on $S$ as well as an equation for the divisor.
The standard choice would be $q-\zeta$ for both but we choose $q$ and $(p)_q$ respectively in order to recover the arithmetic in the end.
It is of course necessary to make sure that $(p)_q$ is indeed a uniformizer.
We denote by $(p)_q^{(k)} := \partial^k((p)_q)$ the (usual) $k$th derivative of $(p)_q$ with respect to $q$ and by $(p)_\zeta^{(k)} := (p)_q^{(k)}(\zeta)$ its evaluation at $q=\zeta$.
We shall simply write $(p)_q'$ and $(p)_\zeta'$ when $k=1$, so that
\[
(p)'_\zeta = 1 + 2\zeta + 3\zeta^2 + \cdots + (p-1)\zeta^{p-2}
\]
(we already implicitly used this in our examples of $\mathbb\Delta$-connections in section \ref{secabs}). 
Since $\zeta$ is a primitive $p$th root of unity, a direct calculation gives the following equality:
\begin{lem} $p=(\zeta^2-\zeta) (p)'_\zeta$. \qed
\end{lem}

As a consequence of this lemma, we see that $\displaystyle v_p((p)'_\zeta) = \frac {p-2}{p-1} > 0$ when $p$ is odd.

\begin{lem}
We have $(p)_q = (q-\zeta)\nu$ with $\nu\in S^\times$.
\end{lem}

\begin{proof}
The element $(p)_q \in K[[q - \zeta]]$ has Taylor expansion
\[
(p)_q = \sum_{k=1}^{p-1} \frac 1{k!} (p)_\zeta^{(k)} (q-\zeta)^{k}
\]
and the above equality therefore holds with
\begin{equation} \label{upsi}
\nu := (p)'_\zeta + (q-\zeta) \sum_{k=2}^{p-1} \frac 1{k!} (p)_\zeta^{(k)} (q-\zeta)^{k-2} \in S^\times.\qedhere
\end{equation}
\end{proof}

We have $S \widehat \otimes_K S \simeq S[[\xi]]$ with $\xi = 1 \otimes q - q \otimes 1$ and therefore $P_{\log} \simeq S[[\omega]]$ with $\xi = (p)_q\omega$.
The log-Taylor map $\theta$ is given by $q \mapsto q + (p)_q\omega$, the flip map $\tau$ by $\omega \mapsto -L(\omega)\omega$ and the comultiplication map $\Delta$ by $\omega \mapsto L(\omega) \otimes \omega + \omega \otimes 1$ with
\[
L(\omega) := 1 + \sum_{k=1}^{p-1} (p)_q^{(k)} (p)_q^{k-1}\omega^{[k]}
\]
(so that $\theta((p)_q) = (p)_q L(\omega)$).
More generally, we have
\begin{enumerate}
\item $\displaystyle \forall n \in \mathbb N, \quad \theta(q^n) = \sum_{k=0}^n k! {n \choose k} (p)_q^kq^{n-k}\omega^{[k]}$,
\item $\displaystyle\forall n \in \mathbb N, \quad \tau(\omega^{[n]}) = (-1)^nL(\omega)^n\omega^{[n]}$,
\item $\displaystyle\forall n \in \mathbb N, \quad \Delta(\omega^{[n]}) = \sum_{k=0}^n L(\omega)^k\omega^{[n-k]} \otimes \omega^{[k]}$.
\end{enumerate}

The $R$-module $\Omega_{\log}$ is a free module of rank one.
We will choose as generator the element
\[
\frac 1{\nu}\mathrm{dlog}(q-\zeta) = \frac 1{(p)'_q} \mathrm{dlog}(p)_q
\]
which is the class of $\omega$.
If we denote by $\partial_{\log}^{\langle k \rangle}$ the log-differential operator of order $k$ on $S$ such that
\[
\widetilde \partial_{\log}^{\langle k \rangle}(\omega^{[l]}) = \left\{\begin{array} l 1 \ \mathrm{if} \ k=l \\ 0\ \mathrm{otherwise},\end{array}\right.
\]
then the ring $\mathrm D_{S,\log}$ is a free $S$-module with basis $\left(\partial_{\log}^{\langle n \rangle}\right)_{n \in \mathbb N}$.
Actually, $\mathrm D_{S,\log}$ is a subring of the usual ring of differential operators $\mathrm D_{S}$ and we have $\partial_{\log}^{\langle n \rangle} = (p)_q^n\partial^n$ if $\partial$ denotes the standard derivative with respect to $q$.
The equivalence between log-connection, $\mathrm D_{S,\log}$-module structure and log-stratification (or log-Taylor structure) is given by
\[
\nabla_M(s) =\nu^{-1}\partial_{M,\log}(s) \otimes \mathrm{dlog}(q-\zeta) \quad \Leftrightarrow \quad \theta_{n}(s) = \sum_{k=0}^n \partial_{\log}^{\langle k \rangle}(s) \otimes \omega^{[k]}
\]
with $\partial_{M,\log}(s) = \partial_{\log}^{\langle 1 \rangle}(s)$.

In order to go further, it is necessary to better understand composition of log-differential operators and we shall first show the following (recall that $\nu$ is defined by equality \eqref{upsi}):

\begin{lem} $\displaystyle \forall k \in \mathbb N, \quad \Delta(\nu^k\omega^{[k]}) = \sum_{i+j=k} (1 + \nu\omega)^j\nu^i\omega^{[i]} \otimes \nu^j\omega^{[j]}.$
\end{lem}

\begin{proof}
It is sufficient to show that
\[
\Delta(\nu\omega) = (1 + \nu\omega) \otimes \nu\omega + \nu\omega \otimes 1
\]
and then take divided powers.
Thus, we have to show that
\[
\nu(L(\omega) \otimes \omega + \omega \otimes 1) =(1 + \nu\omega) \otimes \nu\omega + \nu\omega \otimes 1
\]
or, equivalently,
\[
\nu L(\omega) \otimes 1 = (1 + \nu\omega) \otimes \nu.
\]
We know that $\theta((p)_q) = (p)_qL(\omega)$ and $(p)_q = (q-\zeta)\nu$ and we compute
\begin{linenomath}
\begin{align*}
(p)_q\nu L(\omega) \otimes 1 & = \nu\theta((p)_q) \otimes 1 \\ &= \nu \otimes (p)_q \\& = \nu \otimes (q-\zeta)\nu \\ &= \nu(q+ (p)_q\omega-\zeta) \otimes \nu \\ &= (p)_q(1 + \nu\omega) \otimes \nu. \qedhere
\end{align*}
\end{linenomath}
\end{proof}

\begin{prop} \label{ucompplus2}
$\displaystyle \forall n \in \mathbb N, \quad \nu^{-n}\partial_{\log}^{\langle n \rangle} = \prod_{k=0}^{n-1}\left( \nu^{-1} \partial_{\log}- k\right)$.
\end{prop}

\begin{proof}
By definition, we have for all $k \in \mathbb N$,
\begin{linenomath}
\begin{align*}
(\widetilde \partial_{\log}^{\langle n \rangle} \widetilde \circ\ \nu^{-1} \widetilde \partial_{\log})(\omega^{[k]})
& = \nu^{-k}(\widetilde \partial_{\log}^{\langle n \rangle} \widetilde \circ\ \nu^{-1} \widetilde \partial_{\log})(\nu^k\omega^{[k]}) \\
& = \nu^{-k}\widetilde \partial_{\log}^{\langle n \rangle} ((\mathrm{Id} \otimes \nu^{-1} \widetilde \partial_{\log})(\Delta(\nu^k\omega^{[k]})))
\\ & = \nu^{-k}\widetilde \partial_{\log}^{\langle n \rangle} \left((\mathrm{Id} \otimes \nu^{-1} \widetilde \partial_{\log})\left(\sum_{i+j=k} (1 + \nu\omega)^j\nu^i\omega^{[i]} \otimes \nu^j\omega^{[j]} \right)\right)
\\ & = \nu^{-k}\widetilde \partial_{\log}^{\langle n \rangle} ((1 + \nu\omega)\nu^{k-1}\omega^{[k-1]})
\\ & = \nu^{-1} \widetilde \partial_{\log}^{\langle n \rangle} (\omega^{[k-1]}) + k\widetilde \partial_{\log}^{\langle n \rangle} (\omega^{[k]}).
\end{align*}
\end{linenomath}
It follows that
\[
\widetilde \partial_{\log}^{\langle n \rangle} \widetilde \circ\ \nu^{-1}\widetilde \partial_{\log} = \nu^{-1} \widetilde \partial_{\log}^{\langle n+1 \rangle} + n\widetilde \partial_{\log}^{\langle n \rangle}
\]
or, equivalently,
\[
\widetilde \partial_{\log}^{\langle n+1 \rangle} = \nu \widetilde \partial_{\log}^{\langle n \rangle} \widetilde \circ (\nu^{-1}\widetilde \partial_{\log} - n).
\]
Our formula is then obtained by induction.
\end{proof}

\begin{rmk}
As a consequence of the proposition, we see that
\[
\partial_{\log}^{\langle n \rangle} = \sum_{k=1}^n s(n,k) \nu^{n}\left( \nu^{-1} \partial_{\log}\right)^k
\]
and
\[
\left( \nu^{-1} \partial_{\log}\right)^n = \sum_{k=1}^n S(n,k) \nu^{-k}\partial_{\log}^{\langle k\rangle} 
\]
where $s(n,k)$ and $S(n,k)$ denote the Stirling numbers of the first and second kind respectively.
\end{rmk}

From now on, we endow $K$ with the $p$-adic valuation and $S$ with the product topology (or equivalently the inverse limit topology).
This is a $W$-linear topology with $p^n\mathcal O_S + (q-\zeta)^mS$ as basis of neighborhoods of zero.
Then, $S$ is a complete topological ring\footnote{This is not an adic ring and not even an $f$-adic ring in the sense of Huber.} with both $p$ and $q-1$ as topologically nilpotent units.
We endow the polynomial ring $S[\omega]$ (resp.\ the divided polynomial ring $S\langle \omega \rangle$) with the direct sum topology with respect to the standard basis.
Be careful that the isomorphism of $S$-algebras $S[\omega] \simeq S\langle \omega \rangle$ is \emph{not} a homeomorphism because we use powers on the one hand and divided powers on the other in order to define the topology.
The log-Taylor (resp.\ flip, resp.\ comultiplication) map extends uniquely to a continuous $K$-linear ring homomorphism
\[
\theta : S \mapsto \widehat {S\langle \omega \rangle}, \quad q \mapsto q + (p)_q\omega
\]
\[
\left(\mathrm{resp.} \quad \tau : \widehat {S\langle \omega \rangle}\to \widehat {S\langle \omega \rangle},
\quad q \mapsto q + (p)_q\omega, \quad \omega \mapsto -L(\omega)\omega \right.,
\]
\[
\left.\mathrm{resp.} \quad \Delta : \widehat {S\langle \omega \rangle}\to \widehat {S\langle \omega \rangle} \widehat \otimes'_S \widehat {S\langle \omega \rangle},
\quad q \mapsto q, \quad \omega \mapsto L(\omega) \otimes \omega + \omega \otimes 1\right).
\]
One may then define a \emph{log-hyperstratification} (resp.\ a \emph{log-Taylor map}) on an $S$-module $M$ as a an isomorphism (resp.\ a map)
\[
\epsilon : \widehat {S\langle \omega \rangle} \widehat\otimes'_S M \simeq M \widehat\otimes_{S} \widehat {S\langle \omega \rangle} \quad (\mathrm{resp.}\ \theta : M \to M \widehat\otimes_{S} \widehat {S\langle \omega \rangle})
\]
satisfying the usual cocycle condition and show that these two notions are equivalent.
A log-hyperstratification induces a log-connection on $M$ and we have
\[
\forall s \in M, \quad \theta(s) = \sum_{k=0}^\infty \partial_{\log}^{\langle k \rangle}(s) \otimes \omega^{[k]}.
\]
This provides a fully faithful functor on finite $S$-modules.
Comultiplication $\Delta$ induces a log-Taylor map on $\widehat {S\langle \omega \rangle}$ for the \emph{right} structure and, after conjugation through the flip map, a log-Taylor map $\theta$ for the left structure.
We shall denote by $L_{\log}^{\langle k \rangle}$ and $\partial_{\log}^{\langle k \rangle}$ the corresponding higher derivatives.
We have
\[
\forall k \leq n \in \mathbb N, \quad L_{\log}^{\langle k \rangle}(\omega^{[n]}) = L(\omega)^k\omega^{[n-k]}
\]
but it is much harder to find an explicit formula for $\partial_{\log}^{\langle k \rangle}$ and we shall rely on its definition:
\[
\forall k \in \mathbb N, \quad \partial_{\log}^{\langle k \rangle} = \tau \circ L_{\log}^{\langle k \rangle} \circ \tau.
\]
However, one can show that $\partial_{\log}$ is a log-derivation of the \emph{ring} $ \widehat {S\langle \omega \rangle}$.
If $M$ is an $S$-module, we may then consider its log-linearization $L_{\log}(M) := \widehat {S\langle \omega \rangle} \widehat\otimes'_S M$.
Both the log-hyperstratification and the augmentation map $e$ of $\widehat {S\langle \omega \rangle}$ extend to $L_{\log}(M)$.
A \emph{log-differential operator} is a $K$-linear map $u : M \to N$ between $S$-modules that extends (automatically uniquely) to an $S$-linear map $\widetilde u : L_{\log}(M) \to N$.
It has a log-linearization
\[
L_{\log}(u) = (\mathrm{Id} \otimes \widetilde u) \circ (\Delta \otimes \mathrm{Id}): L_{\log}(M) \to L_{\log}(N)
\]
which is compatible with the log-hyperstratifications and such that $e \circ L_{\log}(u) = \widetilde u$.
If $M$ is a $\nabla_{\log}$-module, then the corresponding log-derivation $\partial_{M,\log}$ is a log-differential operator of order one (with $\widetilde \partial_{M,\log}(s \otimes \omega) = s$) and we can consider $L_M := L_{\log}(\partial_{M,\log})$.
Then, $\ker L_M$ comes with a log-hyperstratification.
One can show (with the help of a variant of little Poincar\'e lemma) that $\widehat M \simeq \ker L_M$ when the log-connection comes from a log-hyperstratification.

\begin{prop}
The functor $M \mapsto \ker L_M$ is adjoint to the forgetful functor from the category of finite $S$-modules endowed with a log-hyperstratification to the category of finite $\nabla_{\log}$-modules on $S$.
\end{prop}

\begin{proof}
According to the previous discussion, it is sufficient to check that the canonical map
\begin{equation} \label{canmap2}
\ker L_M \hookrightarrow L_{\log}(M) \overset {e} \twoheadrightarrow M
\end{equation}
is horizontal.
First of all, if $\varphi \in \ker L_M$, then $\widetilde \partial_{M,\log}(\varphi) = e(L_M(\varphi)) = 0$.
If we write $\varphi = \sum_k \omega^{[k]} \otimes s_k$, then $\widetilde \partial_{M,\log}(\varphi) = \partial_{M,\log}(s_0) + s_1$ and it follows that $\partial_{M,\log}(e(\varphi)) = \partial_{M,\log}(s_0) = -s_1$.
Thus, we need to show that $e(\partial_{\log}(\varphi)) = -s_1$, or equivalently that $e(\partial_{\log}(\omega^{[k]})) = -1$ when $k=1$ and $0$ otherwise.
We have
\begin{linenomath}
\begin{align*}
L_{\log}(\tau(\omega)) &= L_{\log}(-L(\omega)\omega) 
\\ &= -L_{\log}\left(\sum_{k=0}^{p-1} (p)_q^{(k)} (p)_q^{k-1}\omega^{[k]}\omega\right)
\\ &= -\sum_{k=0}^{p-1} (k+1)(p)_q^{(k)} (p)_q^{k-1} L_{\log}(\omega^{[k+1]})
\\ &= -\sum_{k=0}^{p-1} (k+1)(p)_q^{(k)} (p)_q^{k-1} L(\omega)\omega^{[k]}
\end{align*}
\end{linenomath}
Since $e \circ \tau = e$, we obtain for $k=1$
\[
e(\partial_{\log}(\omega)) =(e \circ \tau)(L_{\log}(\tau(\omega))) = -1.
\]
The case $k=0$ is clear, and for $k>1$ it follows from the fact that $\partial_{\log}$ is a log-derivation.
\end{proof}

\begin{dfn}
A $\nabla_{\log}$-module $M$ on $S$ is said to have \emph{topologically nilpotent monodromy} if  $\partial_{\log}$ when $p$ is odd,
and $\partial_{\log}^2 - \partial_{\log}$ when $p=2$, is topologically nilpotent modulo $(p)_q$.
\end{dfn}

\begin{lem} \label{nilpeq}
A log-connection on a finite free $S$-module $M$ extends to a log-hyperstratification if and only $M$ has topologically nilpotent monodromy.
\end{lem}

\begin{proof}
This is similar to the proof of lemma \ref{techlem} and we only do the case $p$ odd.
If $\varphi = \sum_k \omega^{[k]} \otimes s_k \in L_{\log}(M)$, then
\begin{linenomath}
\begin{align*} 
L_{M} (\varphi) &= \sum_{k=0}^\infty \left(\omega^{[k]} \otimes \partial_{M,\log}( s_k) + L(\omega)\omega^{[k-1]}\otimes s_k \right)
\\ & \equiv \sum_{k=0}^\infty \left(\omega^{[k]} \otimes \partial_{M,\log}( s_k) + (1 + (p)'_\zeta\omega) \omega^{[k-1]}\otimes s_k \right) \mod q-\zeta
\\ & \equiv \sum_{k=0}^\infty \omega^{[k]} \otimes \left(\partial_{M,\log}(s_k) + k(p)'_\zeta s_k + s_{k+1} \right) \mod q-\zeta.
\end{align*}
\end{linenomath}
Therefore,
\begin{linenomath}
\begin{align*} 
&L_{M}(\varphi) \equiv 0 \mod q-\zeta
\\&\Leftrightarrow \forall k \geq 0, s_{k+1} \equiv -(\partial_{M,\log} + k(p)'_\zeta )(s_k) \mod q-\zeta
\\&\Leftrightarrow \forall n >0 , s_{n} \equiv (-1)^n \prod_{k=0}^ {n-1} (\partial_{M,\log} + k(p)'_\zeta) (s_0) \mod q-\zeta
\\&\Leftrightarrow \forall n >0 , s_{n} \equiv \sum_{k=1}^ {n} (-1)^ks(n,k)((p)'_\zeta)^{n-k} \partial_{M,\log}^k(s_0) \mod q-\zeta.
\end{align*}
\end{linenomath}
The last equation is also equivalent to
\[
\partial_{M,\log}^n(s_0) \equiv (-1)^n\sum_{k=1}^ {n} S(n,k)((p)'_\zeta)^{n-k} s_k \mod q-\zeta.
\]
Since we assumed that $p$ is odd, we have $v_p((p')_\zeta) >0$.
Since $s_n \to 0$, we see that, if $L_{M}(\varphi) = 0$, then $\partial_{M,\log}^n(s_0) \to 0 \mod q-\zeta$.

Conversely, when $M$ has topologically nilpotent monodromy it follows from Nakayama lemma that the map \eqref{canmap2} is surjective.
Another application of Nakayama lemma (see the end of the proof of lemma \ref{techlem}) will then show that the map is indeed bijective.
Being horizontal, it is an isomorphism.
\end{proof}

\begin{rmks}
\begin{enumerate}
\item
Proposition \ref{ucompplus2}, provides another equivalent condition for a log-connection on a finite free $S$-module $M$ to extend to a log-hyperstratification:
\begin{equation} \label{prodform}
\forall s \in S, \quad \lim_{n} \nu^n\prod_{k=0}^{n}\left( \nu^{-1} \partial_{\log}- k\right)(s) = 0.
\end{equation}
\item
Condition in \eqref{prodform} is exactly the same as the one in definition 1.15.2 of \cite{GaoMinWang22} with $a = (p)'_\zeta$.
This follows from the identities
\[
\nabla_M(s) = \nu^{-1}\partial_{\log}(s) \otimes \mathrm {dlog}(q- \zeta) \quad \mathrm{and} \quad \nu \equiv (p)'_\zeta \mod q - \zeta.
\]
\end{enumerate}
\end{rmks}

\begin{xmps}
\begin{enumerate}
\item
We can consider the free $S$-module $F_n$ on one generator $s$ such that $\partial_{\log}(s) = n\nu s$ for some $n \in \mathbb Z$.
Then, we have
\[
\partial_{\log}^{\langle k \rangle}(s) = k!{n \choose k}\nu^ks
\]
and therefore
\[
\theta(s) = \sum_{k=0}^nk!{n \choose k}\nu^ks \otimes \omega^{[k]} = s \otimes (1 + \nu\omega)^n.
\]
Of course, $F_n \simeq (p)_q^nS$.
\item
We can also consider the $K$-vector space $G_n$ on one generator $s$ such that $\partial_{\log}(s) = n(p)'_\zeta s$ with $n \in \mathbb Z$.
We have $G_n \simeq \overline F_n \simeq F_n/F_{n+1} = (p)_q^nS/ (p)_q^{n+1}S$.
\item The Breuil-Kisin prism corresponds (see proposition \ref{Bdreq} below) to the free $S$-module $S\{1\}$ on one generator $e_S$ with
\[
\partial_{\log}((q-1)e_S) = \frac {(p)_q}{q\log(q)} (q-1)e_S.
\]
Actually, using the coming development, one can deduce from proposition \ref{BKthet} that
\[
\theta((q-1)e_S) = \left(1 + \frac {(p)_q}{\log(q)} \sum_{k=1}^\infty (-1)^{k-1} (k-1)! (p)_q^{k-1} q^{-k} \omega^{[k]}\right)(q-1)e_S.
\]
\item In fact, as was already expected by Scholze in \cite{Scholze17}, example 6.3, there exists an isomorphism
\[
S\{1\} \simeq (p)_qS, \quad e_S \leftrightarrow \frac {\log(q)}{q-1}
\]
of $\nabla_{\log}$-modules on $S$.
\end{enumerate}
\end{xmps}

\begin{prop} \label{nogap}
There exists a natural isomorphism $\widehat{S\langle \omega \rangle} \simeq \widehat{S\langle \omega \rangle}_{\mathbb\Delta}$.
\end{prop}

\begin{proof}
It follows from proposition \ref{fromMin} that the equality $S\langle \omega \rangle = S\langle \omega \rangle_{\mathbb\Delta}$ of proposition \ref{poweg2} extends to the completions and provides a map $\widehat{S\langle \omega \rangle} \to \widehat{S\langle \omega \rangle}_{\mathbb\Delta}$.
Thanks to derived Nakayama lemma and proposition \ref{poweg}, this is an isomorphism.
\end{proof}

We recall now that $\mathbb B^+_{\mathrm{dR}}$ is the sheaf on the prismatic site $\mathbb \Delta$ given by
\[
\mathbb B^+_{\mathrm{dR}}(B,J) = \widehat{B[1/p]}
\]
where completion is meant with respect to $J$.
A $\mathbb B^+_{\mathrm{dR}}$-vector bundle will be called a \emph{de Rham prismatic crystal}.

We write $R := W[[ q-1]]$ as usual and shall also consider $\overline R := R/(p)_q = W[\zeta] = \mathcal O_K$.

\begin{lem} \label{Bdr}
We have 
\[
\mathbb B^+_{\mathrm{dR}}((R,(p)_q)) = S \quad \mathrm{and} \quad \mathbb B^+_{\mathrm{dR}}\left(\widehat {R\langle \omega \rangle}_{\mathbb\Delta},(p)_q\right) = \widehat {S\langle \omega \rangle}.
\]
\end{lem}

\begin{proof}
One easily checks that
\[
\mathbb B^+_{\mathrm{dR}}(R,(p)_q) = \widehat{R[1/p]} = \varprojlim K[q]/(p)_q^{n+1} = K[[q - \zeta]] = S
\]
and, thanks to proposition \ref{nogap},
\[
 \mathbb B^+_{\mathrm{dR}}\left(\widehat {R\langle \omega \rangle}_{\mathbb\Delta} ,(p)_q \right) = \widehat{R\langle \omega \rangle_{\mathbb\Delta}[1/p]} = \widehat{S\langle \omega \rangle}_{\mathbb\Delta} \simeq \widehat{S\langle \omega \rangle}.\qedhere
\]
\end{proof}

Note that the analogous results also hold for higher powers of $(R, (p)_q)$ in the prismatic site $\mathbb \Delta(\overline R)$.

We recover the left vertical equivalence of \cite{GaoMinWang22}, theorem 5.1 in our situation (see also the last statement -- corollary 7.17 -- of \cite{BhattScholze21}):

\begin{prop} \label{Bdreq}
There exists an equivalence between the category of de Rham prismatic crystals $E$ on $\mathcal O_K$ and the category of finite free $\nabla_{\log}$-modules $M$ on $S$ with topologically nilpotent monodromy.
\end{prop}

\begin{proof}
Using lemma \ref{nilpeq}, this formally follows from corollary \ref{covabs} and lemma \ref{Bdr}.
\end{proof}

\begin{rmk}
As a consequence, we also obtain
\[
\mathrm H^{0}(\mathcal O_{K\mathbb \Delta}, E) \simeq \mathrm{Hom}(\mathbb B^+_{\mathrm{dR}}, E) \simeq \mathrm{Hom}_{\nabla}(S, M) \simeq \mathrm H_{\mathrm{dR,log}}^{0}(M) 
\]
and
\[
\mathrm H^{1}(\mathcal O_{K\mathbb \Delta}, E) \simeq \mathrm{Ext}(\mathbb B^+_{\mathrm{dR}}, E) \simeq \mathrm{Ext}_{\nabla}(S, M) \simeq \mathrm H_{\mathrm{dR,log}}^{1}(M) 
\]
because the category of vector bundles (resp.\ of finite free modules) is stable under extensions.
\end{rmk}

There also exists a theory of absolute calculus on $S := K[[q-\zeta]]$ (similar to the one developed on $R := W[q-1]]$ in the present paper) and we have:

\begin{prop} \label{Bdreq_qp}
There exists an equivalence between the category of de Rham prismatic crystals $E$ on $\mathcal O_K$ and the category of finite free $\nabla_{\mathbb\Delta}$-modules $M$ on $S$ with topologically nilpotent monodromy and we have
\[
\mathrm R \Gamma(\mathcal O_{K\mathbb\Delta}, E) \simeq \mathrm R\Gamma_{\mathrm{dR},\mathbb\Delta}(M).
\]
\end{prop}

\begin{proof}
This is shown as in theorem \ref{prisdR} with the same strategy.
For example, the analog to lemma \ref{techlem} is proved as follows.
Using Nakayama lemma, we can work on $K$.
Then the connection is topologically nilpotent and there exists therefore a stable lattice.
The $\mathbb\Delta$-connection is weakly nilpotent on this lattice and we can use Nakayama again to work now over $k$ so that the same arguments as in lemma \ref{techlem} apply.
\end{proof}

Given a de Rham prismatic crystal $E$ on $\mathcal O_K$, there exists both a $\mathbb\Delta$-connection and a log-connection on $M$.
This is an ultrametric confluence phenomenon:

\begin{cor}
The category of $\mathbb\Delta$-hyperstratified modules on $S$ is isomorphic to the category of log-hyperstratified modules.
\end{cor}

\begin{proof}
This follows from proposition \ref{nogap}.
\end{proof}

There exists explicit formulas to go back and forth:

\begin{prop} \label{invfor}
Let $E$ be a de Rham prismatic crystal on $\mathcal O_K$ and $M$ the corresponding $S$-module.
Then for all $s \in M$, we have
\[
\partial_{\log}^{\langle k \rangle}(s) = \sum_{n=k}^{\infty} \frac {k!}{(n)_{q^p}!}s_{q^p}(n,k)(q-1)^{n-k} q^{n-k} \partial^{\langle n \rangle}_{\mathbb\Delta}(s)
\]
and
\[
 \partial^{\langle k \rangle}_{\mathbb\Delta}(s) = \sum_{n=k}^{\infty} \frac{(k)_{q^p}!}{n!} S_{q^p}(n,k) (q-1)^{n-k} q^{n-k} \partial_{\log}^{\langle n \rangle}(s)
\]
where $s_{q^p}(n,k)$ (resp.\ $S_{q^p}(n,k)$) denotes the $q^p$-Stirling number of the first (resp.\ second) kind.
\end{prop}

\begin{proof}
Since $ \widehat{S\langle \omega \rangle}_{\mathbb\Delta} = \widehat{S\langle \omega \rangle}$, we can write
\[
\theta(s) = \sum_{k=0}^\infty \partial^{\langle k \rangle}_{\mathbb\Delta}(s) \otimes \omega^{\{k\}_{\mathbb\Delta}} = \sum_{k=0}^\infty \partial_{\log}^{\langle k \rangle}(s) \otimes \omega^{[k]}.
\]

Moreover, we know from proposition \ref{poweg2} that
\[
\omega^{\{n\}_{\mathbb\Delta}} = \sum_{k=1}^{n} \frac {k!}{(n)_{q^p}!}s_{q^p}(n,k)(q-1)^{n-k} q^{n-k}\omega^{[k]}
\]
and
\[
\omega^{[n]} = \sum_{k=1}^{n} \frac{(k)_{q^p}!}{n!} S_{q^p}(n,k) (q-1)^{n-k} q^{n-k} \omega^{\{k\}_{\mathbb\Delta}}.
\]
Our formulas are then obtained by duality.
\end{proof}

\section{Hodge-Tate prismatic crystals} \label{reduced}

When $(p)_q =0$, one may also rely on usual powers instead of twisted powers.
We will indicate how our former strategy can be applied to recover already known descriptions from \cite{GaoMinWang23} and  \cite{AnschuetzHeuerLeBras22}. In particular, propositions \ref{Sen} and \ref{HTeq} below correspond respectively (in our situation) to one of the equivalences in theorem 4.2 (see also theorem 2.5 in  \cite{AnschuetzHeuerLeBras22}) and to theorem 4.5 in \cite{GaoMinWang23}. Once more, our approach provides, as a bonus, explicit formulas for the equivalence (proposition \ref{invforSen})  between twisted and ``classical'' operators

We let $k$ be a perfect field of positive characteristic $p$ and denote by $W$ its ring of Witt vectors.
We set $R := W[[q-1]]$ and we write $\overline M := M/(p)_qM$ for the reduction modulo $(p)_q$ of an $R$-module $M$.
In particular, $\overline R = W[\zeta]$ where $\zeta$ denotes a primitive $p$th root of unity and we shall denote by $K$ its fraction field so that $\mathcal O_K = \overline R$.

As already mentioned in remark \ref{classd}.2, there exists a theory of log-calculus on $R$ with respect to the divisor of $(p)_q$ and we will systematically use $q$ as a coordinate and $(p)_q$ as a generator for the divisor.
This is completely similar to the situation in section \ref{generic} and we will use the same notation.
When we reduce modulo the divisor, the log-theory becomes a monodromy theory.
More precisely, a \emph{monodromy operator} on an $\mathcal O_K$-module $M$ is simply an $\mathcal O_K$-linear map
\[
\nabla_M : M \to M \otimes_{\mathcal O_K} \overline \Omega_{R,\log}.
\]
This is equivalent to giving an $\mathcal O_K$-linear map $N : M \to M$ via
\[
\nabla_M(s) =N(s) \otimes \overline{\mathrm dq/(p)_q}
\]
and we may call $N$ the \emph{Sen operator} associated to $\nabla_M$ (adopting in this context the terminology of \cite{AnschuetzHeuerLeBras22}, beginning of section 1.1).
Of course, if $M$ is a $\nabla_{\log}$-module on $R$, then $\overline M$ inherits a monodromy operator, or equivalently a Sen operator $N := \overline \partial_{M,\log}$.
If we call \emph{reduced} an $R$-module $M$ such that $(p)_qM = 0$, then the subcategory of reduced $R$-modules is equivalent to the category of $\mathcal O_K$-modules and a log-connection on $M$ corresponds to a monodromy operator.
It follows that a Sen operator on an $\mathcal O_K$-module $M$ is also equivalent to a $\mathrm D_{R,\log}$-module structure, a log-stratification or a log-Taylor structure.
However, the situation simplifies a lot here because the \emph{reduced} log-Taylor map $\theta : \mathcal O_K \to \mathcal O_K \langle \omega \rangle/\omega^{n+1}$ is \emph{identical} to the canonical map $\iota$.
In particular, a differential operator of order at most $n$ between two reduced $R$-modules is the same thing as an $\mathcal O_K$-linear map
\[
\mathcal O_K \langle \omega \rangle/\omega^{n+1} \otimes_{\mathcal O_K} M \to M'
\]
where we use the canonical structure everywhere.
Also, a log-stratification (resp.\ a log-Taylor structure) on a reduced $R$-module $M$ is the same thing as a compatible family of $\mathcal O_K$-linear isomorphisms (resp.\ maps)
\[
\epsilon_{n} : \mathcal O_K \langle \omega \rangle/\omega^{n+1} \otimes_{\mathcal O_K} M \simeq M \otimes_{\mathcal O_K} \mathcal O_K \langle \omega \rangle/\omega^{n+1} \quad (\mathrm{resp.}\ \theta_{n} : M \to M \otimes_{\mathcal O_K} \mathcal O_K \langle \omega \rangle/\omega^{n+1})
\]
satisfying the usual cocycle conditions.
We use again the canonical structure everywhere.

In order to compose explicitly log-differential operators on $\mathcal O_K$-modules, we shall remark that comultiplication satisfies
\begin{equation} \label{comult}
\Delta(\omega) \equiv 1 \otimes \omega + \omega \otimes 1+ (p)'_\zeta\omega \otimes \omega \mod (p)_q.
\end{equation}

\begin{prop} \label{ucomp2}
In $\mathrm D_{R,\log}$, we have
\[
\forall n \in \mathbb N, \quad \partial_{\log}^{\langle n \rangle} \equiv \prod_{k=0}^{n-1}( \partial_{\log}- k(p)'_\zeta) \equiv \sum_{k=1}^{n} s(n,k) ((p)'_\zeta) ^{n-k} \partial_{\log}^k \mod (p)_q
\]
where $s(n,k)$ denotes the Stirling number of the first kind.
\end{prop}

\begin{proof}
This is proved similarly to proposition \ref{ucompplus2}.
By definition, we have for all $k \in \mathbb N$,
\begin{linenomath}
\begin{align*}
(\widetilde \partial_{\log}^{\langle n \rangle} \widetilde \circ \widetilde \partial_{\log})(\omega^{[k]}) & \equiv \widetilde \partial_{\log}^{\langle n \rangle} ((\mathrm{Id} \otimes \widetilde \partial_{\log})(\Delta(\omega^{[k]})))  \mod (p)_q
\\ & \equiv \widetilde \partial_{\log}^{\langle n \rangle} \left((\mathrm{Id} \otimes \widetilde \partial_{\log})\left( \sum_{i+j=k} (1 + (p)'_\zeta\omega)^j \omega^{[i]} \otimes \omega^{[j]}\right)\right)  \mod (p)_q
\\ & \equiv \widetilde \partial_{\log}^{\langle n \rangle} \left( 1 + (p)'_\zeta\omega) \omega^{[k-1]} \right)  \mod (p)_q
\\ & \equiv \widetilde \partial_{\log}^{\langle n \rangle}( \omega^{[k-1]}) + k(p)'_\zeta \widetilde \partial_{\log}^{\langle n \rangle} (\omega^{[k]})  \mod (p)_q.
\end{align*}
\end{linenomath}
We easily deduce that
\[\widetilde \partial_{\log}^{\langle n+1 \rangle} \equiv \widetilde \partial_{\log}^{\langle n \rangle} \widetilde \circ (\widetilde \partial_{\log} - n (p)'_\zeta)  \mod (p)_q
\]
and get the first congruence by induction.
The second congruence follows from the definition of Stirling numbers.
\end{proof}

\begin{rmks} \phantomsection \label{HTrmks}
\begin{enumerate}
\item 
We will also have
\[
\partial_{\log}^n \equiv \sum_{k=1}^{n} S(n,k) ((p)_\zeta')^{n-k} \partial_{\log}^{\langle k \rangle} \mod (p)_q
\]
using this time Stirling numbers of the second kind.
\item
Another computation shows that composition satisfies
\[
\forall i,j \in \mathbb N, \quad \partial_{\log}^{\langle i \rangle} \circ \partial_{\log}^{\langle j \rangle} \equiv \sum_{k=0}^j k!{i \choose k}{j \choose k} ((p)'_\zeta)^k \partial_{\log}^{\langle i+j-k \rangle} \mod (p)_q.
\]
\item If an $\mathcal O_K$-module $M$ is endowed with a Sen operator $N$, then we have 
\[
\forall m \in \mathbb N, \forall s \in M, \quad \theta_m(s) = \sum_{n=0}^m \prod_{k=0}^{n-1}(N - k(p)'_\zeta)(s) \otimes \omega^{[n]}.
\]
\item
We recover in our situation the formal group of \cite{BhattLurie22b}, proposition 9.5 and \cite{AnschuetzHeuerLeBras22}, proposition 2.3. Indeed, one immediately verifies that the formal groupoid $G := \mathrm{Spf}(R\langle\omega\rangle_{\mathbb\Delta})$ introduced at the beginning of section \ref{stratder} reduces modulo $(p)_q$ to a usual formal group,
\[
\overline{G} := \mathrm{Spf}(\mathcal O_K\langle\omega\rangle).
\]
When $p$ is odd, $\overline{G}$ is isomorphic to the divided power additive group $\mathbb{G}^{\sharp}_a := \mathrm{Spf}(\mathcal O_K\langle X \rangle)$ (with $\Delta(X)= 1 \otimes X + X \otimes 1$), while, when $p=2$, $\overline{G}$ is exactly the divided power multiplicative group $\mathbb{G}^{\sharp}_m$.
For $p=2$ this is immediately deduced from the comultiplication formula \eqref{comult} by simply noting that $(2)'_\zeta =1$. For $p$ odd, it follows from \eqref{comult} that
\[
\log(1+ (p)'_\zeta \Delta(\omega)) = 1 \otimes \log(1+ (p)'_\zeta \omega) + \log(1+ (p)'_\zeta \omega) \otimes 1,
\]
and therefore $X \longmapsto \log(1+ (p)'_\zeta \omega)$ defines an isomorphism $\mathbb{G}^{\sharp}_a \simeq \overline{G}$.
\end{enumerate}
\end{rmks}

There exists the notion of a \emph{log-differential operator}
\[
\widehat{\mathcal O_K \langle \omega \rangle} \widehat \otimes_{\mathcal O_K} M \to M'
\]
on an $\mathcal O_K$-module $M$ which is simply an $\mathcal O_K$-linear map.
They may be composed in the usual way using $\Delta$ (and finite order is preserved).
A \emph{log-hyperstratification} (resp.\ a \emph{log-Taylor map}) on an $\mathcal O_K$-module $M$ is an $\mathcal O_K$-linear isomorphism (resp.\ map)
\[
\epsilon : \widehat{\mathcal O_K \langle \omega \rangle} \widehat\otimes_{\mathcal O_K} M \simeq M \widehat\otimes_{\mathcal O_K} \widehat{\mathcal O_K \langle \omega \rangle} \quad (\mathrm{resp.}\ \theta : M \to M \widehat\otimes_{\mathcal O_K} \widehat{\mathcal O_K \langle \omega \rangle} )
\]
satisfying the usual cocycle conditions.
It then follows from proposition \ref{ucomp2} that if $M$ is a finite free $\mathcal O_K$-module, then a log-Taylor map reads
\begin{equation} \label{thetaN}
\theta(s) = \sum_{n=0}^\infty \prod_{k=0}^{n-1}(N - k(p)'_\zeta)(s) \otimes \omega^{[n]}.
\end{equation}

We recover in our setting the condition in definition 4.1 in \cite{GaoMinWang23}:

\begin{lem} \label{MWcond}
If $p$ is odd, then a Sen operator $N$ on an $\mathcal O_K$-module $M$ is weakly quasi-nilpotent if and only if
\[
\forall s \in M, \quad\prod_{k=0}^{\infty}( N - k(p)'_\zeta)(s) = 0.
\]
When $p=2$ this last condition is equivalent to $N^2-N$ being weakly quasi-nilpotent.
\end{lem}

\begin{proof}
Assume first that $p$ is odd.
It then follows from proposition \ref{ucomp2} that
\[
\partial_{\log}^{\langle n \rangle}(s) = \prod_{k=0}^{n-1}(N - k(p)'_\zeta)(s) = \sum_{k=1}^{n} s(n,k) ((p)_\zeta')^{n-k} N^k(s)
\] 
and
\[
N^n(s) = \sum_{k=1}^{n} S(n,k) ((p)_\zeta')^{n-k} \partial_{\log}^{\langle k \rangle}(s).
\]
Moreover, $(p)'_\zeta \in (p,\zeta-1)$ because $(p)'_q \equiv \frac {p(p-1)}2 \mod q-1$ and we assumed $p$ odd.
It follows that
\[
N^n(s) \to 0\quad \Leftrightarrow \quad \partial_{\log}^{\langle n \rangle} (s) \to 0.
\]
The condition is thus equivalent to $N$ being topologically quasi-nilpotent.
Since $N$ is $\mathcal O_K$-linear, this is also equivalent to being weakly quasi-nilpotent.

In the case $p=2$, we will show by induction on $k \in \mathbb N$ that
\begin{equation} \label{indp2}
\partial_{\log}^{\langle 2k \rangle} \equiv (\partial_{\log}^{\langle 2 \rangle})^k + \sum_{i=0}^{k-1} 2^{k-i}(a_{ki} \partial_{\log} + b_{ki})(\partial_{\log}^{\langle 2 \rangle})^i \mod (p)_q
\end{equation}
with $a_{ki}, b_{ki} \in \mathcal O_K$ and then one can proceed exactly as before using $(N^2 -N)(s) = \partial_{\log}^{\langle 2 \rangle}(s)$ for $s \in M$.
In order to prove equivalence \eqref{indp2}, let us first notice that if $n \in \mathbb N$, then
\begin{align*}
\partial_{\log}^{\langle n+2 \rangle} & \equiv (\partial_{\log}- n-1)(\partial_{\log}- n)\partial_{\log}^{\langle n \rangle}  \mod (p)_q \\
& \equiv \left(\partial_{\log}^{\langle 2 \rangle} -2\left( n \partial_{\log} - {n \choose 2}\right)\right)\partial_{\log}^{\langle n \rangle}  \mod (p)_q.
\end{align*}
In particular, we can write $\partial_{\log}^{\langle 2k+2 \rangle} \equiv \left(\partial_{\log}^{\langle 2 \rangle} +2\left( a_k \partial_{\log} + b_k\right)\right) \partial_{\log}^{\langle 2k \rangle}$ with $a_k, b_k \in \mathcal O_K$.
Since $\partial_{\log}^2 =  \partial_{\log} + \partial_{\log}^{\langle 2 \rangle}$, it will follow that
\begin{align*}
\partial_{\log}^{\langle 2k+2 \rangle} &\equiv \left(\partial_{\log}^{\langle 2 \rangle} +2\left( a_k \partial_{\log} + b_k\right)\right) \left((\partial_{\log}^{\langle 2 \rangle})^k + \sum_{i=0}^{k-1} 2^{k-i}(a_{ki} \partial_{\log} + b_{ki})(\partial_{\log}^{\langle 2 \rangle})^i\right) \\
&\equiv (\partial_{\log}^{\langle 2 \rangle})^{k+1} + 2\left( a_k \partial_{\log} +b_k\right) (\partial_{\log}^{\langle 2 \rangle})^k + \sum_{i=0}^{k-1} 2^{k-i}(a_{ki} \partial_{\log} + b_{ki})(\partial_{\log}^{\langle 2 \rangle})^{i+1}\\
&+ \sum_{i=0}^{k-1}2^{k-i+1}((a_kb_{ki} + b_ka_{ki}) \partial_{\log} +b_kb_{ki} + a_ka_{ki})(\partial_{\log}^{\langle 2 \rangle})^i + 2^{k-i}(2a_ka_{ki})(\partial_{\log}^{\langle 2 \rangle})^{i+1}. \qedhere
\end{align*}

\end{proof}

\begin{prop} \label{Sen_eq}
If $p$ is odd, then it is equivalent to give a log-hyperstratification on an $\mathcal O_K$-module $M$ or a weakly quasi-nilpotent Sen operator.
When $p=2$, a log-hyperstratification is equivalent to a Sen operator $N$ such that $N^2-N$ is weakly quasi-nilpotent.
\end{prop}

\begin{proof}
We know that a Sen operator is equivalent to a log-Taylor structure and we may then consider the map
\[
\varprojlim \theta_{n} : M \to M \otimes_{\mathcal O_K} \varprojlim \mathcal O_K\langle \omega \rangle/ \omega^{[>n]}, \quad s \mapsto \sum_{k=0}^\infty \partial_{\log}^{\langle k \rangle}(s) \otimes \omega^{[k]}.
\]
We have to show that it takes values inside $M \otimes_{\mathcal O_K} \widehat {\mathcal O_K\langle \omega \rangle}$.
It simply means that
\[
\forall s \in M, \quad \partial_{\log}^{\langle k \rangle}(s) \to 0
\]
and this is equivalent to our nilpotence condition thanks to equality \eqref{thetaN} and lemma \ref{MWcond}.
\end{proof}

We consider now the sheaf of rings $\overline {\mathcal O}_{\mathbb \Delta(\mathcal O_K)}$ on the prismatic site of $\mathcal O_K$ that sends a prism $(B,J)$ to $\overline B := B/J$.
An $\overline {\mathcal O}_{\mathbb \Delta(\mathcal O_K)}$-vector bundle is also called a \emph{Hodge-Tate crystal} on $\mathcal O_K$.

We recover the first equivalence in theorem 4.2 of \cite{GaoMinWang23} (see also theorem 2.5 in \cite{AnschuetzHeuerLeBras22}) in our setting:

\begin{prop} \label{Sen}
If $p$ is odd, then there exists an equivalence between the category of Hodge-Tate crystals on $\mathcal O_K$ and the category of finite free $\mathcal O_K$-modules endowed with a weakly nilpotent Sen operator.
When $p=2$, Hodge-Tate crystals are equivalent to finite free $\mathcal O_K$-modules endowed with a Sen operator $N$ such that $N^2-N$ is weakly nilpotent.
\end{prop}

With some extra work, one can show that this is compatible with cohomology (see proposition 2.7 in \cite{AnschuetzHeuerLeBras22}, theorem 4.5 of \cite{GaoMinWang23} or example 9.6 in \cite{BhattLurie22b}).

\begin{proof}
It follows from proposition \ref{poweg} that $\widehat{\mathcal O_K\langle\omega\rangle}_{\mathbb\Delta} = \widehat{\mathcal O_K\langle\omega\rangle}$.
Using proposition \ref{Sen_eq}, our assertion is then a formal consequence of corollary \ref{covabs} (and corollary \ref{prodpr}).
\end{proof}

We shall now describe the relationship with absolute calculus.
A \emph{$\mathbb\Delta$-monodromy operator} on an $\mathcal O_K$-module $M$ is an $\mathcal O_K$-linear map
\[
\nabla_{M} : M \to M \otimes_{\mathcal O_K} \overline \Omega_{R,\mathbb\Delta}.
\]
This is the same thing as a $\mathbb\Delta$-connection on $M$.
This is equivalent to giving an $\mathcal O_K$-linear map $\partial_{M,\mathbb\Delta} : M \to M$ via
\[
\nabla_{M}(s) = \partial_{M,\mathbb\Delta}(s) \otimes \overline{\mathrm{d}_{\mathbb\Delta}q}.
\]

Cohomology is defined as
\[
\mathrm R\Gamma (M) = [M \overset {\partial_{\mathbb\Delta}} \to M].
\]

\begin{prop} \label{HTeq}
If $p$ is odd, then there exists an equivalence between the category of Hodge-Tate crystals $E$ on $\mathcal O_K$ and the category of finite free $\mathcal O_K$-modules $M$ endowed with a weakly nilpotent $\mathbb\Delta$-monodromy operator.
Moreover,
\[
\mathrm R\Gamma(\mathcal O_{K\mathbb{\Delta}}, E) \simeq \mathrm R\Gamma (M).
\]
\end{prop}

\begin{proof}
This is shown exactly as theorem \ref{prisdR} (following the same process).
\end{proof}

\begin{rmks}
\begin{enumerate}
\item Example 3 after proposition \ref{invforSen} shows why we require here that $p$ be odd.
\item As a consequence of the proposition, if $E$ is a Hodge-Tate crystal on $\mathcal O_K$, then
\[
\tau_{\geq 2}\mathrm R\Gamma(\mathcal O_{K\mathbb{\Delta}}, E) =0.
\]
\item Starting from a Hodge-Tate crystal $E$ on $\mathcal O_K$, we obtain \emph{both} a Sen operator and a $\mathbb\Delta$-monodromy operator.
They provide two $\mathcal O_K$-linear maps $N, \partial_{\mathbb\Delta}: M \to M$.
Be careful that $N\neq \partial_{\mathbb\Delta}$ in general because $\omega^{\{2\}_{\mathbb\Delta}} = \omega^{[2]} + q\omega$.
Again, this is a confluence phenomenon.
\item It should be noticed however that $N$ and $\partial_{\mathbb\Delta}$ do agree modulo $\zeta-1$.
\item Proposition \ref{poweg} implies that the categories of reduced log-hyperstratified modules and reduced $\mathbb\Delta$-hyperstratified modules are isomorphic.
\end{enumerate}
\end{rmks}

\begin{prop} \label{invforSen}
If $M$ is the $\mathcal O_K$-module associated to a Hodge-Tate crystal $E$ on $\mathcal O_K$, then for all $s \in M$, we have
\[
\partial_{\log}^{\langle k \rangle}(s) = \sum_{n=k}^{\infty} \left(\frac{k!}{{n!}} s(n,k)(\zeta-1)^{n-k} \right) \zeta^{n-k} \partial^{\langle n \rangle}_{\mathbb\Delta}(s)
\]
and
\[
 \partial^{\langle k \rangle}_{\mathbb\Delta}(s) = \sum_{n=k}^{\infty} \left(\frac{k!}{n!} S(n,k)(\zeta-1)^{n-k} \right) \zeta^{n-k} \partial_{\log}^{\langle n \rangle}(s).
\]
\end{prop}

\begin{proof}
Using proposition \ref{poweg}, this is shown exactly as proposition \ref{invfor}.
\end{proof}

In particular, we can write
\[
N(s) = \sum_{n=1}^{\infty} \frac{(\zeta-1)^{n-1}}{n!} s(n,1) \zeta^{n-1}\partial^{\langle n \rangle}_{\mathbb\Delta}(s) =  \sum_{n=1}^{\infty}  \frac{(\zeta-\zeta^2)^{n-1}}{n}\partial^{\langle n \rangle}_{\mathbb\Delta}(s).
\]

\begin{xmps}
\begin{enumerate}
\item
We consider a free $\mathcal O_K$-module $G$ of rank one and denote the basis by $s$.
A Sen operator on $G$ is then uniquely determined by $N(s) = c s$ for some $c \in \mathcal O_K$.
Then, we have for all $k \in \mathbb N$, $N^k(s) = c^ks$ and we see that $N$ is weakly nilpotent if and only if $v_p(c) > 0$.
\item
We specialize now to the case $c = c_n := n(p)'_\zeta$ for some $n \in \mathbb N$.
It then follows from proposition \ref{ucomp2} that
\[
\partial_{\log}^{\langle k \rangle}(s) = \left\{\begin{array} {ll} \frac {n!}{(n-k)!}((p)'_\zeta)^ks & \mathrm{if}\ k \leq n \\ 0 & \mathrm{otherwise}, \end{array}\right.
\]
or equivalently,
\[
\theta(s) = s \otimes \sum_{k=0}^{n} \frac {n!}{(n-k)!}((p)'_\zeta)^k\omega^{[k]} = s \otimes (1+(p)'_\zeta\omega)^n.
\]
Of course, this is $G_n \simeq (p)^n_qR/(p)_q^{n+1}R$ and it corresponds to the Hodge-Tate crystal $I_{\mathbb\Delta}^n/I_{\mathbb\Delta}^{n+1}$.
\item On $G_n$, we have $N(s) = c_ns$ with $c_n = n(p)'_\zeta$ and $\partial_{\mathbb\Delta}(s) = a_ns$ with $a_n = (n)_{p+1}(p)'_\zeta$.
In the case $p$ is odd, the \emph{lifting-the-exponent lemma}\footnote{If $a \equiv b \mod p$ odd but $p \nmid a, b$, then $v_p(a^n-b^n) = v_p(a-b) + v_p(n)$.} implies $v_p((n)_{p+1}) = v_p(n)$, so that $H^1(G_n, N) \simeq H^1(G_n, \partial_{\mathbb\Delta})$ as expected.
However, when $p=2$, then $H^1(G_2, N) \simeq \mathbb Z/2\mathbb Z$ but $H^1(G_2, \partial_{\mathbb\Delta}) \simeq \mathbb Z/4 \mathbb Z$.
This shows that an extra ingredient (besides modifying the nilpotence condition) is necessary in order to extend proposition \ref{HTeq} to the case $p=2$.
\end{enumerate}
\end{xmps}

We end this section with a brief presentation of the mixed de Rham and Hodge-Tate theory.
A \emph{monodromy operator} on a $K$-vector space $M$ is a $K$-linear map
\[
\nabla_M : M \to M \otimes_{\mathcal O_K} \overline \Omega_{S,\log},
\]
or, equivalently, a $K$-linear map $N : M \to M$, which is again called a \emph{Sen operator}, via
\[
\nabla_M(s) =N(s) \otimes \overline{\mathrm {dlog}(q-\zeta)}.
\]
Then, all the results from this section and the previous one have an analog in this situation.
In particular, the category of vector bundles on $\overline {\mathcal O}_{\mathbb \Delta}[1/p] = \overline {\mathbb B}^+_{\mathrm{dR}}$ is equivalent to the category of finite dimensional $K$-vector spaces endowed with a weakly nilpotent Sen operator (and there exists a $\mathbb\Delta$-analog).

We can put all our equivalences together in the following diagram, in which double arrows represent equivalences, $\mathrm{Vect}$ stands for ``prismatic vector bundles'', $\nabla_{\log}$ (resp. $\nabla_{\mathbb\Delta}$) stands for ``finite free modules endowed with $\log$- (resp. $\mathbb\Delta$-) connection'', and  wn (resp. tn, resp. tnm) is short for ``weakly nilpotent'' (resp. ``topologically nilpotent'', resp. ``with topologically nilpotent monodromy'').

\begin{equation} \label{bigdiag}
\xymatrix{
&&&&\nabla_{\log}^{\mathrm{wn}} (\mathcal O_K)\ar[ddd]
\\&&&&&
\\&\mathrm{Vect}(\mathcal O_K, \mathcal O_{\mathbb \Delta}) \ar[rr] \ar[ddd] |!{[ldd];[rdd]}\hole &&\mathrm{Vect}(\mathcal O_K, \overline {\mathcal O}_{\mathbb \Delta}) \ar[ddd] \ar@{<->}[uur]&
\\&&\nabla_{\log}^{\mathrm{tnm}} (S)\ar[rr] |!{[d];[ur]}\hole |!{[ur];[ddr]}\hole &&\nabla_{\log}^{\mathrm{tn}} (K)
\\ \nabla_{\mathbb\Delta}^{\mathrm{wn}} (R)\ar[rr] \ar[ddd] \ar@{<->}[uur] && \nabla_{\mathbb\Delta}^{\mathrm{wn}} (\mathcal O_K) \ar[ddd] \ar@{<->}[uur]&&
\\& \mathrm{Vect}(\mathcal O_K, \mathbb B^+_{\mathrm{dR}}) \ar[rr] |!{[ur];[ddr]}\hole \ar@{<->}[uur] |!{[ur];[ul]}\hole&&\mathrm{Vect}(\mathcal O_K, \overline {\mathcal O}_{\mathbb \Delta}[1/p]) \ar@{<->}[uur]
\\&&&&&
\\ \nabla_{\mathbb\Delta}^{\mathrm{tnm}} (S)\ar[rr] \ar@{<->}[uur] && \nabla_{\mathbb\Delta}^{\mathrm{tn}} (K) \ar@{<->}[uur]&&
}
\end{equation}

It is important to remark that usual logarithmic theory on $R$ itself does not capture the prismatic essence and it is therefore necessary to rely on our twisted approach in this case.

\section{Prismatic $F$-crystals} \label{frob}

We shall now investigate how frobenius interacts with our theory.

For the moment, we only need to assume that $W$ is a complete $\delta$-ring but in the end, when we consider the prismatic site, we shall focus on the case where $W = W(k)$ is the ring of Witt vectors of a perfect field $k$ as in section \ref{prissec}.
We endow $R := W[[q-1]]$ with the unique $\delta$-structure such that $q$ has rank one (meaning $\delta(q) = 0$ or equivalently $\phi(q) = q^p$).

As a consequence of proposition \ref{foncr}, if $M$ is a $\nabla_{\mathbb\Delta}$-module and $\phi^*$ denotes the pullback along frobenius, then $\phi^*M$ inherits a $\mathbb\Delta$-connection given by
\begin{equation} \label{phstarm}
\forall s \in M, \quad \partial_{\phi^*M,\mathbb\Delta}(1 \otimes s) = (p)_qq^{p-1} \otimes \partial_{M,\mathbb\Delta}(s).
\end{equation}

On the other hand, if $M$ is a $\nabla_{\mathbb\Delta}$-module, then
\[
(p)_q^{-r}M \simeq (p)_q^{-r}R \otimes_R M \simeq \mathrm{Hom}_R((p)_q^rR,M)
\]
and
\[
M[1/(p)_q] = \varinjlim_r\ (p)_q^{-r}M
\]
inherit a $\mathbb\Delta$-connection.
Recall also that there exists the general notion of an isogeny and that a \emph{$(p)_q$-isogeny} between two \emph{finite projective} $R$-modules $M'$ and $M$ is simply an $R[1/(p)_q]$-linear isomorphism
\[
M'[1/(p)_q] \simeq M[1/(p)_q].
\]

\begin{dfn}
Let $M$ be a finite projective $\nabla_{\mathbb\Delta}$-module on $R$.
Then, a \emph{relative frobenius} on $M$ is a horizontal $(p)_q$-isogeny
\begin{equation} \label{efem}
F_M : (\phi^*M)[1/(p)_q] \simeq M[1/(p)_q].
\end{equation}
We shall then call $M$ an \emph{$F$-$\nabla_{\mathbb\Delta}$-module} on $R$.
\end{dfn}

\begin{rmks}
\begin{enumerate}
\item
It is equivalent to give a $(p)_q$-isogeny $F_M$ between $\phi^*M$ and $M$ or an injective $\phi$-linear map
\[
\phi_M : M \to (p)_q^{-r} M
\]
for some $r \in \mathbb N$ such that, for some $k \in \mathbb N$, $(p)_q^kM$ is contained in the $R$-submodule generated by the image of $\phi_M$.
We then have the commutative diagram
\[
\xymatrix{M \ar[r]^{\phi_M} \ar@{^{(}->}[d] & (p)_q^{-r}M \ar@{^{(}->}[d] \\ (\phi^*M)[1/(p)_q] \ar[r]_-{F_M}^-\simeq & M[1/(p)_q].}
\]
\item
The $(p)_q$-isogeny is horizontal (and defines therefore a relative frobenius) if and only if
\begin{equation} \label{frobcon}
\partial_{(p)_q^{-r}M,\mathbb\Delta} \circ \phi_M = (p)_qq^{p-1} \phi_M \circ \partial_{M,\mathbb\Delta}.
\end{equation}
We shall then call $\phi_M$ the \emph{absolute frobenius} on $M$.
\item
The absolute frobenius $\phi$ of $R$ can be extended to a $\phi$-linear map on $\Omega_{q^p}$ as follows:
\[
\phi: \Omega_{q^p} \to \Omega_{q^p}, \quad \mathrm d_{q^p} q \mapsto \mathrm d_{q^p}(q^p) = (p)_qq^{p-1}\mathrm d_{q^p}q.
\]
Then, condition \eqref{frobcon} may be expressed as a commutative diagram
\[
\xymatrix{M \ar[rr]^{\phi_M} \ar[d]^{\nabla_M} && (p)_q^{-r}M \ar[d]^{\nabla_{(p)_q^{-r}M}} \\ M \otimes_R \Omega_{q^p} \ar[rr]^-{\phi_M \otimes \phi} && (p)_q^{-r}M \otimes_R \Omega_{q^p}.
}
\]
\item If $M_1$ and $M_2$ are two $F$-$\nabla_{\mathbb\Delta}$-modules on $R$, then both $M_1 \otimes_R M_2$ and $\mathrm{Hom}_R(M_1, M_2)$ inherit a frobenius structure.
As a consequence, if $M$ is a $F$-$\nabla_{\mathbb\Delta}$-module, then, for all $n \in \mathbb Z$, $M^{\otimes n}$ also has a frobenius structure.
\end{enumerate}
\end{rmks}

We may remove the subscripts and simply write $F$ and $\phi$ when the module $M$ is clear from the context.
In order to give some examples, we start with the following one:

\begin{lem} \label{BKfrob}
The map
\[
\phi_{R\{1\}}(e_R) = (p)_q^{-1}e_R
\]
defines a frobenius structure on $R\{1\}$.
\end{lem}

\begin{proof}
We drop the subscript $R\{1\}$ and simply write $\phi$ in this proof.
Let us first notice that (with this definition), we shall have $\phi((q-1)e_R) = (q-1)e_R$.
Then, on the one hand, we have
\[
(\partial_{\mathbb\Delta} \circ \phi)((q-1)e_R) = \partial_{\mathbb\Delta}((q-1)e_R) = \frac p{q}e_R,
\]
and, on the other hand,
\[
(p)_qq^{p-1} (\phi \circ \partial_{\mathbb\Delta})((q-1)e_R) = (p)_qq^{p-1} \phi \left(\frac pq e_R\right)
= (p)_qq^{p-1} \frac p{(p)_qq^p} e_R = \frac p{q}e_R.
\]
Since $(\gamma \circ \phi)(q-1) = q^{p^2+p}- 1$ is regular in $R$, a similar result holds with $e_R$ in the place of $(q-1)e_R$.
\end{proof}

\begin{xmps}
\begin{enumerate}
\item We may consider the trivial $\mathbb\Delta$-connection on $R$ with the identity as relative frobenius.
We have
\begin{equation} \label{frobtriv}
\partial_{\mathbb\Delta} \circ \phi = (p)_qq^{p-1} \phi \circ \partial_{\mathbb\Delta}.
\end{equation}
This particular formula looks more familiar if we use $q\partial_{\mathbb\Delta}$ instead of $\partial_{\mathbb\Delta}$ (but see remark \ref{rmksDqp}.1):
\[
q\partial_{\mathbb\Delta} \circ \phi = (p)_q \phi \circ q\partial_{\mathbb\Delta}.
\]
\item 
More generally, $(p)_q^nR$ is canonically endowed with a frobenius structure for all $n \in \mathbb Z$.
\item It follows from lemma \ref{BKfrob} that the Breuil-Kisin $\mathbb\Delta$-connection has a frobenius structure given by
\[
\phi : R\{1\} \to (p)_q^{-1}R\{1\}, \quad e_R \mapsto (p)_q^{-1} e_R.
\]
The relative frobenius actually comes from an isomorphism of $\nabla_{\mathbb\Delta}$-modules $\phi^*R\{1\} \simeq (p)_q^{-1}R\{1\}$.
\item
This extends to all $R\{n\}$ for $n \in \mathbb Z$ and the absolute frobenius of $R\{n\}$ is given by $\phi(e_R^{\otimes n}) = (p)_q^{-n}e_R^{\otimes n}$.
\item
If $M$ is a $F$-$\nabla_{\mathbb\Delta}$-module on $R$, then both the $(p)_q$-adic twist $(p)_q^nM$ and the Breuil-Kisin twist $M\{n\}$ also have a frobenius structure for all $n \in \mathbb Z$.
\end{enumerate}
\end{xmps}

\begin{lem} \label{Fnilp} \label{FrobWN}
A $F$-$\nabla_{\mathbb\Delta}$-module on $R$ is automatically weakly nilpotent.
\end{lem}

\begin{proof}
There exists $r \in \mathbb N$, such that $M \subset (p)_q^{-r}\phi^*M$.
Formula \eqref{phstarm} implies that $\phi^*M$ is weakly nilpotent.
We also know that $(p)_qR$ is weakly nilpotent.
Since the property is stable under tensor product and internal Hom, it follows that $(p)_q^{-r}\phi^*M$ is weakly nilpotent and so is $M$.
\end{proof}

Up to the end of this section, we place ourselves in the situation of section \ref{prissec}: we assume that $W = W(k)$ is the ring of Witt vectors of a perfect field $k$ of characteristic $p>0$.

Recall from \cite{BhattScholze21}, definition 3.2, that a \emph{prismatic $F$-crystal} is a prismatic vector bundle $E$ endowed with an isomorphism
\[
F_E : (\phi^*E)[1/I_{\mathbb \Delta}] \simeq E[1/I_{\mathbb \Delta}].
\]
Then, we have the following corollary of theorem \ref{prisdR}:

\begin{cor} \label{Frobeq}
Assume $W$ is the Witt vector ring of a perfect field $k$ of characteristic $p > 0$ and $\zeta$ is a primitive $p$th root of unity.
Then, there exists an equivalence between the category of prismatic $F$-crystals on $W[\zeta]$ and the category of $F$-$\nabla_{\mathbb\Delta}$-modules on $W[[q-1]]$. \qed
\end{cor}

We shall now recall the following beautiful result of Bhatt and Scholze:

\begin{thm}[\cite{BhattScholze21}, theorem 5.6]
If $K$ is a complete discrete valuation field of mixed characteristic with perfect residue field $k$, then there exists an equivalence between the category of prismatic $F$-crystals on $\mathcal O_K$ and the category of lattices in crystalline representations of $G_K$.
\end{thm}

Combining this with corollary \ref{Frobeq}, we obtain:

\begin{cor} \label{corBS}
Let $k$ be a perfect field of characteristic $p >0$, $W$ its ring of Witt vectors, $\zeta$ a primite $p$th root of unity and $K$ the fraction field of $W[\zeta]$.
Then, there exists an equivalence between the category of $F$-$\nabla_{\mathbb\Delta}$-modules on $W[[q-1]]$ and the category of lattices in crystalline representations of $G_K$. \qed
\end{cor}

Usually, Wach modules are only defined in the unramified case but this can easily be extended to our situation (recall that we write $\Gamma := \mathrm{Gal}(K_{\infty}/K)$):

\begin{dfn} \label{Wach}
A \emph{Wach module} over $W[\zeta]$ is finite free $R$-modules $M$ endowed with a continuous semilinear action of $\Gamma$ that becomes trivial modulo $q-1$ and a $(p)_q$-isogeny
\[
F_M : (\phi^*M)[1/(p)_q] \simeq M[1/(p)_q]
\]
which is compatible with the actions of $\Gamma$.
\end{dfn}

\begin{prop} \label{Wath}
The category of Wach modules over $W[\zeta]$ is isomorphic to the category of $F$-$\nabla_{\mathbb\Delta}$-modules on $W[[q-1]]$. 
\end{prop}

\begin{proof}
Immediate consequence of proposition \ref{nabGam}.
\end{proof} 

We recover the following analog over $W[\zeta]$ of \cite{Berger04}, proposition III.4.2 (see also \cite{Colmez99}):

\begin{cor}[Berger] \label{Berger}
There exists an equivalence between the category of Wach modules over $W[\zeta]$ and the category of lattices in crystalline representations of $G_K$. \qed
\end{cor}

\addcontentsline{toc}{section}{References}

\printbibliography

\Addresses

\end{document}